\pdfoutput=1
\documentclass[11pt, a4paper, twoside,leqno]{amsart}
\usepackage[centering, totalwidth = 400pt, totalheight = 600pt]{geometry}
\usepackage{amssymb, amsmath, amsthm, bbold}
\usepackage{microtype, stmaryrd, url, lmodern, eucal,mathtools}
\usepackage[shortlabels]{enumitem}
\usepackage[latin1]{inputenc}
\usepackage{color}
\definecolor{darkgreen}{rgb}{0,0.45,0}
\usepackage[colorlinks,citecolor=darkgreen,linkcolor=darkgreen]{hyperref}
\usepackage[british]{babel}
\usepackage{cmll}
\usepackage{stmaryrd} 
  \setlist[enumerate,1]
        {%
            leftmargin = 2.5\parindent, 
            labelsep = 0.5\parindent 
        }
  \setlist[itemize]
        {%
            leftmargin = 2.5\parindent, 
            labelsep = 0.75\parindent 
        }

\makeatletter
\def\@cite#1#2{[{#1\if@tempswa ,~#2\fi}]}
\makeatother

\DeclareMathAlphabet{\mathbf}{OT1}{cmr}{b}{n}
\DeclareFontSeriesDefault[rm]{bf}{b}

\usepackage[arrow, matrix, tips, curve, graph, rotate]{xy}
\SelectTips{cm}{10}

\makeatletter
\def\matrixobject@{%
  \edef \next@{={\DirectionfromtheDirection@ }}%
  \expandafter \toks@ \next@ \plainxy@
  \let\xy@@ix@=\xyq@@toksix@
  \xyFN@ \OBJECT@}
\let\xy@entry@@norm=\entry@@norm
\def\entry@@norm@patched{%
  \let\object@=\matrixobject@
  \xy@entry@@norm }
\AtBeginDocument{\let\entry@@norm\entry@@norm@patched}
\makeatother

\newcommand{\twocong}[2][0.5]{\ar@{}[#2] \save ?(#1)*{\cong}\restore}
\newcommand{\twoeq}[2][0.5]{\ar@{}[#2] \save ?(#1)*{=}\restore}
\newcommand{\rtwocell}[3][0.5]{\ar@{}[#2] \ar@{=>}?(#1)+/l 0.2cm/;?(#1)+/r 0.2cm/^{#3}}
\newcommand{\ltwocell}[3][0.5]{\ar@{}[#2] \ar@{=>}?(#1)+/r 0.2cm/;?(#1)+/l 0.2cm/^{#3}}
\newcommand{\ltwocello}[3][0.5]{\ar@{}[#2] \ar@{=>}?(#1)+/r 0.2cm/;?(#1)+/l 0.2cm/_{#3}}
\newcommand{\dtwocell}[3][0.5]{\ar@{}[#2] \ar@{=>}?(#1)+/u  0.2cm/;?(#1)+/d 0.2cm/^{#3}}
\newcommand{\dltwocell}[3][0.5]{\ar@{}[#2] \ar@{=>}?(#1)+/ur  0.2cm/;?(#1)+/dl 0.2cm/^{#3}}
\newcommand{\drtwocell}[3][0.5]{\ar@{}[#2] \ar@{=>}?(#1)+/ul  0.2cm/;?(#1)+/dr 0.2cm/^{#3}}
\newcommand{\dthreecell}[3][0.5]{\ar@{}[#2] \ar@3{->}?(#1)+/u  0.2cm/;?(#1)+/d 0.2cm/^{#3}}
\newcommand{\utwocell}[3][0.5]{\ar@{}[#2] \ar@{=>}?(#1)+/d 0.2cm/;?(#1)+/u 0.2cm/_{#3}}
\newcommand{\dtwocelltarg}[3][0.5]{\ar@{}#2 \ar@{=>}?(#1)+/u  0.2cm/;?(#1)+/d 0.2cm/^{#3}}
\newcommand{\utwocelltarg}[3][0.5]{\ar@{}#2 \ar@{=>}?(#1)+/d  0.2cm/;?(#1)+/u 0.2cm/_{#3}}

\newdir{(}{{}*!<0em,-.14em>-\cir<.14em>{l^r}}
\newdir{ (}{{}*!/-5pt/\dir{(}}
\newdir{ >}{{}*!/-5pt/\dir{>}}
\newcommand{\sh}[2]{**{!/#1 -#2/}}

\DeclareMathOperator{\Sym}{Sym}

\newcommand{\cat}[1]{\mathsf{#1}}

\newcommand{\thg}{{\mathord{\text{--}}}}

\newcommand{\dbr}[1]{\llbracket{#1}\rrbracket}
\newcommand{\res}[2]{\left.{#1}\right|_{#2}}

\newcommand{\spn}[1]{{\langle{#1}\rangle}}

\newcommand{\defeq}{\mathrel{\mathop:}=}
\newcommand{\cd}[2][]{\vcenter{\hbox{\xymatrix#1{#2}}}}

\renewcommand{\phi}{\varphi}

\newcommand{\C}{{\mathcal C}}
\newcommand{\D}{{\mathcal D}}
\newcommand{\E}{{\mathcal E}}

\let\sec=\S
\renewcommand{\S}{{\mathcal S}}

\newcommand{\V}{{\mathcal V}}

\newcommand{\xtor}[1]{\cdl[@1]{{} \ar[r]|-{\object@{|}}^{#1} & {}}}

\makeatletter

\def\hookleftarrowfill@{\arrowfill@\leftarrow\relbar{\relbar\joinrel\rhook}}
\def\twoheadleftarrowfill@{\arrowfill@\twoheadleftarrow\relbar\relbar}
\def\leftbararrowfill@{\arrowdoublefill@{\leftarrow\mkern-5mu}\relbar\mapstochar\relbar\relbar}
\def\Leftbararrowfill@{\arrowdoublefill@{\Leftarrow\mkern-2mu}\Relbar\Mapstochar\Relbar\Relbar}
\def\leftringarrowfill@{\arrowdoublefill@{\leftarrow\mkern-3mu}\relbar{\mkern-3mu\circ\mkern-2mu}\relbar\relbar}
\def\lefttriarrowfill@{\arrowfill@{\mathrel\triangleleft\mkern0.5mu\joinrel\relbar}\relbar\relbar}
\def\Lefttriarrowfill@{\arrowfill@{\mathrel\triangleleft\mkern1mu\joinrel\Relbar}\Relbar\Relbar}

\def\hookrightarrowfill@{\arrowfill@{\lhook\joinrel\relbar}\relbar\rightarrow}
\def\twoheadrightarrowfill@{\arrowfill@\relbar\relbar\twoheadrightarrow}
\def\rightbararrowfill@{\arrowdoublefill@{\relbar\mkern-0.5mu}\relbar\mapstochar\relbar\rightarrow}
\def\Rightbararrowfill@{\arrowdoublefill@{\Relbar\mkern-2mu}\Relbar\Mapstochar\Relbar\Rightarrow}
\def\rightringarrowfill@{\arrowdoublefill@\relbar\relbar{\mkern-2mu\circ\mkern-3mu}\relbar{\mkern-3mu\rightarrow}}
\def\righttriarrowfill@{\arrowfill@\relbar\relbar{\relbar\joinrel\mkern0.5mu\mathrel\triangleright}}
\def\Righttriarrowfill@{\arrowfill@\Relbar\Relbar{\Relbar\joinrel\mkern1mu\mathrel\triangleright}}

\def\leftrightarrowfill@{\arrowfill@\leftarrow\relbar\rightarrow}
\def\mapstofill@{\arrowfill@{\mapstochar\relbar}\relbar\rightarrow}

\renewcommand*\xleftarrow[2][]{\ext@arrow 20{20}0\leftarrowfill@{#1}{#2}}
\providecommand*\xLeftarrow[2][]{\ext@arrow 60{22}0{\Leftarrowfill@}{#1}{#2}}
\providecommand*\xhookleftarrow[2][]{\ext@arrow 10{20}0\hookleftarrowfill@{#1}{#2}}
\providecommand*\xtwoheadleftarrow[2][]{\ext@arrow 60{20}0\twoheadleftarrowfill@{#1}{#2}}
\providecommand*\xleftbararrow[2][]{\ext@arrow 10{22}0\leftbararrowfill@{#1}{#2}}
\providecommand*\xLeftbararrow[2][]{\ext@arrow 50{24}0\Leftbararrowfill@{#1}{#2}}
\providecommand*\xleftringarrow[2][]{\ext@arrow 10{26}0\leftringarrowfill@{#1}{#2}}
\providecommand*\xlefttriarrow[2][]{\ext@arrow 80{24}0\lefttriarrowfill@{#1}{#2}}
\providecommand*\xLefttriarrow[2][]{\ext@arrow 80{24}0\Lefttriarrowfill@{#1}{#2}}

\renewcommand*\xrightarrow[2][]{\ext@arrow 01{20}0\rightarrowfill@{#1}{#2}}
\providecommand*\xRightarrow[2][]{\ext@arrow 04{22}0{\Rightarrowfill@}{#1}{#2}}
\providecommand*\xhookrightarrow[2][]{\ext@arrow 00{20}0\hookrightarrowfill@{#1}{#2}}
\providecommand*\xtwoheadrightarrow[2][]{\ext@arrow 03{20}0\twoheadrightarrowfill@{#1}{#2}}
\providecommand*\xrightbararrow[2][]{\ext@arrow 01{22}0\rightbararrowfill@{#1}{#2}}
\providecommand*\xRightbararrow[2][]{\ext@arrow 04{24}0\Rightbararrowfill@{#1}{#2}}
\providecommand*\xrightringarrow[2][]{\ext@arrow 01{26}0\rightringarrowfill@{#1}{#2}}
\providecommand*\xrighttriarrow[2][]{\ext@arrow 07{24}0\righttriarrowfill@{#1}{#2}}
\providecommand*\xRighttriarrow[2][]{\ext@arrow 07{24}0\Righttriarrowfill@{#1}{#2}}

\providecommand*\xmapsto[2][]{\ext@arrow 01{20}0\mapstofill@{#1}{#2}}
\providecommand*\xleftrightarrow[2][]{\ext@arrow 10{22}0\leftrightarrowfill@{#1}{#2}}
\providecommand*\xLeftrightarrow[2][]{\ext@arrow 10{27}0{\Leftrightarrowfill@}{#1}{#2}}

\makeatother

\numberwithin{equation}{section}

\theoremstyle{plain}
\newtheorem{Thm}{Theorem}
\newtheorem*{Thm*}{Theorem}
\newtheorem{Prop}[Thm]{Proposition}

\newtheorem{Lemma}[Thm]{Lemma}

\theoremstyle{definition}
\newtheorem{Defn}[Thm]{Definition}

\newtheorem{Ex}[Thm]{Example}

\newtheorem{Rk}[Thm]{Remark}

\newcommand{\SET}{\cat{SET}}

\newcommand{\kmod}{k\text-\cat{Mod}}
\newcommand{\kvec}{k\text-\cat{Vect}}
\renewcommand{\vec}[1]{\boldsymbol{#1}}
\newcommand{\act}[1][X]{\cat{Act}[#1]}
\newcommand{\ca}[1][X]{\cat{CAct}[#1]}
\renewcommand{\mod}[1][SX]{\cat{Mod}[#1]}

\newcommand{\REL}{\cat{REL}}
\newcommand{\SP}{k\text-\cat{Sp}}
\newcommand{\SVEC}{k\text-\cat{SVect}}

\newcommand{\ocf}{{\oc^\partial}}

\allowdisplaybreaks 

\title{Free differential modalities}
\author{Richard Garner} 
\address{School of Mathematical and Physical Sciences, Macquarie University, NSW 2109, Australia} 
\email{richard.garner@mq.edu.au}

\author{Jean-Simon Pacaud Lemay} 
\address{School of Mathematical and Physical Sciences, Macquarie University, NSW 2109, Australia} 
\email{js.lemay@mq.edu.au}

\subjclass[2000]{Primary: }
\date{\today}

\thanks{The support of Australian Research Council grants DP190102432
  and DE230100303 is gratefully acknowledged. The second named author is supported by the AFOSR under award number FA9550-24-1-0008.}

\begin{document}
\maketitle
\begin{abstract}
  Categorical models of the exponential modality of linear logic will
  often, but not always, support an operation of
  \emph{differentiation}. When they do, we speak of a \emph{monoidal
    differential modality}; when they do not, we have merely a
  \emph{monoidal coalgebra modality}. More generally, there are
  notions of \emph{differential modality} and \emph{coalgebra
    modality} which stand in the same relation, but which move outside
  the realm of linear logic; they model important structures such as
  smooth differentiation on Euclidean space.

  In this paper, we show that, in a suitably well-behaved $k$-linear
  symmetric monoidal category, each (monoidal) coalgebra modality can
  be freely completed to a (monoidal) differential modality. In
  particular, we prove the existence of an \emph{initial} monoidal
  differential modality, which, even in simple examples such as the
  category of sets and relations, yields new models of differential
  linear logic.

  Key to our proofs is the concept of an
  \emph{algebraically-free commutative monoid} in a symmetric monoidal
  category. A commutative monoid $M$ is said to be algebraically-free
  on an object $X$ if actions by the monoid $M$ correspond to
  self-commuting actions by the mere object $X$. Along the way, we
  study the theory of algebraically-free commutative monoids and
  self-commuting actions, and their interplay with differential
  modality structure.
\end{abstract}
\setcounter{tocdepth}{1} 
\tableofcontents
\section{Introduction}
\label{sec:introduction}

\emph{Linear logic}~\cite{Girard1995Linear} refines intuitionistic
logic by its more nuanced view on context manipulation. Careful
attention is paid to the duplication and discard of variables, such
that the basic intuitionistic connectives of ``and'' and ``or'' are
split into two de Morgan dual pairs, $(\with, \oplus)$ and
$(\otimes, \parr)$---the \emph{additive} and \emph{multiplicative}
connectives respectively. Via the Curry--Howard correspondence, linear
logic also gives us a more nuanced view on variable use in
computation, and indeed, many semantic models of linear logic---i.e.,
symmetric monoidal categories endowed with suitable further
structure---draw on constructions from areas such as stable domain
theory, game semantics, and sequential algorithms;
see~\cite{Hyland2003Glueing} for a comprehensive overview.

However, what really gives linear logic its distinctive character is
the structure linking the linear and non-linear worlds, the 
\emph{exponential modality} $\oc$, which is so named because it relates the
additives and multiplicatives via the equation
$\oc(A\oplus B) \cong \oc A \otimes \oc B$. The sense in which the
exponential relates linear to non-linear is made precise by the Girard
translation of classical logic into linear logic; when reified in a
computational model, this allows us to construct programming language
semantics in a cartesian closed category wherein maps carry
intensional information about resource usage.

In this paper, it is another aspect of the exponential modality which
will concern us. In semantic models, the exponential is not
infrequently implemented as some kind of power-series construction.
This perhaps reminds us of the place in mathematics that one first
meets the idea of relating linear and non-linear maps, namely, in the
differential calculus---and indeed, it turns out that many exponential
modalities $\oc$ support an abstract notion of \emph{differentiation}.

This was first observed by Ehrhard~\cite{Ehrhard2002On-Kothe} for a
particular semantic model coming from linear algebra; and the insight
was then incorporated into the proof theory of linear logic via
Ehrhard and Regnier's \emph{differential
  $\lambda$-calculus}~\cite{Ehrhard2003The-differential}. This is
closely related to the earlier \emph{resource $\lambda$-calculus}
of~\cite{Boudol1993Lambda-calculus} and indeed, from a computational
view, it is often helpful to think of the differentiation associated
to an exponential modality as providing direct syntactic access to
information about resource usage. Of course, with the recent explosion
of interest in programming via gradient-based optimisation methods, it
may be equally helpful to think of the differentiation as actual
differentiation.

On the side of the categorical semantics, the appropriate
formalisation of differential structure on an exponential was given
in~\cite{Blute2006Differential}---though we prefer to direct the
reader to the subsequent~\cite{Blute2019Differential}, which cleans up
the theory and ties up all known loose ends. It is important here to
note that the exponential modality of a given model may not necessarily support
differentiation. One basic obstruction is that, to even express the
axioms that the differential structure must satisfy, in particular the famous Leibniz rule, 
the symmetric monoidal category in which one is working must bear additive, or
better, $k$-linear structure on the hom-sets\footnote{Though
  see~\cite{Ehrhard2023Coherent} for a generalisation which partially
  relaxes this requirement.}. Yet even when such structure exists, it
is still a non-trivial requirement that a given exponential should
support differentiation: see, for
example,~\cite[Theorem~7]{Blute2019Differential}, or our
Lemma~\ref{lem:28} below.

The objective of this paper is to show that, in any reasonably
well-behaved category---in particular, in any cocomplete, symmetric
monoidal closed $k$-linear category $\C$---it is always possible to
\emph{freely adjoin} differential structure to any exponential $\oc$
supported by $\C$ and so obtain a new exponential $\oc^\partial$. This
immediately yields many new models of differential linear logic---and
ones which are new even \emph{qua} model of multiplicative exponential
linear logic. Furthermore, we will show that there is always an
\emph{initial} exponential $P$ on any such category $\C$---so that,
consequently, we obtain an \emph{initial}
exponential-with-differentiation $P^\partial$. Even in such a basic
setting as the relational model of linear logic, this $P^\partial$
appears to be new.

To say more about our main results, we need to say a little more about
the categorical models we are working with. To start with, a model of
multiplicative intuitionistic linear logic is simply a symmetric
monoidal closed category $(\C, \otimes, I, \multimap)$; in fact, for
our purposes, the internal hom $\multimap$ will play no role and we do
not assume it for our main results.

On top of this basic structure, a model of the exponential modality
is, first of all, a \emph{monoidal comonad} $\oc$ on $\C$; this is an
endofunctor $\oc \colon \C \rightarrow \C$ endowed with maps:
\begin{align}
    \varepsilon_X \colon \oc X &\rightarrow X \qquad &
    \delta_X \colon \oc X &\rightarrow \oc \oc X \qquad \label{eq:71} \\
    m_\otimes \colon \oc X \otimes \oc Y &\rightarrow \oc(X \otimes Y)
    \qquad & m_I \colon I &\rightarrow \oc I\rlap{ ,} \label{eq:73}
\end{align}
which, when subject to suitable axioms, are the appropriate
proof-relevant formulation of the familiar S4 rules for a modal
operator. Moreover, this $\oc$ should be a \emph{monoidal coalgebra
  modality}\footnote{Also called a \emph{linear exponential comonad}.}
on $\C$: meaning that it comes equipped with further maps
\begin{equation}
  \label{eq:72}
  \Delta_X \colon \oc X \rightarrow \oc X \otimes \oc X \qquad \text{and} \qquad 
  \mathsf{e}_X \colon \oc X \rightarrow I
\end{equation}
equipping each $\oc X$ with the structure of a cocommutative comonoid,
suitably compatible with the remaining structure. These maps get
to the heart of what an exponential is; they express that
$\oc X$ is a resource derived from $X$ that can be freely duplicated
and discarded. Thus, in the Girard translation, a classical
proof of $X \Rightarrow Y$---or equally, a program of type
$X \Rightarrow Y$---is encoded as a linear term of type
$\oc X \multimap Y$.

Finally, to express that $\oc$ is a model of differentiation we must
first, as advertised above, assume that $\C$ has $k$-linear structure
on each hom-set, which is preserved by composition and $\otimes$ in
each variable. Here, $k$ can be a field, but also a ring, or even a
\emph{rig}: that is, a ring without negatives. In particular, when
$k = \mathbb{N}$, the $k$-linear structure on the homs is the same as
commutative monoid structure.

At this point, we may describe the additional structure on $\oc$
required for it to support differentiation in two ways. Most
frequently in the linear logic literature, this is done via a
\emph{codereliction}: a family of maps
\begin{equation}
  \label{eq:69}
  \eta_X \colon X \rightarrow \oc X
\end{equation}
satisfying suitable axioms.
The intuition is that precomposing a non-linear map of type
$\oc X \multimap Y$ with $\eta_X$ yields a map $X \multimap Y$ which
is its best linear approximation at $0$ (where the meaning of ``$0$''
comes from the $k$-linear structure.)

However, this same structure can alternatively be encoded via a
\emph{deriving transformation}: a family of maps
\begin{equation}
  \label{eq:68}
  \mathsf{d}_X \colon \oc X \otimes X \rightarrow \oc X
\end{equation}
satisfying suitable axioms. This time, precomposing a non-linear map
$f \colon {\oc X \multimap Y}$ with $\eta_X$ yields a map
$\mathsf{D}f \colon \oc X \otimes X \multimap Y$, which we think of as
a map $X \times X \Rightarrow Y$ linear in its second argument, whose
interpretation is as giving the directional derivative of $f$ at the
point specified by its first argument, in the direction specified by
the second argument. These two formulations are entirely equivalent
(\cite[Theorems~2 \& 4]{Blute2019Differential}), and in either case,
we may call $\oc$ a \emph{monoidal differential modality}.

With these definitions under our belt, we can state what will be the
second main theorem of our paper (we will return to the first later):

\begin{Thm*}\looseness=-1
  If $\C$ is $k$-linear symmetric monoidal with finite biproducts and
  algebraically-free commutative monoids, then every monoidal
  coalgebra modality on $\C$ can be completed freely to a monoidal
  differential modality; i.e., the forgetful functor from monoidal
  differential modalities on $\C$ to monoidal coalgebra modalities has
  a left adjoint.
\end{Thm*}

We should explain what we mean here by \emph{algebraically-free
  commutative monoids}. In~\cite{Kelly1980A-unified}, Kelly draws the
distinction between \emph{free} and \emph{algebraically-free } monoids
in a monoidal category $\C$. The free monoid on $X \in \C$ is, of
course, a monoid $TX$ endowed with a map
$\eta \colon X \rightarrow TX$ which is universal among maps from $X$
to a monoid. Such a monoid is said to be \emph{algebraically-free} if,
in addition, the forgetful functor
\begin{equation*}
  \eta^\ast \colon \mod[TX] \rightarrow \act
\end{equation*}
from $TX$-modules to \emph{$X$-actions} is an isomorphism of
categories. Here, an $X$-action is simply an object $A$ endowed with a
map $A \otimes X \rightarrow A$; whereas a $TX$-module is an
$TX$-action satisfying the usual associativity and unit laws. Free
monoids need not be algebraically-free but are, for example, in the
category of $\mathbb{R}$-vector spaces; this expresses, among other
things, the well-known fact that $\mathbb{R}[x]$-modules are the same
as vector spaces with a linear endomorphism.

In a symmetric monoidal category $\C$ we will, by analogy, say that
the free commutative monoid $SX$ on some $X \in \C$ is
\emph{algebraically-free} if the forgetful functor
\begin{equation*}
  \eta^\ast \colon \mod[SX] \rightarrow \ca
\end{equation*}
is an isomorphism of categories. Here, $\ca$ is the category of
\emph{commuting $X$-actions}, that is, actions
$\alpha \colon A \otimes X \rightarrow A$ for which the composite
morphism
${\alpha \circ (A \otimes \alpha) \colon A \otimes X \otimes X
  \rightarrow A}$ is unchanged on precomposing by the symmetry of $X$.
Again, free commutative monoids need not be algebraically-free, but
are in many categories of practical interest.

The relevance of this to our situation is as follows: for any monoidal
differential modality, the deriving transformation~\eqref{eq:68}
endows each $\oc X$ with an $X$-action, and the so-called
\emph{interchange} axiom for a deriving transformation states
precisely that this action is \emph{commuting}. Thus, perhaps naively,
we may think that, in order to pass from a monoidal coalgebra modality
$\oc$ to the free monoidal differential modality $\oc^\partial$
thereon, all we need do is freely adjoin commuting action structure.
Since, in the presence of algebraically-free commutative monoids,
the free commuting $X$-action on an object $A$ is given by
$A \otimes SX$, this suggests the formula:
\begin{equation}
  \label{eq:70}
  \oc^\partial X = \oc X \otimes SX\rlap{ .}
\end{equation}

Remarkably, this apparently naive approach works: the above formula
\emph{does} describe the free monoidal differential modality on the
monoidal coalgebra modality $\oc$. However, things are far from
straightforward: for as well as~\eqref{eq:70}, we must also figure out
the structure maps~\eqref{eq:71},~\eqref{eq:73}, \eqref{eq:72}
and~\eqref{eq:68} associated to $\oc^\partial$, and this is rather
subtle. To get a sense of this, note that the mere existence of free
commutative monoids is sufficient to equip each $\oc^\partial X$ with
a commuting $X$-action; however, just the universal property of a free
commutative monoid is \emph{not} sufficient to build the remaining
structure maps of $\oc^\partial$---we need the full strength of
algebraic-freeness.

Let us now return to the \emph{first} main theorem of our paper. This
is much the same as the second, but is concerned not with
\emph{monoidal coalgebra modalities} and \emph{monoidal differential
  modalities}, but with \emph{coalgebra modalities} and
\emph{differential modalities}. A coalgebra modality is like a
monoidal coalgebra modality, but comes equipped only with the
structure maps~\eqref{eq:71} and~\eqref{eq:72}, and not those
of~\eqref{eq:73}. A differential modality is a coalgebra modality
equipped with a deriving transformation~\eqref{eq:68} satisfying the
same axioms as before.

Note that coalgebra modalities and differential modalities are not
part of a linear logic picture, but are important in their own right.
For example, the \emph{$C^\infty$-ring
  modality}~\cite{Cruttwell2021Integral} on the opposite of the
category of vector spaces is a differential modality which captures
smooth differentiation---but is \emph{not} a monoidal differential
modality.

Happily, the story we have told above makes equal sense in the
non-monoidal context. Thus, our first main result is as follows; the
proof uses exactly the same construction~\eqref{eq:70} as before.
\begin{Thm*}\looseness=-1
  If $\C$ is $k$-linear symmetric monoidal with finite biproducts and
  algebraically-free commutative monoids, then every coalgebra
  modality on $\C$ can be completed freely to a differential modality;
  i.e., the forgetful functor from differential modalities on $\C$ to
  coalgebra modalities has a left adjoint.
\end{Thm*}

The third main result of our paper is an application of the second
one: we use it to construct the \emph{initial monoidal differential
  modality} on a suitable category $\C$. In order to do this, we need
only construct the initial monoidal coalgebra modality, and then apply
our free construction. (Note that the corresponding question of
finding the initial differential modality has a trivial answer: it is
the comonad constant at the initial object.)

It turns out that an initial monoidal coalgebra modality $P$
will exist whenever our base category $\C$ admits infinite coproducts
of its unit object $I$, which are preserved by tensor in each
variable; in this situation we can take:
\begin{equation*}
  PX = \bigoplus_{x \colon I \rightarrow X} I\rlap{ .}
\end{equation*}
Combining this with~\eqref{eq:70}, we conclude that the initial
monoidal differential modality on a base category satisfying the
necessary hypotheses is:
\begin{equation}
  \label{eq:59}
  P^\partial X = \bigoplus_{x \colon I \rightarrow X} SX\rlap{ .}
\end{equation}
Even in simple models of linear logic such as $\REL$, the category of
sets and relations, this gives a new exponential, and a new model of
differential linear logic. For a general $\C$, it seems like an
interesting problem---and one we hope to explore in future work---to
understand the co-Kleisli category and the co-Eilenberg--Moore
category of the comonad $P^\partial$. The co-Kleisli category is a
cartesian closed category and cartesian differential
category~\cite{Blute2009Cartesian}; the co-Eilenberg--Moore category
is, under mild assumptions on $\C$, a tangent
category---see~\cite{Cockett2020Tangent}.

We conclude this introduction with a brief overview of the contents of
the paper. In Section~\ref{sec:mono-coalg-mono}, we recall in detail
the notions of (monoidal) coalgebra modality and (monoidal)
differential modality which are our main concern. We also recall two
additional reformulations of monoidal coalgebra modalities in terms of
storage modalities\footnote{Also known as \emph{new-Seely
    categories}.} and additive bialgebra modalities.

In Section~\ref{sec:algebr-free-comm}, we introduce and study
algebraically-free commutative monoids. In particular, we see that
free monoids need \emph{not} be algebraically-free, but describe
broadly applicable conditions under which they will be---namely, when
they can be constructed via the familiar exponential formula
$X \mapsto \sum_{n} X^{\otimes n} / \mathfrak{S}_n$.

At this point, we could jump straight into the construction of the
free differential modality on a coalgebra modality; however, if we did
so, the arguments involved would seem rather opaque. So instead, we
first prove some results which make transparent \emph{why} the
construction has to be the way it is---while also being of interest in
their own right.

Section~\ref{sec:lift-mono-diff} gives the first of these, which
explains how a differential modality $\oc$ on a base category $\C$ can
be lifted to a differential modality $\oc^X$ on the category $\ca$ of
commuting $X$-actions over any $X \in \C$. Interestingly, this lifted
differential modality need not be monoidal, even if the one on the
base is so; thus, this result yields an ample new supply of
differential modalities which are not monoidal.

Section~\ref{sec:diff-modal-axioms} builds on the results of
Section~\ref{sec:lift-mono-diff} by explaining how each of the axioms
for a differential modality may be explained in terms of commuting
$X$-actions. We already noted above that the \emph{interchange} axiom
states precisely that the deriving transformation~\eqref{eq:68} is a
commuting $X$-action; but to give a further example, the \emph{chain
  rule} axiom turns out to be equivalent to the assertion that the
component of the comonad counit
$\delta \colon \oc \Rightarrow \oc \oc$ at $X$ underlies a map of
commuting $X$-actions
$(\oc X, \mathsf{d}_X) \rightarrow \oc^X (\oc X,
\mathsf{d}_X)$---where here we use the results of
Section~\ref{sec:lift-mono-diff} to lift $\oc$ to a modality $\oc^X$
on the category of commuting $X$-actions.

In Section~\ref{sec:constr-free-mono}, we are finally ready to give
our first main result, by constructing the free differential modality
$\oc^\partial$ on a given coalgebra modality $\oc$. As explained
above, we assume the existence of algebraically-free commutative
monoids, and take $\oc^\partial X = \oc X \otimes SX$. The unit
exhibiting this as free will be given by
$\xi_X = \oc X \otimes \eta_X \colon \oc X \rightarrow \oc X \otimes
SX$, where $\eta_X \colon I \rightarrow SX$ is the unit of the
algebraically-free commutative monoid. As for the remaining structure
of $\oc^\partial$, this is now completely determined by the
characterising conditions from Section~\ref{sec:diff-modal-axioms}
together with the fact that $\oc^\partial X$ is the free commuting
$X$-action on $SX$, and the fact that $\xi$ is to be a map of
coalgebra modalities. Thus, the constructions involved---which
otherwise would seem somewhat contrived---are in fact natural, and
completely forced upon us.

Section~\ref{sec:free-mono-diff} builds upon
Section~\ref{sec:constr-free-mono} to give our second main result,
describing the free \emph{monoidal} differential modality on a
\emph{monoidal} coalgebra modality: thus, the construction which turns
a model of multiplicative exponential linear logic (MELL) into a model
of \emph{differential} MELL. As noted above, the construction is the
same construction, and thus it is simply a matter of showing that
monoidality is preserved.

In Section~\ref{sec:init-mono-diff}, we turn to our third and final
main result, which constructs the initial monoidal differential
modality on any suitable $k$-linear symmetric monoidal category; and
finally, in Section~\ref{sec:examples}, we work this out in a number
of familiar contexts, including the category of sets and relations,
the category of $k$-modules over a rig $k$, the category of linear
species, and the category of super vector spaces.

\section{Modalities for differentiation}
\label{sec:mono-coalg-mono}

In this section, we recall the key players that we will be concerned
with in the rest of the paper. First, we define the notions of
\emph{coalgebra modality} and \emph{monoidal coalgebra modality} on a
symmetric monoidal category; as explained in the introduction, the
latter is exactly what is needed to model the $\oc$ modality of linear
logic. Then we describe how these notions can be enhanced with a
differential operation to yield the notions of \emph{differential
  modality} and \emph{monoidal differential modality}.

Along the way, we discuss various reformulations of the above notions
which are available in categories with finite products or with
$k$-linear structure. Much of this is background, but there are new
results---Propositions~\ref{prop:1} and~\ref{prop:9}---and attention
is paid to \emph{morphisms} between these modalities, which has
not been a focus in previous work.

\subsection{Coalgebra modalities and monoidal coalgebra modalities}
\label{sec:mono-coalg-modal}

The modalities we consider will be \emph{comonads} on a symmetric
monoidal category $(\C, \otimes, I)$. Following the notation
of~\cite{Blute2019Differential}, we will write $\oc$ for such a
comonad, and $\varepsilon \colon \oc \Rightarrow \mathrm{id}_\C$ and
$\delta \colon \oc \Rightarrow \oc\oc$ for its counit and
comultiplication. The most fundamental additional structure we will
consider on such a comonad is:
\begin{Defn}[Coalgebra modality]
  \label{def:5}
  A comonad $(\oc, \varepsilon, \delta)$ on a symmetric monoidal
  $(\C, \otimes, I)$ is a \emph{(cocommutative) coalgebra modality} if
  it comes equipped with maps of the following form, natural in 
  $X \in \C$:
  \begin{equation}\label{eq:14}
    \mathsf{e}_X \colon \oc X \rightarrow I \qquad \text{and}
    \qquad \Delta_X \colon \oc X \rightarrow \oc X \otimes \oc X\rlap{ ,}
  \end{equation}
  which endow each $\oc X$ with cocommutative comonoid structure in
  such a way that each $\delta_X$ is a map of comonoids
  $(\oc X, \Delta_X, \mathsf{e}_X) \rightarrow (\oc\oc X, \Delta_{\oc
    X}, \mathsf{e}_{\oc X})$.
\end{Defn}
\begin{Rk}
  \label{rk:3}
  To make a comonad $\oc$ into a coalgebra
  modality is equally to give a comparison functor
  \begin{equation}
    \label{eq:47}
    \cd[@!C@C-2.5em@-0.2em]{
      {\cat{Coalg}(\oc)} \ar@{-->}[rr]^-{\smash{K}} \ar[dr]_-{U^\oc} & &
      {\cat{Comon}(\C)} \ar[dl]^-{U^\cat{Comon}} \\ &
      {\C}
    }
  \end{equation}
  from the category of Eilenberg--Moore coalgebras for $\oc$ to the
  category of cocommutative comonoids in $\C$. Indeed, since each
  $\oc$-coalgebra is a $U^\oc$-absolute equaliser of cofree
  $\oc$-coalgebras, such a factorisation is uniquely determined by its
  restriction to the cofree coalgebras---and the data of such a
  restriction amounts exactly to comonoid structures~\eqref{eq:14}
  making each $\delta_X$ a homomorphism. For a general coalgebra $(X, \xi
  \colon X \rightarrow \oc X)$, its image
  under $K$ is then forced to be $X$ endowed with the comonoid structure maps
  \begin{equation}
    \label{eq:48}
    X \xrightarrow{\xi} \oc X \xrightarrow{\Delta_X} \oc X\otimes \oc X \xrightarrow{\varepsilon_X \otimes \varepsilon_X} X \otimes X \ \ \text{and} \ \  X \xrightarrow{\xi} \oc X \xrightarrow{\mathsf{e}_X} I\rlap{ .}
  \end{equation}
\end{Rk}

\begin{Ex}[Cofree comonoids as a coalgebra modality]
  \label{ex:5}
  If the forgetful functor $\cat{Comon}(\C) \rightarrow \C$ has a
  right adjoint, then it will be strictly comonadic, so that the
  category of coalgebras for the induced \emph{cofree cocommutative
    comonoid comonad} $\oc^C$ is isomorphic to $\cat{Comon}(\C)$ over
  $\C$. In particular, this isomorphism witnesses $\oc^C$ as a
  coalgebra modality.
\end{Ex}

The comonad $\oc^C$ is not just a coalgebra modality: it is in fact a
model of the exponential modality of linear logic. The additional
ingredient needed for this is:

\begin{Defn}[Monoidal comonad]
  \label{def:15}
  A comonad $(\oc, \varepsilon, \delta)$ a symmetric monoidal
  $(\C, \otimes, I)$ is \emph{(symmetric) monoidal} if it comes
  equipped with maps
  \begin{equation}\label{eq:15}
    m_I \colon I \rightarrow \oc I
    \qquad \text{and} \qquad m_\otimes \colon \oc X \otimes \oc Y \rightarrow \oc(X \otimes Y) \text{ for all $X,Y \in \C$}
  \end{equation}
  making $\oc$ into a symmetric monoidal functor, and $\varepsilon$
  and $\delta$ into monoidal natural transformations; see, for
  example,~\cite[\sec 7]{Moerdijk2002Monads} for the conditions
  involved.
\end{Defn}

\begin{Rk}
  \label{rk:8}
  Giving the structure in~\eqref{eq:15} on a comonad $\oc$ is
  equivalent to giving a lifting of the symmetric monoidal structure
  of $\C$ to the category of Eilenberg--Moore coalgebras
  $\cat{Coalg}(\oc)$. On the one hand, monoidal structure on
  $\oc$ yields a lifted monoidal structure on $\cat{Coalg}(\oc)$,
  with unit $\hat I$ and binary tensor $\hat \otimes$ given by:
  \begin{equation}
    \label{eq:13}
    \begin{aligned}
      \hat I &= (m_I \colon I \rightarrow \oc I) \\
      \text{and }\ \ (X \xrightarrow{x} \oc X) \mathbin{\hat\otimes}
      (Y \xrightarrow{y} \oc Y) &= \bigl(X \otimes Y \xrightarrow{x
        \otimes y} \oc X \otimes \oc Y \xrightarrow{m_\otimes} \oc(X
      \otimes Y)\bigr)\rlap{ .}
    \end{aligned}
  \end{equation}
  
  On the other hand, from a lifted monoidal structure, we re-find
  $m_I$ and $m_\otimes$ as the unique maps of $\oc$-coalgebras
  $\hat I \rightarrow (\oc I, \delta_I)$ and
  $(\oc X, \delta_X) \mathbin{\hat \otimes} (\oc Y, \delta_Y)
  \rightarrow (\oc(X \otimes Y), \delta_{X \otimes Y})$ which render
  commutative the respective diagrams
  \begin{equation*}
    \cd[@!C]{
      I \ar[r]^-{m_I} \ar@{=}[dr]^-{} & \oc I \ar[d]^-{\varepsilon_I} & \oc X \otimes \oc Y \ar[dr]_-{\varepsilon_X \otimes \varepsilon_Y} \ar[r]^-{m_\otimes} & \oc (X \otimes Y) \ar[d]^-{\varepsilon_{X \otimes Y}} \\
      & I & & X \otimes Y\rlap{ .}
    }
  \end{equation*}
\end{Rk}

\begin{Ex}[Cofree comonoids as a monoidal comonad]
  \label{ex:2}
  The symmetric monoidal structure of $\C$ \emph{always} lifts to
  $\cat{Comon}(\C)$; the unit of the lifted monoidal structure is the
  cocommutative comonoid
  \begin{equation*}
    \tilde I = (I,\, I \xrightarrow{\mathrm{id}} I,\, I \xrightarrow{\cong} I \otimes I)
  \end{equation*}
  and the binary tensor of cocommutative comonoids
  $(X^1, \mathsf{e}^1, \Delta^1) \mathbin{\tilde{\otimes}} (X^2,
  \mathsf{e}^2, \Delta^2)$ is given by $X^1 \otimes X^2$ endowed with the
  counit and comultiplication maps
  \begin{equation}
    \label{eq:46}
    \begin{gathered}
      X^1 \otimes X^2 \xrightarrow{\mathsf{e}^1 \otimes \mathsf{e}^2} I \ \ \ \text{ and } \\[3pt]
      \smash{X^1 \otimes X^2 \xrightarrow{\Delta^1 \otimes
        \Delta^2} X^1 \otimes X^1 \otimes X^2 \otimes X^2
      \xrightarrow{1 \otimes \sigma \otimes 1} X^1 \otimes X^2 \otimes
      X^1 \otimes X^2}\text{ .}
    \end{gathered}
  \end{equation}
  Thus, when cofree cocommutative comonoids exist, the cofree
  cocommutative comonoid comonad $\oc^C$ is a monoidal comonad.
\end{Ex}
\begin{Rk}
  \label{rk:5}
  We observe the important fact that the lifted monoidal structure on
  $\mathsf{Comon}(\C)$ is in fact cartesian; this dualises the
  well-known fact that the tensor product of commutative rings is also
  their coproduct. Here, to be clear, we say that a monoidal structure
  is \emph{cartesian} if its unit object $I$ is terminal, and for each
  pair of objects $X,Y$ the following span is a product cone:
  \begin{equation}
    \label{eq:45}
    X \xleftarrow{\cong} X \otimes I \xleftarrow{X \otimes \,\text{unique}} X \otimes Y \xrightarrow{\text{unique}\, \otimes Y} I \otimes Y \xrightarrow{\cong} Y\rlap{ .}
  \end{equation}
\end{Rk}

A comonad is a model of the exponential modality of linear logic when
it is is both a coalgebra modality and a monoidal comonad in a
compatible way; we speak in this situation of a \emph{monoidal
  coalgebra modality}.

\begin{Defn}[Monoidal coalgebra modality]
  A \emph{monoidal coalgebra modality} on a symmetric monoidal
  $(\C, \otimes, I)$ is a comonad $\oc $ endowed with coalgebra
  modality structure and monoidal comonad structure, in such a way
  that the maps in~\eqref{eq:15} are comonoid homomorphisms
  \begin{equation*}
    \tilde I \xrightarrow{\!m_I\!} (\oc I, \Delta_I, \mathsf{e}_I) \ \ \text{and} \  \  (\oc X, \Delta_X, \mathsf{e}_X) \mathbin{\tilde \otimes} (\oc Y,
    \Delta_Y, \mathsf{e}_Y) \xrightarrow{\!m_{XY}\!} (\oc(X \otimes Y),
    \Delta_{X \otimes Y}, \mathsf{e}_{X \otimes Y})
  \end{equation*}
  and the maps in~\eqref{eq:14} are maps of $\oc $-coalgebras
  \begin{equation*}
    (\oc X, \delta_X) \xrightarrow{\mathsf{e}_X} \hat I \ \ \text{and
    } \ \  (\oc X, \delta_X) \xrightarrow{\Delta_X} (\oc X, \delta_X) \mathbin{\hat \otimes} (\oc X,
    \delta_X) \text{.} 
  \end{equation*}
\end{Defn}
It may be helpful to take a semantic perspective on the above
conditions.
\begin{Prop}
  \label{prop:1}
  Let $\oc$ be a coalgebra modality and a monoidal comonad on the
  symmetric monoidal $\C$; so we have a
  comparison functor $K \colon {\cat{Coalg}(\oc) \rightarrow \cat{Comon}(\C)}$ as
  in Remark~\ref{rk:3} and a lifted monoidal structure
  $\tilde \otimes$ on $\cat{Coalg}(\oc)$ as in Remark~\ref{rk:8}.
  These data comprise a monoidal coalgebra modality on $\C$ if and
  only if:
  \begin{enumerate}[(i)]
  \item $K$ is a strict monoidal functor
    $(\cat{Coalg}(\oc), \tilde \otimes) \rightarrow (\cat{Comon}(\C),
    \hat\otimes)$; and 
  \item The lifted monoidal structure $\smash{\tilde \otimes}$ on
    $\cat{Coalg}(\oc)$ is cartesian.
  \end{enumerate}
\end{Prop}

\begin{proof}
  We first deal with the ``if'' direction. The maps in~\eqref{eq:15}
  are, by Remark~\ref{rk:8}, maps of $\oc$-coalgebras
  $\hat I \rightarrow (\oc I, \delta_I)$ and
  $(\oc X, \delta_X) \mathbin{\hat \otimes} (\oc Y, \delta_Y)
  \rightarrow (\oc(X \otimes Y), \delta_{X \otimes Y})$, whence their
  images under the strict monoidal $K$ are comonoid homomorphisms of
  the required form. On the other hand, we have in the cartesian
  $(\cat{Coalg}(\oc), \hat \otimes)$, a unique map
  $(\oc X, \delta_X) \rightarrow \hat I$ for each $X$. Because $K$ is
  strict monoidal, this map is also a comonoid map
  $(\oc X, \Delta_X, \mathsf{e}_X) \rightarrow \tilde I$; but (by
  cartesianness of $\tilde \otimes$) the unique such is
  $\mathsf{e}_X$, and so $\mathsf{e}_X$ is a $\oc$-coalgebra
  homomorphism. Similarly, the product cone~\eqref{eq:45} in
  $\cat{Coalg}(\oc)$ is sent by $K$ to the corresponding product cone
  in $\cat{Comon}(\C)$; whence the diagonal map
  $(\oc X, \delta_X) \rightarrow (\oc X, \delta_X) \hat \otimes (\oc
  X, \delta_X)$ in $\cat{Coalg}(!)$ is also the diagonal map
  $(\oc X, \Delta_X, \mathsf{e}_X) \rightarrow (\oc X, \Delta_X,
  \mathsf{e}_X) \tilde \otimes (\oc X, \Delta_X, \mathsf{e}_X)$ in
  $\cat{Comon}(\C)$---and so must equal $\Delta_X$, which is thus a
  $!$-coalgebra homomorphism.

  We now prove the ``only if'' direction. First, we show that $K(\hat
  I) = \tilde I$. From~\eqref{eq:13} and~\eqref{eq:48}, $K(\hat I)$ 
  has the comonoid structure
  \begin{equation*}
    I \xrightarrow{m_I} \oc I \xrightarrow{\Delta_I} \oc I\otimes \oc I \xrightarrow{\varepsilon_I \otimes \varepsilon_I} I \otimes I \ \ \text{and} \ \  I \xrightarrow{m_I} \oc I \xrightarrow{\mathsf{e}_I} I\rlap{ .}
  \end{equation*}
  Because $m_I \colon \tilde I \to (\oc I, \Delta_I, \mathsf{e}_I)$,
  the right-hand composite is the counit $1_I \colon I \rightarrow I$
  of $\tilde I$; while the left-hand composite is equally the map
  \begin{equation*}
    I \xrightarrow{\Delta^{\tilde I}} I \otimes I \xrightarrow{m_I \otimes m_I} \oc I \otimes \oc I \xrightarrow{\varepsilon_I \otimes \varepsilon_I} I \otimes I
  \end{equation*}
  which equals $\Delta^{\tilde I}$ due to the monoidal comonad axiom
  $\varepsilon_I \circ m_I = 1_I$.

  We now show 
  $K\bigl((X, \xi) \hat \otimes (Y, \gamma)\bigr) = K(X, \xi) \tilde
  \otimes K(Y, \gamma)$. The counit maps of the two sides are, 
  by~\eqref{eq:48},~\eqref{eq:13} and~\eqref{eq:46}, given by
  \begin{equation*}
    X \otimes Y \xrightarrow{\!\xi \otimes \gamma\!} \oc X \otimes \oc Y \xrightarrow{m_{XY}} \oc(X \otimes Y) \xrightarrow{\!\mathsf{e}_{X \otimes Y}\!} I \ \text{and} \ 
    X \otimes Y \xrightarrow{\!\xi \otimes \gamma\!} \oc X \otimes \oc Y \xrightarrow{\!\mathsf{e}_{X} \otimes \mathsf{e}_Y\!} I\rlap{ ,}
\end{equation*}
which are equal since $m_{XY}$ is a comonoid homomorphism. On the
other hand, the comultiplications are given by the outside composites
in:
\begin{equation*}
  \cd[@C-0.5em]{
    \sh{r}{1em}X\otimes Y \ar[r]^-{\xi \otimes  \gamma} &
    \oc X \otimes  \oc Y \ar[d]_-{\Delta_X \otimes  \Delta_Y} \ar[r]^-{m_{XY}} &
    \oc(X \otimes Y) \ar[r]^-{\Delta_{X\otimes  Y}} &
    \oc(X \otimes Y)\otimes \oc(X \otimes Y) \ar[d]^-{\varepsilon_{X\otimes Y}\otimes  \varepsilon_{X\otimes Y}} \\
    & \sh{l}{1.2em} \oc X\otimes  \oc X\otimes  \oc Y \otimes  \oc Y \ar[r]_-{1\otimes \sigma\otimes 1} &
    \oc X \otimes \oc Y\otimes \oc X\otimes \oc Y \ar[ur]|-{m_{XY}\otimes m_{XY}} \ar[r]_-{\varepsilon_X\otimes \varepsilon_Y\otimes \varepsilon_X\otimes \varepsilon_Y}&
    \sh{r}{1em}X\otimes Y\otimes X\otimes Y\rlap{ .}
  }
\end{equation*}
Herein, the left-hand region commutes since $m_{XY}$ is a comonoid
homomorphism, and the right-hand region commutes by monoidality of
$\varepsilon$.
Thus $K$ is strict monoidal, and it remains to show that the lifted
monoidal structure $\tilde \otimes$ is cartesian. This is done, for
example, in~\cite[Lemma~4]{Benton1995A-mixed}.
\end{proof}

\begin{Ex}[Cofree comonoids as a monoidal coalgebra modality]
  \label{ex:3}
  For the cofree cocommutative comonoid comonad $\oc^C$, made monoidal
  as in Example~\ref{ex:2}, the map $K$ in (i) is strict monoidal by
  construction. The lifted monoidal structure on
  $\mathsf{Coalg}(\oc^C)$ is cartesian by Remark~\ref{rk:5}, and so
  $\oc^C$ is a monoidal coalgebra modality. In particular, it models
  the exponential modality of linear logic; a model of linear logic
  whose exponential is modelled by $\oc^C$ is often referred to as a
  \emph{Lafont category}.
\end{Ex}

From the preceding proposition we obtain the following remarkable
result. While the first part is well-known, the second appears to be
new---though for the uniqueness claim, see also
Proposition~12 of the second author's recent~\cite{Lemay2025Additive}.
\begin{Prop}
  \label{prop:9}
  Let $\oc$ be a comonad on the symmetric monoidal $\C$.
  \begin{enumerate}[(i), itemsep=0.25\baselineskip]
  \item If $\oc$ is a monoidal comonad, then there is at most one
    coalgebra modality structure on $\oc$ making it into a monoidal
    coalgebra modality---which exists precisely when the lifted
    monoidal structure $\tilde \otimes$ on $\cat{Coalg}(\oc)$ is cartesian.
  \item If $\oc$ is a coalgebra modality, then there is at most one
    monoidal comonad structure on $\oc$ making it into a monoidal
    coalgebra modality---which exists precisely when
    $K \colon \cat{Coalg}(\oc) \rightarrow \cat{Comon}(\C)$ creates
    finite products.
  \end{enumerate}
\end{Prop}
Here, recall that if $\E$ is a category with finite products, then a
functor $F \colon \D \rightarrow \E$ \emph{creates finite products} if
$\D$ has finite products, and $F$ preserves and reflects them.
\begin{proof}
  For (i), it is certainly necessary for the existence that the lifted
  monoidal structure $\tilde \otimes$ on $\cat{Coalg}(\oc)$ be
  cartesian. But then, by the main result of~\cite{Fox1976Coalgebras},
  there is a \emph{unique} strict monoidal functor
  $K \colon \cat{Coalg}(\oc) \rightarrow \cat{Comon}(\C)$
  rendering~\eqref{eq:47} commutative, which provides the unique
  compatible coalgebra modality structure.

  For (ii), note that the given condition is again necessary. Indeed,
  if $\oc$ is a monoidal coalgebra modality, then $\cat{Coalg}(\oc)$
  will certainly have finite products; and since $K$ is strict
  monoidal, it will preserve them. Moreover, since $K$ fits into a
  triangle~\eqref{eq:47} wherein the functors into $\C$ are
  conservative, it is itself conservative, and so also reflects
  products.

  Now, under this condition, we show there is a unique way of lifting
  $\otimes$ to a cartesian monoidal structure $\hat \otimes$ on
  $\cat{Coalg}(!)$ which is preserved strictly by $K$. Firstly, for
  the unit, let $(T, \tau)$ be any terminal object of
  $\cat{Coalg}(\oc)$. Since $K$ preserves finite products,
  $K(T, \tau)$ is terminal in $\cat{Comon}(\C)$, and so there is a
  unique isomorphism $\theta \colon K(T, \tau) \rightarrow \tilde I$.
  We can transport the $\oc$-coalgebra structure of $T$ across this
  isomorphism to obtain
  $\hat I \defeq (I, \oc \theta \circ \tau \circ \theta^{-1})$ and an
  isomorphism $\theta \colon (T, \tau) \rightarrow \hat I$. In
  particular, $\hat I$ is terminal and is mapped by $K$ to $\tilde I$.
  Suppose now that $J \in \cat{Coalg}(\oc)$ is also terminal and
  mapped by $K$ to $\tilde I$. By uniqueness of terminal objects,
  there is a unique isomorphism $\psi \colon J \rightarrow \hat I$;
  but then $K\psi \colon \tilde I \rightarrow \tilde I$ must be the
  identity, whence $\psi = 1_I$ and so $J = \hat I$.

  Secondly, for the binary multiplication, note that each
  diagram~\eqref{eq:45} in $\cat{Coalg}(\oc)$ must be a product cone
  sitting over the corresponding product in $\cat{Comon}(\C)$.
  Thus, for $(X, \xi), (Y, \gamma)$ in $\cat{Coalg}(\oc)$, let
  $\pi_1 \colon (X, \xi) \leftarrow(X, \xi) \times (Y, \gamma)
  \rightarrow (Y, \gamma) \colon \pi_2$ be \emph{any} product cone;
  since $K$ preserves it, we have a unique isomorphism rendering
  commutative the diagram
  \begin{equation*}
    \cd[@R-0.2em@C+2em]{
      & K((X, \xi) \times (Y, \gamma)) \ar[dl]_-{K\pi_1} \ar[dr]^-{K\pi_2} \ar[d]^-{\theta} \\
      K(X, \xi) & K(X, \xi) \tilde \otimes K(Y, \gamma) \ar[l]^-{1 \otimes \,\text{unique}} \ar[r]_-{\text{unique}\, \otimes 1} & K(Y, \gamma)\rlap{ .} \\
    }
  \end{equation*}
  Transporting the $\oc$-coalgebra structure of
  $(X, \xi) \times (Y, \gamma)$ along this isomorphism $\theta$, and
  arguing as before, we obtain a unique object
  $(X, \xi) \hat \otimes (Y, \gamma)$ for which the cone~\eqref{eq:45}
  is a product which is mapped to the corresponding cone in
  $\cat{Comon}(\C)$. To conclude the proof, it remains only to define
  $\hat \otimes$ on morphisms and provide the coherence cells for the
  monoidal structure. The choices are forced by faithfulness of
  $U^\oc \colon \cat{Coalg}(\oc) \rightarrow \C$, and are well-defined
  via the universality of finite products.
\end{proof}

Thus, being a monoidal coalgebra modality is a mere property of either
a coalgebra modality, or of a monoidal comonad. We can strengthen this
statement by considering the categories formed by our various
kinds of comonad. Of course, if $\oc^1$ and $\oc^2$ are two comonads
on $\C$, then a \emph{comonad morphism}
$\varphi \colon \oc^1 \rightarrow \oc^2$ is a natural transformation
such that we have commutativity in diagrams of the form:
\begin{equation*}
  \cd[@C+0.5em@-0.3em]{
    {\oc ^1 X} \ar[rr]^-{\varphi_X} \ar[dr]_-{\varepsilon^1_X} & &
    {\oc ^2 X} \ar[dl]^-{\varepsilon^2_X} & \oc ^1 X \ar[r]^-{\varphi_X} \ar[d]_-{\delta^1_X}  & \oc ^2 X \ar[d]^-{\delta^2_X}\\ &
    {X} & & \oc ^1 \oc ^1 X \ar[r]^-{(\varphi\varphi)_X} & \oc ^2 \oc ^2 X \rlap{ .}
  }
\end{equation*}
\begin{Defn}[Morphisms of coalgebra modalities, monoidal
  comonads, and monoidal coalgebra modalities]
  \label{def:8}
  Let $\varphi \colon \oc^1 \rightarrow \oc^2$ be a morphism between
  comonads on the symmetric monoidal $\C$.
  \begin{enumerate}[(i)]
  \item If $\oc^1$ and $\oc^2$ are coalgebra modalities, then
    $\varphi$ is a \emph{coalgebra modality morphism} if it renders
    commutative diagrams of the form:
    \begin{equation*}
      \cd[@!C@C-2em@-0.3em]{
        {\oc ^1 X} \ar[rr]^-{\varphi_X} \ar[dr]_-{\mathsf{e}^1_X} & &
        {\oc ^2 X} \ar[dl]^-{\mathsf{e}^2_X} & \oc ^1 X \ar[rr]^-{\varphi_X} \ar[d]_-{\Delta^1_X}  & & \oc ^2 X \ar[d]^-{\Delta^2_X}\\ &
        {I} & & \oc^1 X \otimes \oc^1 X \ar[rr]^-{\varphi_{X} \otimes \varphi_X} & & \oc^2 X \otimes \oc^2 X\rlap{ ,}
      }
    \end{equation*}
    i.e., if each $\varphi_X$ is a comonoid homomorphism. \vskip0.25\baselineskip
  \item If $\oc^1$ and $\oc^2$ are monoidal comonads, then $\varphi$
    is a \emph{monoidal comonad morphism} if its underlying natural
    transformation is monoidal, i.e., renders commutative diagrams of
    the form:
    \begin{equation*}
      \cd[@!C@C-2em@-0.3em]{
        & {I} \ar[dl]_-{m^1_I} \ar[dr]^-{m^2_I} & & 
        {\oc^1 X \otimes \oc^1 Y} \ar[d]_-{m^1_\otimes} \ar[rr]^-{\varphi_X \otimes \varphi_Y} &&
        {\oc^2 X \otimes \oc^2 Y} \ar[d]_-{m^2_\otimes}\\
        {\oc ^1 I} \ar[rr]_-{\varphi_I} & &
        {\oc ^2 I} &
        {\oc^1(X \otimes Y)} \ar[rr]^-{\varphi_{X \otimes Y}} && \oc^2(X\otimes Y)\rlap{ .}
      }
    \end{equation*}
  \item If $\oc^1$ and $\oc^2$ are monoidal coalgebra modalities, then
    $\varphi$ is a \emph{monoidal coalgebra modality morphism} if it
    is both a coalgebra modality morphism and a monoidal comonad morphism.
  \end{enumerate}
  We write $\cat{Comonad}(\C)$, $\cat{CoalgMod}(\C)$, $\cat{MonComonad}(\C)$
  and $\cat{MonCoalgMod}(\C)$ for the categories of comonads,
  coalgebra modalities, monoidal comonads and monoidal coalgebra
  modalities on $\C$ respectively.
\end{Defn}
Recall the semantic characterisation of comonad maps
$\varphi \colon \oc^1 \rightarrow \oc^2$; they correspond functorially
to maps between the categories of coalgebras:
\begin{equation}
  \label{eq:49}
  \cd[@-1em]{
    {\cat{Coalg}(\oc^1)} \ar[rr]^-{\varphi_\ast} \ar[dr]_-{U^{\oc^1}} & &
    {\cat{Coalg}(\oc^2)}\rlap{ ,} \ar[dl]^-{U^{\oc^2}} \\ &
    {\C}
  }
\end{equation}
where the $\varphi_\ast$ associated to $\varphi$ sends a
$\oc^1$-coalgebra $(X, \xi)$ to the $\oc^2$-coalgebra
$(X, \varphi_X \circ \xi)$. (We re-find $\varphi$ from $\varphi_\ast$
by looking at its action on cofree $\oc^1$-coalgebras.) We can extend
this semantic characterisation to maps of coalgebra modalities 
and monoidal comonads; the proof is routine and left to the reader.
\begin{Lemma}
  \label{lem:19}
  Let $(\C, \otimes, I)$ be a symmetric monoidal category, and let
  $\varphi \colon \oc^1 \rightarrow \oc^2$ be a morphism between
  comonads on $\C$.
  \begin{enumerate}[(i)]
  \item \label{item:1} If $\oc^1$ and $\oc^2$ are coalgebra modalities, then
    $\varphi$ is a coalgebra modality morphism if and only the
    $\varphi_\ast$ of~\eqref{eq:49} commutes with the comparison
    functors to $\cat{Comon}(\C)$:
    \begin{equation}
      \label{eq:50}
      \cd[@-1em@C-1em]{
        {\cat{Coalg}(\oc^1)} \ar[rr]^-{\varphi_\ast} \ar[dr]_-{K^{\oc^1}} & &
        {\cat{Coalg}(\oc^2)}\rlap{ .} \ar[dl]^-{K^{\oc^2}} \\ &
        {\cat{Comon}(\C)}}
    \end{equation}
  \item \label{item:2} If $\oc^1$ and $\oc^2$ are monoidal comonads, then $\varphi$
    is a monoidal comonad morphism if and only if $\varphi_\ast$ is
    strict monoidal for the lifted monoidal structures on
    $\cat{Coalg}(\oc^1)$ and $\cat{Coalg}(\oc^2)$.
  \end{enumerate}
\end{Lemma}
Using this, we can strengthen the sense in which being a monoidal
coalgebra modality is a mere property of a
coalgebra modality or a monoidal comonad: 
\begin{Prop}
  \label{prop:4}
  For any symmetric monoidal $\C$, the forgetful functors
  \begin{equation}
    \label{eq:54}
    \cd[@-1em]{
      & \cat{MonCoalgMod}(\C) \ar[dl]_-{} \ar[dr]^-{} \\
      \cat{CoalgMod}(\C) & & \cat{MonComonad}(\C)
    }
  \end{equation}
  are full and faithful, and injective on objects, with the objects in
  their images characterised as in Proposition~\ref{prop:9}.
\end{Prop}
\begin{proof}
  The novel content of this statement is that a comonad map
  $\varphi \colon \oc^1 \rightarrow \oc^2$ between monoidal coalgebra
  modalities is a monoidal comonad map if and only if it is a
  coalgebra modality map. For this, we show that
  the induced $\varphi_\ast$ of~\eqref{eq:49} satisfies
  condition~\ref{item:1} of Lemma~\ref{lem:19} if and only if it
  satisfies~\ref{item:2}.

  For~\ref{item:1} $\Rightarrow$~\ref{item:2}, suppose $\varphi_\ast$
  is strict monoidal. By~\cite{Fox1976Coalgebras},
  $U^{\cat{Comon}} \colon \cat{Comon}(\C) \rightarrow \C$ is terminal
  in the category $\cat{CartMon}/\C$ of cartesian monoidal categories
  over $\C$ and strict monoidal functors. Since $\varphi_\ast$ is
  strict monoidal,~\eqref{eq:50} underlies a diagram of maps in this category
  over the terminal object, and thus commutes.
 
  For~\ref{item:2} $\Rightarrow$~\ref{item:1}, suppose~\eqref{eq:50}
  commutes. Since by Proposition~\ref{prop:9} both $K^1$ and $K^2$
  preserve and reflect finite products, so too does $\varphi_\ast$.
  As both $\cat{Coalg}(\oc^1)$ and $\cat{Coalg}(\oc^2)$ are
  cartesian, it follows that $\varphi_\ast$ is strong monoidal and
  that~\eqref{eq:50} commutes in the category of monoidal categories
  and strong monoidal functors. But since $K^2$ therein reflects
  identities and $K^1$ is strict monoidal, $\varphi_\ast$ must also be
  strict monoidal.
\end{proof}

\subsection{Storage modalities and additive bialgebra modalities}
\label{sec:stor-modal-addit}

In this section, we recall two additional ways of formulating the
notion of monoidal coalgebra modality when the
base symmetric monoidal category $\C$ has additional structure.

The first of these concerns the case where $\C$
has finite products. As motivation for this, note that if $\oc $ is a
monoidal coalgebra modality on such a $\C$, then its finite products are
preserved by the (right adjoint) cofree functor
$\C \rightarrow \cat{Coalg}(\oc )$. Since finite products in
$\cat{Coalg}(\oc )$ are realised by the lifted monoidal structure of
$\C$, this preservation says that the following maps, known as the
\emph{Seely maps}, or the \emph{storage maps}, are isomorphisms:
\begin{equation}
  \label{eq:11}
  \begin{gathered}
    \chi_\top \defeq \oc \top \xrightarrow{\mathsf{e}_\top} I \qquad \text{and} \\[0.4em]
    \smash{\chi_{XY} \defeq \oc(X \times Y) \xrightarrow{\Delta_{X
          \times Y}} \oc(X \times Y) \otimes \oc(X \times Y)
      \xrightarrow{\oc\pi_0 \otimes \oc\pi_1} \oc X \otimes \oc Y\rlap{
        .}}
  \end{gathered}
\end{equation}

The property that these maps are invertible provides an alternative
characterisation of the coalgebra modalities that underlie monoidal
coalgebra modalities. Indeed, \cite[Theorem~3.1.6]{Blute2015Cartesian}
shows that, if a coalgebra modality has invertible storage
maps~\eqref{eq:11}, then it is monoidal via the maps $m_I$ and
$m_\otimes$ found as:
\begin{equation}
  \label{eq:42}
  \begin{gathered}
    I \xrightarrow{\chi_1^{-1}} \oc\top \xrightarrow{\delta_\top}
    \oc\oc\top \xrightarrow{\oc\chi_\top} \oc I \qquad \text{and} \\[0.4em]
    \!\smash{\oc X \otimes \oc Y \xrightarrow{\!\chi_{XY}^{-1}\!} \oc(X \times
      Y) \xrightarrow{\!\delta\!} \oc\oc(X \times Y)
      \xrightarrow{\oc\chi_{XY}} \oc(\oc X \otimes \oc Y)
      \xrightarrow{\!\oc(\varepsilon \otimes \varepsilon)\!} \oc(X
      \otimes Y)}\rlap{ .}
  \end{gathered}
\end{equation}
A coalgebra modality with invertible storage maps is sometimes
referred to as a \emph{storage modality}; thus, we may summarise the
preceding discussion as follows:
\begin{Prop}
  \label{prop:6}
  If $\oc $ is a coalgebra modality on a symmetric monoidal category with finite
  products, then it is a monoidal coalgebra modality if and
  only if it is a storage modality, i.e., if and only if the maps~\eqref{eq:11} are invertible.
\end{Prop}

The second reformulation of the notion of monoidal coalgebra modality
concerns the following situation. Let $k$ be a commutative \emph{rig}
(= ring without negatives), and suppose that our base symmetric
monoidal category $(\C, \otimes, I)$ is \emph{$k$-linear}, meaning
that each hom-set is equipped with $k$-module structure such that both
composition and tensor product of morphisms is $k$-linear in each
variable. 

Again, we motivate things by considering the extra structure borne by a monoidal coalgebra modality $\oc$ on such a
$\C$. To begin with, observe that for each $X \in \C$ and
$(Y, \gamma) \in \cat{Coalg}(\oc)$, the $k$-module structure on the
left-hand side of the following adjointness
isomorphism
\begin{equation}
  \label{eq:53}
  \C(Y, X) \cong \cat{Coalg}(\oc )\bigl((Y, \gamma), (\oc X, \delta_X) \bigr)
\end{equation}
can be transported to yield a $k$-module structure on the right-hand
side. Since this happens naturally in $(Y, \gamma)$, it can equally be
expressed as a lifting of the hom-functor
\begin{equation*}
  \cat{Coalg}(\oc )\bigl(\,\thg, \,(\oc X, \delta_X)\,\bigr) \colon
  \cat{Coalg}(\oc )^\mathrm{op} \rightarrow \cat{Set}
\end{equation*}
through the category of $k$-modules; which, by the Yoneda lemma, is
the same as equipping $(\oc X, \delta_X)$ with the structure of an
internal $k$-module in $\cat{Coalg}(\oc )$ with respect to the
cartesian product $\hat \otimes$.

Since an internal $k$-module is a fortiori an internal commutative
monoid, we conclude that $\oc X$ is not just a cocommutative
coalgebra, but a bicommutative \emph{bialgebra} in
$\C$. By the Yoneda lemma, the additional structure maps
$\mathsf{u}_X \colon I \rightarrow \oc X $ and
$\nabla_X \colon \oc X \otimes \oc X \rightarrow \oc X$ involved can
be found as the zero map in the transported $k$-module structure on
$\cat{Coalg}(\oc )(\hat I, (\oc X, \delta_X) )$, and as the sum of the
projections $\mathsf{e}_X \otimes 1$ and $1 \otimes \mathsf{e}_X$ in
the transported $k$-module structure on
$\cat{Coalg}(\oc )\bigl((\oc X, \delta_X) \hat \otimes (\oc X,
\delta_X), (\oc X, \delta_X)\bigr)$. Working this out, we see that
they are the unique $\oc$-coalgebra maps
\begin{equation}
  \label{eq:51}
  \hat I \rightarrow (\oc X, \delta_X) \ \ \text{ and
  } \ \ (\oc X, \delta_X) \mathbin{\hat \otimes} (\oc X, \delta_X)
  \rightarrow (\oc X, \delta_X)
\end{equation}
which render commutative the diagram
\begin{equation}
  \label{eq:16}
  \cd[@!C]{
    \sh{l}{0.5em}\oc X \otimes \oc X \ar[r]^-{\nabla_X} \ar[dr]_-{\mathsf{e}_X \otimes \varepsilon_X + \varepsilon_X \otimes \mathsf{e}_X\ \ } & \oc X \ar[d]|-{\varepsilon_X} & I\rlap{ .} \ar[dl]^-{0} \ar[l]_-{\mathsf{u}_X} \\ & X
  }
\end{equation}

Note that these characterising conditions involve the lifted monoidal
structure $\hat \otimes$ on $\oc$-coalgebras, and hence the monoidal
comonad structure of $\oc$. However, it is possible to give
alternative characterisations of $\mathsf{u}_X$ and $\nabla _X$ which
do not mention this. The following result is a consequence
of~\cite[Theorem~1]{Blute2019Differential}, and we omit the proof,
since we will not need it.
\begin{Lemma}
  \label{lem:20}
  If $\oc$ is a monoidal coalgebra modality on a $k$-linear $\C$, then
  the maps $\mathsf{u}_X$ and $\nabla_X$ of~\eqref{eq:51} are the
  unique maps which render commutative the diagrams in~\eqref{eq:16}
  and the ``convolution'' diagrams:
  \begin{equation}
    \label{eq:52}
    \cd{
      {\oc Y} \ar[r]^-{\oc 0} \ar[d]_{\mathsf{e}_Y} &
      {\oc X} \ar@{<-}[d]^{\mathsf{u}_X} & &
      {\oc Y} \ar[r]^-{\oc(f+g)} \ar[d]_{\Delta_Y} &
      {\oc X} \ar@{<-}[d]^{\nabla_X} \\
      {I} \ar@{=}[r] &
      {I} &&
      {\oc Y \otimes \oc Y} \ar[r]^-{\oc f \otimes \oc g} &
      {\oc X \otimes \oc X} 
    }
  \end{equation}
\end{Lemma}

In particular, this result implies that if $\oc$ is a monoidal coalgebra
modality on a $k$-linear symmetric monoidal category, then its
underlying coalgebra modality supports a unique structure of
\emph{additive bialgebra modality}:

\begin{Defn}[Additive bialgebra modality]
  \label{def:9}
  An \emph{additive bialgebra modality} on a $k$-linear symmetric
  monoidal category $(\C, \otimes, I)$ is a coalgebra modality $\oc$
  endowed with further maps
  \begin{equation}\label{eq:10}
    \mathsf{u}_X \colon I \rightarrow \oc X \qquad \text{and}
    \qquad \nabla_X \colon \oc X \otimes \oc X \rightarrow \oc X
  \end{equation}
  natural in $X$, such that each
  $(\oc X, \mathsf{e}_X, \Delta_X, \mathsf{u}_X, \nabla_X)$ is a
  commutative and cocommutative bialgebra, and such that the diagrams
  in~\eqref{eq:16} and~\eqref{eq:52}  commute.
\end{Defn}
In fact, additive bialgebra modalities are the same as monoidal
coalgebra modalities; this is the full form
of~\cite[Theorem~1]{Blute2019Differential}:
\begin{Prop}
  \label{prop:10}
  If $\oc$ is a coalgebra modality on the $k$-linear symmetric monoidal
   $\C$, then there is at most one additive bialgebra
  modality structure on it---which exists precisely when
  $\oc$ underlies a monoidal coalgebra modality.
\end{Prop}
The general proof is quite involved; however, when $\C$ has finite
products, it is much more straightforward. Indeed, in this context, it
suffices by Proposition~\ref{prop:6} to show that if $\oc$ is an
additive bialgebra modality, then its underlying coalgebra modality
has invertible storage maps~\eqref{eq:11}. Note that, due to
$k$-linearity, the finite products of $\C$ are in fact finite
\textit{bi}products, which in the usual way we notate using $\oplus$
and $0$. It is now not hard to show that the storage maps of an
additive bialgebra modality are invertible, with inverses given by the
``costorage'' maps:
  \begin{equation}
    \label{eq:17}
    \begin{gathered}
      \chi_1^{\ast} \defeq I \xrightarrow{\mathsf{u}_0} \oc 0 \qquad \text{and} \\[0.4em]
      \smash{\chi_{XY}^\ast \defeq \oc X \otimes \oc Y \xrightarrow{!\iota_0 \otimes !\iota_1} \oc(X \oplus Y) \otimes \oc(X \oplus Y) \xrightarrow{\nabla_{X
            \times Y}} \oc(X \oplus Y)\rlap{ .}}
    \end{gathered}
  \end{equation}

\subsection{Differential modalities}
\label{sec:diff-modal-1}
We now extend the notions of the preceding section to deal with
differential structure; in the monoidal case, this allows us to get
from classical linear logic to the \emph{differential linear logic}
of~\cite{Ehrhard2018An-introduction}.

The following definition was first given
in~\cite{Blute2006Differential}, but was stated there without the
interchange rule; the necessity of this rule for the theory was first
recognised in~\cite[Lemma 3.2.3]{Blute2009Cartesian}, and it has since
become an accepted part of the definition.

\begin{Defn}[Differential modality, monoidal differential modality]
  \label{def:19} A \emph{differential modality} $\oc $ on a $k$-linear
  symmetric
  monoidal category $(\C, \otimes, I)$ is a coalgebra
  modality $\oc $ endowed with a \emph{deriving transformation}: a family
  of maps
  \begin{equation*}
    \mathsf{d}_X \colon \oc X \otimes X \rightarrow \oc X\rlap{ ,}
  \end{equation*}
  natural in $X \in \C$, subject to the following five axioms:
  \begin{itemize}
  \item The \emph{constant rule} $\mathsf{e}_X \circ \mathsf{d}_X = 0$;\vskip0.25\baselineskip
  \item The \emph{product rule} expressed by commutativity of the
    diagram:
    \begin{equation*}
      \cd[@!C@C-2em]{ \oc X \otimes X \ar[d]_-{\mathsf{d}}
      \ar[r]^-{\Delta \otimes 1} & \oc X \otimes \oc X \otimes X
      \ar[d]^-{(1 \otimes \mathsf{d}) + (\mathsf{d} \otimes 1)(1
        \otimes \sigma)}  \\
      \oc X \ar[r]^-{\Delta} & \oc X \otimes \oc X\rlap{ ;}}
    \end{equation*}
  \item The \emph{linear rule} expressed by commutativity of the
    diagram:
    \begin{equation*}
      \cd[@!C@C-2em]{
      {\oc X \otimes X} \ar[rr]^-{\mathsf{d}} \ar[dr]_-{e \otimes 1} &
      &
      {\oc X} \ar[dl]^-{\varepsilon}\rlap{ ;}\\
      & {X}
    }
    \end{equation*}
  \item The \emph{chain rule} expressed by commutativity of the
    diagram:
    \begin{equation*}
          \cd[@C-0.88em]{ \oc X \otimes X \ar[d]_-{\Delta \otimes 1}
      \ar[rr]^-{\mathsf{d}} & & \oc X \ar[d]^-{\delta} \\
      \oc X \otimes \oc X \otimes X \ar[r]^-{\delta \otimes \mathsf{d}}
      & \oc\oc X \otimes \oc X \ar[r]^-{\mathsf{d}} & \oc\oc X\rlap{ ;}} 
    \end{equation*}
  \item The \emph{interchange rule} expressed by commutativity of the diagram:
    \begin{equation*}
          \cd[@C-0.88em]{ \oc X \otimes X \otimes X \ar[r]^-{1 \otimes
        \sigma} \ar[d]_-{\mathsf{d} \otimes 1} & \oc X \otimes X
      \otimes X \ar[r]^-{\mathsf{d} \otimes 1} & \oc X
      \otimes X \ar[d]^-{\mathsf{d}} \\
      \oc X \otimes X \ar[rr]^-{\mathsf{d}} && \oc X\rlap{ .} }
    \end{equation*}
  \end{itemize}
  A \emph{monoidal differential modality} is a differential modality
  whose underlying coalgebra modality is
  monoidal.
\end{Defn}

Note that we could drop the constant rule from the above list of
axioms, since it is derivable from the other ones as
in~\cite[Lemma~4.2]{Blute2019Differential}; however, it is convenient
to include it anyway, since we will need to refer to it later on.

\begin{Ex}[Cofree comonoids as a monoidal differential modality]
  \label{ex:6}
  By an argument due to~\cite{Blute2016Derivations}, if $\C$ is
  $k$-linear symmetric monoidal with biproducts, then the cofree
  coalgebra modality $\oc^C$, when it exists, underlies a monoidal
  differential modality. For this, note that if
  $(C, \Delta, \mathsf{e})$ is a cocommutative comonoid in $\C$, and
  $\gamma \colon H \rightarrow C \otimes H$ is a comodule, then the
  object $C \oplus H$ has a cocommutative comonoid structure given by
  \begin{equation*}
    C \!\oplus\! H \xrightarrow{
      \biggl(\begin{smallmatrix}
        \Delta & 0 \\ 0 & \gamma \\
        0 & \sigma \gamma \\
        0 & 0
      \end{smallmatrix}\biggr)
    } (C \otimes C) \!\oplus\! (C \otimes H) \!\oplus\! (H \otimes C) \!\oplus\! (H \otimes H) \cong (C \!\oplus\! H)^{\otimes 2}\rlap{ .}
  \end{equation*}
  In particular, for any $X$ we have the cofree cocommutative comonoid
  $\oc^C X$, and cofree $\oc^C X$-comodule $\oc^C X \otimes X$; thus, we
  obtain a cocommutative comonoid structure on
  $\oc^C X \oplus (\oc^C X \otimes X)$. We now consider the map
  \begin{equation*}
    \varphi \defeq (\varepsilon_X, \mathsf{e}_X \otimes X) \colon \oc^C X \oplus (\oc^C X \otimes X) \rightarrow X
  \end{equation*}
  and let $\varphi^\sharp \colon \oc^C X \oplus (\oc^C X \otimes X)
  \rightarrow \oc^C X$ be the unique comonoid homomorphism that lifts
  $\varphi$ through $\varepsilon_X$. The first component of
  $\varphi^\sharp$ is easily seen to be the identity; the second
  component is the desired deriving transformation $\mathsf{d}_X
  \colon \oc^C X \otimes X \rightarrow \oc^C X$.
\end{Ex}

As noted in the introduction, monoidal differential modalities are
often axiomatised in the linear logic literature as monoidal coalgebra
modalities equipped with a \emph{codereliction}; this is a natural
transformation with components
\begin{equation*} \eta_X \colon X \to \oc X
\end{equation*}
satisfying the following rules; the first two are analogues of the linear rule
and the chain rule, while the third is the so-called ``monoidal rule''\footnote{Also sometimes called the ``strengthen rule''.}
(c.f.~\cite{Blute2019Differential,Blute2006Differential,
  Ehrhard2018An-introduction, Fiore2007Differential}); note that these
rules involve the bialgebra structure maps of $\oc X$ found as
in~\eqref{eq:16}.
\begin{gather*}
  \cd[@!C@C-2em]{
    X \ar@{=}[dr] \ar[rr]^-{\eta} & & \oc X
    \ar[dl]^-{\varepsilon} \\ & X
  } \qquad \cd[@C-0.35em]{
    X
    \ar[rr]^-{\eta} \ar[d]_-{\mathsf{u} \otimes \eta} &&
    \oc X \ar[d]^-{\delta} \\
    \oc X \otimes\oc X \ar[r]^-{\delta \otimes \eta} & \oc\oc X
    \otimes \oc\oc X \ar[r]^-{\nabla} & \oc\oc X
  } \qquad \cd[@C-0.35em]{
    X \otimes \oc Y
    \ar[r]^-{\eta \otimes 1} \ar[d]_-{1 \otimes \varepsilon} &
    \oc X \otimes \oc Y \ar[d]^-{m} \\
    X \otimes Y \ar[r]^-{\eta} & \oc(X \otimes Y)\rlap{ .}
  }
\end{gather*}

There is also an analogue of the constant rule, which we omit to write
since, like before, it turns out to be a consequence of the other
axioms; see~\cite[Lemma~6]{Blute2019Differential}. As shown
in~\cite[Theorem 4]{Blute2019Differential}, the presentations of
monoidal differential modalities in terms of coderelictions and
deriving transformations are equivalent, where the deriving
transformation corresponding to a codereliction
$\eta\colon X \to \oc X$ is given by:
\begin{equation} \oc X \otimes X \xrightarrow{1 \otimes \eta} \oc X
  \label{eq:57}
  \otimes \oc X \xrightarrow{\nabla} \oc X\rlap{ ;}
\end{equation}
and the codereliction associated to a deriving transformation
$\mathsf{d}\colon \oc X \otimes X \to \oc X$ is:
\begin{equation}
  \label{eq:56}
  X \xrightarrow{\mathsf{u} \otimes 1} \oc X \otimes X \xrightarrow{\mathsf{d}} \oc X\rlap{ .}
\end{equation}

There is an analogue of the ``monoidal rule'' for the presentation of
monoidal differential modalities via deriving transformations, but
this, along with a further rule known as the ``$\nabla$ rule'' turn
out to be derivable from the axioms listed above;
see~\cite{Blute2019Differential} for a comprehensive treatment.

In this paper we will mostly work with deriving transformations and
(monoidal) coalgebra modalities, because these make the calculations
most straightforward; however, we will be sure to work out the
remaining structure maps in examples.

\begin{Ex}[Codereliction for cofree comonoids]
  \label{ex:4}
  We saw in Example~\ref{ex:6} that, if $\C$ is $k$-linear symmetric
  monoidal with biproducts, then the cofree coalgebra modality $\oc^C$
  is always a monoidal differential modality. In this case, the
  codereliction $\eta_X \colon X \rightarrow \oc^C X$ is obtained from
  the $\mathsf{d}_X$ of Example~\ref{ex:6} via the
  formula~\eqref{eq:56}, but in fact we can do better than this.

  Indeed, in Example~\ref{ex:6}, we used the cofree comonoid $\oc^C X$
  and the cofree comodule $\oc^C X \otimes X$ over it to produce a
  cocommutative comonoid structure on
  $\oc^C X \oplus (\oc^C X \otimes X)$. But we can also consider the
  trivial comonoid $\tilde I$ and its unique coaction on $X$, yielding
  a cocommutative comonoid structure on $I \oplus X$. Since
  $\mathsf{u}_X \colon I \rightarrow \oc^C X$ is a map of comonoids,
  it follows easily that the left-hand map in
  \begin{equation*}
    I \oplus X \xrightarrow{\mathsf{u}_X \oplus (\mathsf{u}_X \otimes X)} \oc^C X \oplus (\oc^C X \otimes X) \xrightarrow{(\mathrm{id}, \mathsf{d}_X)} \oc^C X
  \end{equation*}
  is a map of comonoids. But the right-hand map is, by
  Example~\ref{ex:6}, the unique map of cocommutative comonoids which
  yields
  $(\varepsilon_X, \mathsf{e}_X \otimes X) \colon \oc^C X \oplus
  (\oc^C X \otimes X) \rightarrow X$ on postcomposition with
  $\varepsilon_X$; whence the composite above, which is
  $(\mathsf{u}_X, \eta_X) \colon I \oplus X \rightarrow \oc^C X$, is
  the unique map of comonoids which postcomposes with $\varepsilon_X$
  to yield $\pi_1 \colon I \oplus X \rightarrow X$. Thus,
  $\eta_X \colon X \rightarrow \oc^C X$ is given by the second
  component of the unique map of comonoids $I\oplus X \to \oc^C X$
  which lifts $\pi_1 \colon I \oplus X \rightarrow X$ through
  $\varepsilon_X$. See also~\cite[Theorem~21]{Lemay2021Coderelictions}
  for a direct verification of this fact.
\end{Ex}

To conclude this section, we describe the appropriate notion of morphisms
between differential modalities. 

\begin{Defn}[Morphisms of differential modalities]
  \label{def:4}
  Let $(\C, \otimes, I)$ be a $k$-linear symmetric monoidal category. A
  (monoidal) coalgebra modality morphism
  $\varphi \colon \oc^1 \rightarrow \oc^2$ between (monoidal)
  differential modalities on $\C$ is a (monoidal) \emph{differential
    modality morphism} if it renders commutative diagrams of the form:
  \begin{equation}
    \label{eq:18}
    \cd[@!C@-0.3em]{
      {\oc ^1} X \otimes X \ar[d]_-{\mathsf{d}^1} \ar[r]^-{\varphi \otimes 1} &
      {\oc ^2} X \otimes X \ar[d]^-{\mathsf{d}^2} \\
      {\oc ^1} X \ar[r]^-{\varphi} & {\oc ^2} X\rlap{ .}
    }
  \end{equation}
  We write $\cat{DiffMod}(\C)$ and  $\cat{MonDiffMod}(\C)$
  for the categories of differential modalities and monoidal differential
  modalities on $\C$ respectively.
\end{Defn}

\begin{Rk}
  \label{rk:6}
Once again, for monoidal differential modalities, we could recast this definition in terms of the
coderelictions, replacing the commutativity of~\eqref{eq:18} with that
of triangles of the form:
\begin{equation*}
  \cd[@-0.7em]{
    & {X} \ar[dl]_-{\eta^1} \ar[dr]^-{\eta^2} \\
    {{\oc ^1} X} \ar[rr]^-{\varphi} & &
    {{\oc ^2} X}\rlap{ .}
  }
\end{equation*}
\end{Rk}

In light of Definition~\ref{def:4}, for a $k$-linear symmetric monoidal category $\C$ we can
extend~\eqref{eq:54} to a diagram 
\begin{equation*}
  \cd[@-1em]{
    & \cat{MonDiffMod}(\C) \ar[dl]_-{} \ar[d]_-{}\\
    \cat{DiffMod}(\C) \ar[d]_-{} & \cat{MonCoalgMod}(\C) \ar[dl]_-{} \ar[dr]^-{} \\
    \cat{CoalgMod}(\C) & & \cat{MonComonad}(\C)
  }
\end{equation*}
of forgetful functors, wherein each of the diagonal arrows is full and
faithful. Our objective in the rest of the paper is to show that,
under suitable assumptions on $\C$, the two vertical arrows admit left
adjoints.

\section{Algebraically-free commutative monoids}
\label{sec:algebr-free-comm}

In this section, we introduce \emph{algebraically-free commutative
  monoids}, which are the key piece of structure needed for allow the
construction of the free (monoidal) differential modality on a
(monoidal) coalgebra modality. Algebraically-free commutative monoids
are, in particular, free commutative monoids; the following
establishes our preferred notation for these.

\begin{Defn}[Free commutative monoids]
  \label{def:7}
  \looseness=-1
  A symmetric monoidal category $(\C, \otimes, I)$ \emph{is endowed
    with free commutative monoids}, if for all $X \in \C$, there is
  given a commutative monoid
  $(SX, \mathsf{u}_X \colon I \rightarrow SX, \nabla_X \colon SX \otimes
  SX \rightarrow SX)$ and map $\eta_X \colon X \rightarrow SX$ 
  exhibiting $(SX, \mathsf{u}_X, \nabla_X)$ as the free commutative
  monoid on $X$.
\end{Defn}

However, as discussed in the introduction, the universal property of a
free commutative monoid is insufficient for our purposes; and for the
appropriate strengthening, we look to~\cite{Kelly1980A-unified} for
inspiration. Section 23 of \emph{op.~cit.~}draws the distinction
between free and \emph{algebraically-free} monoids in a
monoidal category: a free monoid $FX$ on an object $X$ being called
\emph{algebraically-free} if monoid actions by $FX$ correspond to
actions by the mere object~$X$. We now extend this distinction to
the commutative~case.

\subsection{The definition}
\label{sec:defin-relat-free}
A free commutative monoid $SX$ will be algebraically-free in our sense
when actions by the monoid $SX$ correspond to \emph{commuting actions}
by the mere object $X$, in the following sense:

\begin{Defn}[Commuting $X$-actions]
  \label{def:1}
  Given $X \in \C$, a commuting $X$-action on an object $A \in \C$ is
  a map $\alpha \colon A \otimes X \rightarrow A$ making the square
  \begin{equation}
    \label{eq:7}
    \cd[@-0.3em]{
      A \otimes X \otimes X \ar[d]^-{A \otimes \sigma} \ar[r]^-{A \otimes \alpha} &
      A \otimes X \ar[r]^-{\alpha} &
      A \ar@{=}[d]_-{} \\
      A \otimes X \otimes X \ar[r]^-{A \otimes \alpha} &
      A \otimes X \ar[r]^-{\alpha} &
      A \rlap{ .}
    }
  \end{equation}
  commute. A map of commuting $X$-actions $(A, \alpha) \rightarrow (B, \beta)$
  is a map $f \colon A \rightarrow B$ making the square
  \begin{equation*}
    \cd{
      A \otimes X \ar[r]^-{\alpha} \ar[d]_-{A \otimes f} & A \ar[d]^-{f} \\
      B \otimes X \ar[r]^-{\beta} & B
    }
  \end{equation*}
  commute. We write $\ca$ for the
  category of commuting $X$-actions and their maps, and $U \colon \ca \rightarrow \C$
  for the forgetful functor.
\end{Defn}

Now, if $M$ is a commutative monoid in $\C$, and
$f \colon X \rightarrow M$ is any map, then each right $M$-module
structure $\rho \colon A \otimes M \rightarrow A$ on an object
$A \in \C$ gives rise to a commuting $X$-action on $A$ via
\begin{equation*}
  A \otimes X \xrightarrow{1 \otimes f} A \otimes M \xrightarrow{\rho} A\rlap{ ,}
\end{equation*}
and this induces a functor
$f^\ast \colon \mod[M] \rightarrow \ca$ from the category of
right $M$-modules to the category of commuting $X$-actions. Using
this, we can now give our notion of ``algebraically-free commutative
monoid''.

\begin{Defn}[Algebraically-free commutative monoids]
  \label{def:6}
  A symmetric monoidal category $(\C, \otimes, I)$ \emph{is endowed
    with algebraically-free commutative monoids}, if it is endowed
  with free commutative monoids and, for all $X \in \C$, the functor
  $\eta_X^\ast \colon \mod \rightarrow \ca$ is an
  isomorphism of categories.
\end{Defn}

Let us now unpack what it means for $(\C, \otimes, I)$ to be endowed
with algebraically-free commutative monoids. For any $A \in \C$, the
free $SX$-module on $A$ is given by
$(A \otimes SX, A \otimes \nabla_X)$, with freeness witnessed by the
map $A \otimes \mathsf{u}_X \colon A \rightarrow A \otimes SX$; and
since, by our assumption of algebraic-freeness, the functor
$\eta_X^\ast \colon \mod \rightarrow \ca$ is an isomorphism, it must
send this free $SX$-module to the free commuting $X$-action on
$A \in \C$. Thus, writing
\begin{equation}
  \label{eq:3}
  \mathsf{d}_X \defeq SX \otimes X \xrightarrow{SX \otimes \eta_X} SX \otimes SX \xrightarrow{\nabla_X} SX\rlap{ ,}
\end{equation}
we find that the free commuting $X$-action on $A \in \C$ is given by
$(A \otimes SX, A \otimes \mathsf{d}_X)$, with freeness witnessed by
the map $A \otimes \mathsf{u}_X \colon A \rightarrow A \otimes SX$.
More concretely, this freeness asserts that:
{\setlength{\abovedisplayskip}{12pt}
  \setlength{\belowdisplayskip}{12pt}
\begin{equation}
  \label{eq:2}
  \begin{gathered}
    \begin{minipage}{0.9\linewidth}
      for any commuting $X$-action $(B, \beta)$ and any map
      $f \colon A \rightarrow B$, there is a unique map of
      commuting $X$-actions
      $\bar f \colon (A \otimes SX, A \otimes \mathsf{d}_X)
      \rightarrow (B, \beta)$ which restricts back to $f$ along
      $A \otimes \mathsf{u}_X$---or in other words, renders commutative the diagram 
    \end{minipage}\\
    \cd[@!C@C-2em@-0.5em]{
      & A \ar[dl]_-{A \otimes \mathsf{u}_X} \ar[dr]^-{f} \\
      A \otimes SX \ar@{-->}[rr]^-{\exists!\, \bar f} && B \\
      A \otimes SX \otimes X \ar[u]^-{A \otimes \mathsf{d}_X}
      \ar[rr]^-{\bar f \otimes X} && B \otimes X \rlap{ .}
      \ar[u]_-{\beta} }
  \end{gathered}
\end{equation}
}

\subsection{Algebraically-free implies free}
\label{sec:algebr-free-impl}
The above consequence of the existence of algebraically-free
commutative monoids is what we will use in practice; but in fact, it
is not just a consequence, but also a characterisation, as the
following result shows.

\begin{Prop}
  \label{prop:3}
  Let $(\C, \otimes, I)$ be a symmetric monoidal category. If each $X \in \C$ comes equipped
  with a commuting $X$-action $(SX, \mathsf{d}_X)$ and a map
  $\mathsf{u}_X \colon I \rightarrow SX$ satisfying the universal
  property~\eqref{eq:2}, then $\C$ can be endowed with
  algebraically-free commutative monoids. More specifically, let us define for
  each $X \in \C$ the map
\begin{equation}
    \label{eq:19}
    \eta_X \defeq X \xrightarrow{\mathsf{u}_X \otimes X} SX \otimes X \xrightarrow{\mathsf{d}_X} SX\rlap{ ,}
  \end{equation}
  and define $\nabla_X \colon SX \otimes SX \rightarrow SX$ to be the
  unique map of free commuting $X$-actions which precomposes with
  $SX \otimes \mathsf{u}_X$ to give the identity, i.e., the unique map
  rendering commutative
  \begin{equation}
    \label{eq:21}
    \cd[@!C@C-2em]{
      & SX \ar[dl]_-{SX \otimes \mathsf{u}_X} \ar[dr]^-{1_{SX}} \\
      SX \otimes SX \ar@{-->}[rr]^-{\nabla_X} && SX \\
      SX \otimes SX \otimes X \ar[u]^-{SX \otimes \mathsf{d}_X} \ar[rr]^-{\nabla_X \otimes X} && SX \otimes X \rlap{ .} \ar[u]_-{\mathsf{d}_X}  
    }
  \end{equation}
  Then for each $X \in \C$, the map $\eta_X$ exhibits
  $(SX, \mathsf{u}_X, \nabla_X)$ as the algebraically-free commutative
  monoid on $X$.
\end{Prop}

This result adapts~\cite[Theorem~23.1]{Kelly1980A-unified} to the
commutative case.

\begin{proof}
  The assumptions of the proposition amount to the provision of a left
  adjoint to the forgetful functor $U \colon \ca \rightarrow \C$.
  Consider the monad $\mathsf{T}$ induced by this adjunction. Its
  underlying endofunctor is $(\thg) \otimes SX$; its unit is
  $(\thg) \otimes \mathsf{u}_X$; and its multiplication has
  $A$-component given by the unique map of free commuting $X$-actions
  $A \otimes SX \otimes SX \rightarrow A \otimes SX$ which precomposes
  with $A \otimes SX \otimes \mathsf{u}_X$ to give the identity. The
  map $A \otimes \nabla_X$ (with $\nabla_X$ as in~\eqref{eq:21}) is
  easily seen to be such a map, so that $\mathsf{T}$ has
  multiplication given by $(\thg) \otimes \nabla_X$.

  From this form of $\mathsf{T}$ it follows that
  $(SX, \mathsf{u}_X, \nabla_X)$ is a monoid, and
  $\mathsf{T}\text-\cat{Alg}$ is the category of right $SX$-modules.
  We claim that this monoid is commutative, and the algebraically-free
  commutative monoid on $X$ when equipped with the unit
  map~\eqref{eq:19}.
  
  We first show that $(SX, \mathsf{u}_X, \nabla_X)$ is commutative. To
  this end, note that $\nabla_X$ and
  $\nabla_X \circ \sigma \colon SX \otimes SX \rightarrow SX$ both
  precompose with $SX \otimes \mathsf{u}_X$ to give the identity;
  thus, if both are maps between free commuting $X$-actions then by
  freeness they coincide. Now, $\nabla_X$ is a map of 
  $X$-actions by definition; as for $\nabla_X \circ \sigma$, 
  consider the diagram:
  \begin{equation*}
    \cd{
      SX \otimes X \otimes SX \ar[d]_-{\mathsf{d}_X \otimes 1} \ar[r]^-{\sigma \otimes 1} & X \otimes SX \otimes SX \ar[r]^-{1 \otimes \nabla_X} & X \otimes SX \ar[d]^-{\mathsf{d}_X \circ \sigma} \\
      SX \otimes SX \ar[rr]^-{\nabla_X} & &
      SX\rlap{ .}
    }
  \end{equation*}
  The top, left, and bottom edges therein are maps of free commuting
  $X$-actions; but so too is the right by virtue of the
  axiom~\eqref{eq:7} for the commuting action $(SX,\mathsf{d}_X)$.
  Since precomposing both paths by $SX \otimes X \otimes \mathsf{u}_X$
  yields the same map $\mathsf{d}_X$, we conclude by freeness that the
  diagram commutes. Precomposing by
  $\sigma_{SX, SX \otimes X} \colon SX \otimes SX \otimes X
  \rightarrow SX \otimes X \otimes SX$, we obtain precisely the
  diagram expressing that $\nabla_X \circ \sigma$ is a map of
  commuting $X$-actions
  $(SX \otimes SX, SX \otimes \mathsf{d}_X) \rightarrow (SX,
  \mathsf{d}_X)$ as required.

  It remains to show that the map $\eta_X$ in~\eqref{eq:19} exhibits
  $(SX, \nabla_X, \mathsf{u}_X)$ as algebraically-free on $X$. This
  means showing that it is free in the usual sense on $X$, and that
  the functor $\eta_X^\ast \colon \mod \rightarrow \ca$ is
  an isomorphism; but in fact, by the argument in the last paragraph
  of the proof of~\cite[Theorem~23.1]{Kelly1980A-unified}, if we can
  show the latter of these two facts, it will imply the former.

  To this end, note that the forgetful functor
  $U \colon \ca \rightarrow \C$ easily satisfies the Beck conditions
  to be strictly monadic, so that the comparison functor
  $K \colon \ca \rightarrow \mod$ into the category of
  $\mathsf{T}$-algebras is invertible; we claim it is inverse to
  $\eta_X^\ast$. Indeed, $K$ sends $(A, \alpha) \in \ca$ to
  $(A, \beta) \in \mod$, where $\beta$ is the unique map of commuting
  $X$-actions
  $(A \otimes SX, A \otimes \mathsf{d}_X) \rightarrow (A, \alpha)$
  that precomposes with $A \otimes \mathsf{u}_X$ to give the identity.
  So $\eta_X^\ast(K(A, \alpha)) = (A, \bar \alpha)$ where
  \begin{equation*}
    \bar \alpha = \beta \circ (A \otimes \eta_X) = \beta \circ (A \otimes \mathsf{d}_X) \circ (A \otimes \mathsf{u}_X \otimes X) = \alpha \circ (\beta \otimes X) \circ (A \otimes \mathsf{\mathsf{u}_X} \otimes X) = \alpha\text{ .}
  \end{equation*}
  So the invertible $K$ is a one-sided inverse, whence an inverse, for
  $\eta_X^\ast$.
\end{proof}

\begin{Rk}\label{rk:4}
  If all algebraically-free commutative monoids exist, then there is a
  unique way of making the assignment $X \mapsto SX$ functorial
  such that the maps $\mathsf{u}_X$ and $\mathsf{d}_X$ (and hence also
  $\eta_X$ and $\nabla_X$) become natural in $X$. Indeed, given
  $f \colon X \rightarrow Y$, we can restrict \emph{any} commuting $Y$-action
  $(B, \beta)$ along $f$ to a commuting $X$-action
  \begin{equation*}
    f^\ast(B, \beta) \defeq \bigl(B, \beta \circ (X \otimes f)\bigr)\rlap{ .}
  \end{equation*}
  In particular, we can do this for the
  $Y$-action $(SY, \mathsf{d}_Y)$, and so can define the map
  $Sf \colon SX \rightarrow SY$ as the unique map of commuting
  $X$-actions
  \begin{equation*}
    Sf \colon (SX, \mathsf{d}_X) \rightarrow f^\ast(SY, \mathsf{d}_Y)
  \end{equation*}
  which precomposes with $\mathsf{u}_X \colon I \rightarrow SX$ to
  yield $\mathsf{u}_Y \colon I \rightarrow SY$. Explicitly, this means
  that the following diagram (which express precisely the naturality
  of $\mathsf{u}_X$ and $\mathsf{d}_X$) is commutative:
  \begin{equation}
    \label{eq:32}
    \cd[@!@C-1em@-2.5em@R-1em]{
      & I \ar[dl]_-{\mathsf{u}_X} \ar[dr]^-{\mathsf{u}_Y} \\
      SX \ar[rr]^-{Sf} & & SY \rlap{ .} \\
      SX \otimes X \ar[u]^-{\mathsf{d}_X} \ar[rr]^-{Sf \otimes f} & &
      SY \ar[u]_-{\mathsf{d}_Y}
    }
  \end{equation}
  
\end{Rk}

\subsection{The additive case}
\label{sec:additive-case}

We now specialise the above to the case of a \emph{$k$-linear}
symmetric monoidal category $\C$. If such a $\C$ is endowed with
algebraically-free commutative monoids, then $X \mapsto SX$ is, by the
dual of Example~\ref{ex:3}, the action on objects of an opmonoidal
algebra modality on $\C$. Thus, dualising the argument of
Section~\ref{sec:diff-modal-1}, each $SX$ is a bialgebra whose
coalgebra maps are the unique commutative monoid homomorphisms
$\mathsf{e}_X \colon (SX, \mathsf{u}_X, \nabla_X) \rightarrow (I, 1_I,
1_I)$ and
$\Delta_X \colon (SX, \mathsf{u}_X, \nabla_X) \rightarrow (SX,
\mathsf{u}_X, \nabla_X) \tilde \otimes (SX, \mathsf{u}_X, \nabla_X)$
that render commutative the following triangles:
\begin{equation}
  \label{eq:8}
  \cd[@!C@C+2em]{
    & X \ar[d]^-{\eta_{X}}  \ar[dl]_-{0} \ar[dr]^-{\ \ \ \ \ \eta_X \otimes \mathsf{u}_X + \mathsf{u}_X \otimes \eta_X} \\
    I & \ar@{-->}[l]^-{\exists!\,\mathsf{e}_X}
    SX  \ar@{-->}[r]_-{\exists!\,\Delta_X} & SX \otimes SX\rlap{ .}
  }
\end{equation}

This bialgebra structure allows us to lift the symmetric monoidal
structure of $\C$ to the category $\mod$ of right $SX$-modules. The
lifted unit object is $(I, \mathsf{e}_X)$, and the tensor product
$(A, \alpha) \boxtimes (B, \beta)$ of right $SX$-modules is given by
$A\otimes B$ with the action
\begin{equation*}
  A \otimes B \otimes SX \xrightarrow{A \otimes B \otimes \Delta_X} A \otimes B \otimes SX \otimes SX \xrightarrow{1 \otimes \sigma \otimes 1} A \otimes SX \otimes B \otimes SX \xrightarrow{\alpha \otimes \beta} A \otimes B\rlap{ .}
\end{equation*}

There is not much more we can say about this monoidal structure in
general; however, when free commutative monoids are
algebraically-free, it transfers along the isomorphism of categories
$\eta_X^\ast \colon \mod \rightarrow \ca$ to yield a
symmetric monoidal structure on $\ca$. From~\eqref{eq:8} and the
definition of the tensor product on $\mod$, we see that
this monoidal structure is given as follows:
\begin{Defn}[Lifted monoidal structure on commuting $X$-actions]
  \label{def:3}
  The \emph{lifted monoidal structure} on $\ca$ has:
  \begin{itemize}[itemsep=0.25\baselineskip]
  \item Unit object given by the commuting $X$-action
    $(I, 0 \colon X \rightarrow I)$;
  \item The tensor product $(A, \alpha) \boxtimes (B, \beta)$ of
    commuting $X$-actions given by the commuting $X$-action
    $(A \otimes B, \alpha \boxtimes \beta)$, where
    $\alpha \boxtimes \beta$ is the sum
    \begin{equation}
      \label{eq:20}
      \alpha \boxtimes \beta = {}
      \cd[@C-0.5em@R-1.8em]{
        A \otimes B \otimes X \ar[rr]^-{A \otimes \beta} & & A \otimes B \\
        & + \\
        A \otimes B \otimes X \ar[r]^-{A \otimes \sigma}_{\phantom{A \otimes \beta}} & A \otimes X \otimes B \ar[r]^-{\alpha \otimes B} & A \otimes B\rlap{ .}
      }
    \end{equation}
  \item Coherence constraints being those of $\C$.
  \end{itemize}
\end{Defn}

We will use this monoidal structure, and the definition~\eqref{eq:20},
extensively in what follows. We note in passing that, even when
algebraically-free commutative monoids do not exist, the preceding
definition still yields a symmetric monoidal structure on $\ca$, but
work is required to ensure that the definitions are well-posed.
However, in our case, this work can be omitted, since the structure was
obtained by transporting a known monoidal structure across an
isomorphism of categories.

\subsection{Existence}
\label{sec:existence}
We now consider the \emph{existence} of algebraically-free commutative
monoids in categories of practical interest. Of course, the existence
of free commutative monoids has been studied extensively, but these
need not always be algebraically-free, as the following example shows:

\begin{Ex}
  \label{ex:1}
  Consider $\cat{Vect}_\mathbb{C}^\mathrm{op}$, the opposite of the
  category of $\mathbb{C}$-vector spaces with its usual tensor
  product. In this case, it is more natural to work in
  $\cat{Vect}_\mathbb{C}$ and speak of cocommutative comonoids,
  commuting coactions, and so on. It is well known that cofree
  cocommutative comonoids (= cocommutative $\mathbb{C}$-coalgebras)
  exist; we show that these are not coalgebraically-cofree.

  Indeed, by~\cite{Sweedler1969Hopf, Murfet2015On-Sweedlers} (see also Section~\ref{sec:k-modules} below),
  the cofree cocommutative $\mathbb{C}$-coalgebra
  $\oc^C \mathbb{C}$ on $\mathbb{C}$ is given by the direct sum of
  coalgebras $\bigoplus_{c \in \mathbb{C}} \mathbb{C}[x]$ where the
  vector space of polynomials $\mathbb{C}[x]$ is equipped with its
  usual coalgebra structure, and where the couniversal map
  $\varepsilon \colon \bigoplus_{c \in \mathbb{C}} \mathbb{C}[x]
  \rightarrow \mathbb{C}$ takes degree-1 coefficients in each summand.

  We now show that the cofree $\oc^C \mathbb{C}$ is not
  coalgebraically-cofree.
  Indeed, note that a $\mathbb{C}$-coaction on a complex vector space
  $A$ is just a linear map
  $\alpha \colon A \rightarrow A \otimes \mathbb{C}$, and that the
  cocommutativity condition in this case always holds. Thus, via the
  isomorphism $A \otimes \mathbb{C} \cong A$, we can identify the
  category of cocommuting coactions of $\mathbb{C}$ with the category
  of vector spaces equipped with a linear
  endomorphism. 

  Now, if $\oc^C \mathbb{C}$ were coalgebraically-cofree, then by the
  argument preceding Proposition~\ref{prop:3}, the cofree cocommuting
  $\mathbb{C}$-coaction on $\mathbb{C}$ itself would be given by
  $\mathbb{C} \otimes \oc^C \mathbb{C} \cong \bigoplus_{c \in \mathbb{C}} \mathbb{C}[x]$
  endowed with the linear endomorphism
  \begin{equation*}
    \oc^C \mathbb{C} \xrightarrow{\Delta} \oc^C \mathbb{C} \otimes \oc^C \mathbb{C} \xrightarrow{\oc^C \mathbb{C} \otimes \varepsilon} \oc^C \mathbb{C} \otimes \mathbb{C} \cong \oc^C \mathbb{C}\rlap{ ,}
  \end{equation*}
  which, if we work it through, is the endomorphism $\partial$ of
  $ \bigoplus_{c \in \mathbb{C}} \mathbb C[x]$ which maps each direct summand
  to itself via the usual differential operator. But $(\bigoplus_{c \in \mathbb{C}} 
  \mathbb{C}[x], \partial)$
  \emph{cannot} be cofree on $\mathbb{C}$: indeed, if it were, then cofreeness
  would, in particular, imply a bijection between linear maps
  $\mathbb{C} \rightarrow \mathbb{C}$ and maps of cocommuting coactions
  $(\mathbb{C}, \mathrm{id}_\mathbb{C}) \rightarrow (\bigoplus_c \mathbb{C}[x], \partial)$; but since every element of
  $\bigoplus_{c \in \mathbb{C}} \mathbb{C}[x]$ is annihilated by some power of
  $\partial$, the only possible map
  $(\mathbb{C}, \mathrm{id}_\mathbb{C}) \rightarrow (\bigoplus_c \mathbb{C}[x], \partial)$ of cocommuting coactions is the zero
  map.
\end{Ex}

The question of when exactly free commutative monoids are
algebraically-free is a delicate one; however, for our purposes, it
will suffice to give a set of \emph{sufficient} conditions which cover
all examples of practical interest.

\begin{Defn}[Symmetric algebra construction]
  \label{def:2}
  Let $\mathbb{B}$ denote the category of finite cardinals and
  bijections; so objects of $\mathbb{B}$ are natural numbers, and the
  homset $\mathbb{B}(n,m)$ is empty if $n \neq m$ and is the symmetric
  group $\mathfrak{S}_n$ if $n = m$.

  For any symmetric monoidal category $(\C, \otimes, I)$ and any
  $X \in \C$, we have a functor
  $X^{\otimes (\thg)} \colon \mathbb{B} \rightarrow \C$ sending the
  object $n$ to the tensor power $X^{\otimes n}$, and the map
  $\sigma \in \mathfrak{S}_n$ to the map
  $X^{\otimes \sigma} \colon X^{\otimes n} \rightarrow X^{\otimes n}$
  that permutes the tensor factors according to $\sigma$. We say $\C$
  \emph{admits the symmetric algebra construction} if each
  $X^{\otimes(\thg)} \colon \mathbb{B} \rightarrow \C$ has a colimit
  which is preserved by the tensor product in each variable.
\end{Defn}

For example, the category of vector spaces over any fixed field admits the free symmetric
algebra construction, with the colimit of $X^{\otimes(\thg)}$ being
found as the direct sum of the symmetrised tensor powers
$X^{\otimes n} / \mathfrak{S}_n$. More generally, we have:

\begin{Lemma}
  \label{lem:2}
  Let $(\C, \otimes, I)$ be a symmetric monoidal category.
  \begin{enumerate}[(i),itemsep=0.25\baselineskip]
  \item If $\C$ is cocomplete and monoidal closed, then it admits the
    symmetric algebra construction.
  \item If $\C$ admits the symmetric algebra construction, and
    $\mathsf{T}$ is a symmetric monoidal monad on $\C$, then the
    Kleisli category $\cat{Kl}(\mathsf{T})$ admits the symmetric
    algebra construction under the induced monoidal
    structure.
  \end{enumerate}
\end{Lemma}
In part (ii), the ``induced monoidal structure'' on
$\cat{Kl}(\mathsf{T})$ is the one which on objects is given as in
$\C$; on morphisms, takes
$(f \colon A \rightarrow TB) \otimes (g \colon C \rightarrow TD)$ to
\begin{equation*}
  A \otimes C \xrightarrow{f \otimes g} TB \otimes TD \xrightarrow{\varphi_{B,D}} T(B \otimes D)
\end{equation*}
where $\varphi$ is the monoidality constraint of $\mathsf{T}$; and has
remaining data determined by the requirement that the left adjoint
identity-on-objects functor
$F_\mathsf{T} \colon \C \rightarrow \cat{Kl}(\mathsf{T})$ sending
$f \colon A \rightarrow B$ to $\eta_B \circ f \colon A \rightarrow TB$ (where $\eta$ is the unit of $\mathsf{T}$) should be \emph{strict} symmetric monoidal.
\begin{proof}
  Part (i) is clear. As for (ii), since
  $F_\mathsf{T} \colon \C \rightarrow \cat{Kl}(\mathsf{T})$ is
  identity-on-objects strict monoidal, we have for any $A, X \in \C$
  that the following triangle commutes:
  \begin{equation*}
    \cd[@!C@C-2em]{
      & {\mathbb{B}} \ar[dl]_-{A \otimes X^{\otimes(\thg)}} \ar[dr]^-{F_\mathsf{T}A \otimes (F_\mathsf{T}X)^{\otimes(\thg)}} \\
      {\C} \ar[rr]_-{F_\mathsf{T}} & &
      {\cat{Kl}(\mathsf{T})}\rlap{ .}
    }
  \end{equation*}
  Thus, if $\iota_n \colon X^{\otimes n} \rightarrow SX$ is a
  colimiting cocone for $X^{\otimes(\thg)}$ in $\C$, then the
  colimiting cocone $A \otimes \iota_n$ for
  $A \otimes X^{\otimes(\thg)}$ is sent by the colimit-preserving
  $F_\mathsf{T}$ to a colimiting cocone
  $F_\mathsf{T} A \otimes F_\mathsf{T} \iota_n$ for
  $F_\mathsf{T}A \otimes (F_\mathsf{T}X)^{\otimes (\thg)}$. Since
  $F_\mathsf{T}$ is bijective on objects, this verifies all the
  conditions required for $\cat{Kl}(\mathsf{T})$ to admit the
  symmetric algebra construction.
\end{proof}

In particular, part (i) of this proposition implies that the
categories of sets; of vector spaces over a field $k$; of modules over
any \emph{rig} $k$; of join-semilattices; and of complete
join-semilattices---all admit the symmetric algebra construction. Part
(ii) implies, morever, that the categories of sets and relations, or
more generally, the category of sets and ``$Q$-relations'' for $Q$ a
quantale will admit the symmetric algebra construction.

The point of the symmetric algebra construction is that it offers a
straightforward way to produce algebraically-free commutative monoids:

\begin{Prop}
  \label{prop:2}
  If $(\C, \otimes, I)$ admits the symmetric algebra construction,
  then $\C$ admits algebraically-free commutative monoids.
\end{Prop}
\begin{proof}
  For any $X \in \C$, let $SX$ be the colimit of
  $X^{\otimes(\thg)} \colon \mathbb{B} \rightarrow \C$ with
  coprojections $\iota_n \colon X^{\otimes n} \rightarrow SX$. As this
  colimit is preserved by tensor, the maps $\iota_n \otimes X$ exhibit
  $SX \otimes X$ as the colimit of the diagram
  $X^{\otimes (\thg)} \otimes X \colon \mathbb{B} \rightarrow \C$. The
  family of maps
  \begin{equation*}
    X^{\otimes n} \otimes X = X^{\otimes(n+1)} \xrightarrow{\iota_{n+1}} SX
  \end{equation*}
  is easily seen to constitute a cocone under this same diagram
  $X^{\otimes(\thg)} \otimes X$, and so we induce a unique map
  $\mathsf{d}_X \colon SX \otimes X \rightarrow SX$ such that
  $\mathsf{d}_X \circ (\iota_n \otimes X) = \iota_{n+1}$. Using the
  fact that $SX \otimes X \otimes X$ is a colimit of
  $X^{\otimes(\thg)} \otimes X \otimes X$, along with the fact that
  $\iota_{n+2}$ is fixed under permutation of the last two tensor
  factors, we easily conclude that $\mathsf{d}_X$ endows $SX$ with a
  commuting $X$-action.

  We now show that, for any $A \in \C$, the object
  $(A \otimes SX, A \otimes \mathsf{d}_X)$ is a free commuting
  $X$-action on $A$, as witnessed by
  $A \otimes \iota_0 \colon A \rightarrow A \otimes SX$. This will, by
  Proposition~\ref{prop:3}, complete the proof. So suppose
  $(B, \beta)$ is a commuting $X$-action and
  $f \colon A \rightarrow B$; we must exhibit a unique
  $\bar f \colon A \otimes SX \rightarrow B$ making~\eqref{eq:3}
  commute.

  Since the maps $A \otimes \iota_n$ exhibit $A \otimes SX$ as the
  colimit of
  $A \otimes X^{\otimes(\thg)} \colon \mathbb{B} \rightarrow \C$, such
  a map $\bar f$ is determined by the composites
  $p_n \defeq \bar f \circ (A \otimes \iota_n) \colon A \otimes
  X^{\otimes n} \rightarrow B$. For the top triangle in~\eqref{eq:2}
  to commute, we must have $p_0 = f \colon A \rightarrow B$. On the
  other hand, if the bottom square is to commute, then because
  \begin{equation}
    \label{eq:5}
    \begin{aligned}
      \bar f \circ (A \otimes \mathsf{d}_X) \circ (A \otimes \iota_n
      \otimes X) &= \bar f \circ (A \otimes \iota_{n+1}) = p_{n+1} \\
      \text{and} \ \ \beta \otimes (\bar f \otimes X) \otimes (A
      \otimes \iota_n \otimes X) &= \beta \circ (p_n \otimes X)
    \end{aligned}
  \end{equation}
  we must have that
  $p_{n+1} = \beta \circ (p_n \otimes X) \colon A \otimes X^{\otimes
    n} \otimes X \rightarrow B$ for all $n$. This proves that
  $\bar f \colon A \otimes SX \rightarrow B$ is unique if it exists.

  To show existence, we must show firstly that the family of maps
  $p_n \colon A \otimes X^{\otimes n} \rightarrow B$ defined
  recursively as above, i.e., $p_0 = f$ and
  $p_{n+1} = \beta \circ (p_n \otimes X)$, comprises a cocone under
  $A \otimes X^{\otimes(\thg)}$; and secondly, that the induced
  $\bar f$ renders~\eqref{eq:3} commutative. For the first of these,
  we must show that each $p_n$ is fixed on precomposition by each
  $A \otimes X^{\otimes \sigma}$. This is clear for $n = 0,1$. For
  $n \geqslant 2$, note that
  \begin{equation*}
    p_n = A \otimes X^{\otimes(n-2)} \otimes X \otimes X \xrightarrow{p_{n-2} \otimes X \otimes X} B \otimes X \otimes X \xrightarrow{\beta \otimes X} B \otimes X \xrightarrow{\beta} B\rlap{ ,}
  \end{equation*}
  which, since $\beta$ is a commuting $X$-action, is fixed under
  precomposition by $A \otimes X^{\otimes(n-1\,\, n)}$. Moreover, as
  $p_n = \beta \circ (p_{n-1} \otimes X)$, we see inductively that
  $p_n$ is fixed under precomposition by
  $A \otimes X^{\otimes(1\,\,2\,\,\dots \,\,n-1)}$. But
  $(1\,\,2\,\,\dots\,\,n-1)$ and $(n-1\,\,n)$ generate
  $\mathfrak{S}_n$; so $p_n$ is fixed by precomposition by any
  $A \otimes X^{\otimes \sigma}$.

  So there is a unique map $\bar f \colon A \otimes SX \rightarrow B$
  such that $\bar f \circ (A \otimes \iota_n) = p_n$ for all $n$, and
  it remains to show that this map renders commutative~\eqref{eq:3}.
  For the top triangle this is clear since
  $\bar f \circ (A \otimes \iota_0) = p_0 = f$. On the other hand, for
  any $n$ we have using the equalities in~\eqref{eq:5} that
  $\bar f \circ (A \otimes \mathsf{d}_X) \circ (A \otimes \iota_n
  \otimes X) = \beta \otimes (\bar f \otimes X) \otimes (A \otimes
  \iota_n \otimes X)$ which, since the maps
  $A \otimes \iota_n \otimes X$ constitute a colimiting cocone,
  implies the commutativity of the bottom diagram in~\eqref{eq:3}.
\end{proof}

There are other conditions under which a category will admit
algebraically-free commutative monoids. A refinement of the proof of
the last part of~\cite[Proposition~23.2]{Kelly1980A-unified} shows
that in a symmetric monoidal closed category with pullbacks and
equalisers, any free commutative monoid is necessarily
algebraically-free. This refinement is not entirely trivial, as it
uses the construction of \emph{centralisers} and \emph{centres} in
arbitrary monoidal categories; and since Proposition~\ref{prop:2}
covers all the examples of interest to us, we do not pursue this
further here.

\section{Lifted differential modalities on categories of actions}
\label{sec:lift-mono-diff}

We have now described the two key ingredients required for our 
results---the notion of (monoidal) differential modality, and the
notion of (free) commuting $X$-action. However, before proceeding, it
will be helpful to do some preparatory work to clarify the
interactions between these two notions.

In this section, we show that a differential modality $\oc$ on a
$k$-linear symmetric monoidal $(\C, \otimes, I)$ can be lifted to the
category $\ca$ of commuting $X$-actions for any $X \in \C$. This
result is interesting in its own right as a source of new examples of
differential modalities; however, it will also play a crucial role in
our main results in Sections~\ref{sec:constr-free-mono}
and~\ref{sec:free-mono-diff}. With an eye to this, we will be
rather careful about which parts of the structure of a differential
modality are used in establishing our various results.

\subsection{Lifting the endofunctor $\oc$}
\label{sec:lift-endof-oc}

First of all, let us recall the definition of the \emph{coderiving
  transformation} $\mathsf{b}$ associated to a differential modality
$\oc$; this is the natural transformation with components
\begin{equation}
  \label{eq:35}
  \mathsf{b}_X \defeq \oc X \xrightarrow{\Delta_X} \oc X \otimes \oc X \xrightarrow{\oc X \otimes \varepsilon_X} \oc X \otimes X\rlap{ .}
\end{equation}

Coderiving transformations may in fact be defined more generally for
any coalgebra modality---see~\cite[\sec2.2]{Cockett2019Integral}---and
part (i) of the following lemma is then part 4 of Proposition 2.1 of
\emph{op.~cit.}

\begin{Lemma}
  \label{lem:11}
  The coderiving transformation of a differential modality $\oc$
  satisfies the identities:
  \begin{enumerate}[(i)]
    \item $(\mathsf{b}_X \otimes X) \circ \mathsf{b}_X = (\oc X
      \otimes \sigma) \circ (\mathsf{b}_X \otimes X) \circ
      \mathsf{b}_X$;
  \item
    $\mathsf{b}_X \circ \mathsf{d}_X = \mathrm{id}_{\oc X
      \otimes X} + ((\mathsf{d}_X
    \otimes X) \circ (\oc X \otimes \sigma) \circ (\mathsf{b}_X
    \otimes X))$.
  \end{enumerate}
\end{Lemma}
\begin{proof}
  (i) is immediate from the cocommutativity of $\Delta$. (ii) follows
  from the product rule on postcomposing by $\oc X \otimes
  \varepsilon_X$ and applying the linear rule and one of the
  counitality identities for $\Delta$.
\end{proof}

Given $\mathsf{b}$ and  $\mathsf{d}$ satisfying the conditions of this
lemma---even if these data are not known to come from a differential
or coalgebra modality---we can define, for each $X \in \C$, a lifting
of the functor $\oc$ and the natural transformations $\mathsf{b}$ and
$\mathsf{d}$ to the category of commuting $X$-actions. Recall in what
follows that $\boxtimes$ is the monoidal structure on $\ca$ as
described in Definition~\ref{def:3}.

\begin{Lemma}
  \label{lem:15}
  Let $\oc$ be an endofunctor on a $k$-linear symmetric monoidal category
  $(\C, \otimes, I)$, and let
  $\mathsf{b}_X \colon \oc X \rightarrow \oc X \otimes X$ and
  $\mathsf{d}_X \colon \oc X \otimes X \rightarrow \oc X$ be natural
  families of maps satisfying the interchange law for $\mathsf{d}$ and
  the two axioms of Lemma~\ref{lem:11}.
  \begin{enumerate}[(i),itemsep=0.25\baselineskip]
  \item For any $X \in \C$ and any
  commuting $X$-action $(A, \alpha)$, the map
  \begin{equation}
    \label{eq:23}
    \oc^X \alpha \defeq \oc A \otimes X \xrightarrow{\mathsf{b}_A \otimes X} \oc A \otimes A \otimes X \xrightarrow{\oc A \otimes \alpha} \oc A \otimes A \xrightarrow{\mathsf{d}_A} \oc A
  \end{equation}
  provides the structure of a commuting $X$-action
  $\oc^{X}(A, \alpha) = (\oc A, \oc^X \alpha)$;
\item If $f \colon (A, \alpha) \rightarrow (B, \beta)$ is a morphism
  of commuting $X$-actions, then so too is
  $\oc f \colon \oc^X(A, \alpha) \rightarrow \oc^X(B, \beta)$;
\item For any commuting $X$-action $(A, \alpha)$, the following are
  maps of commuting $X$-actions:
  \begin{equation*}
    \mathsf{b}_A \colon \oc^X(A, \alpha) \rightarrow \oc^X(A, \alpha) \boxtimes (A, \alpha) \ \  \text{and} \ \ 
    \mathsf{d}_A \colon \oc^X(A, \alpha) \boxtimes (A, \alpha) \rightarrow \oc^X(A, \alpha) \rlap{ .}
  \end{equation*}
\end{enumerate}
Hence, the structure $(\oc, \mathsf{b}, \mathsf{d})$ on $(\C, \otimes, I)$ lifts to
a corresponding
structure $(\oc^X, \mathsf{b}, \mathsf{d})$ on $\left(\ca, \boxtimes, (I,0) \right)$.
\end{Lemma}
\begin{proof}
  For (i), the only issue is showing commutativity of the action. We
  begin with three preliminary remarks:
  \begin{itemize}[itemsep=0.25\baselineskip]
  \item The definition of
    the tensor product $\boxtimes$ on $\ca$ given in
    Definition~\ref{def:3} makes sense also for $X$-actions which do
    \emph{not} necessarily commute, thus giving a monoidal structure on
    the category of $X$-actions.
  \item If we abusively write $C$ for the $X$-action $(C, 0)$, then
    for any $X$-action $(B, \beta)$, the tensor product
    $(B, \beta) \boxtimes C$ is given by $B \otimes C$ with the action
    \begin{equation*}
      B \otimes C \otimes X \xrightarrow{1 \otimes \sigma} B \otimes X \otimes C \xrightarrow{\beta \otimes C} B \otimes C\rlap{ .}
    \end{equation*}
  \item In these terms, requiring $(B, \beta)$ to be a commuting
    action is equivalent to requiring
    $\beta \colon (B, \beta) \boxtimes X \rightarrow (B, \beta)$ to be
    a map of $X$-actions.
  \end{itemize}
  By virtue of these remarks, to show that $\oc^X \alpha$ is a
  commuting action, it suffices to show that the following is a
  composite of maps of $X$-actions:
  \begin{equation*}
    \oc^X(A, \alpha) \boxtimes X \xrightarrow{\!\mathsf{b}_A \otimes X\!}
    \oc^X(A, \alpha) \boxtimes (A, \alpha) \boxtimes X \xrightarrow{\!\oc A \otimes \alpha\!} 
    \oc^X(A, \alpha) \boxtimes (A, \alpha) \xrightarrow{\!\mathsf{d}_A\!} (\oc A, \oc^X \alpha) \rlap{.}
  \end{equation*}
  For the second map, this is immediate, since $(A, \alpha)$ is a
  commuting $X$-action and so
  $\alpha \colon (A, \alpha) \boxtimes X \rightarrow (A, \alpha)$. For
  the first and third maps we must do some work---which is exactly the
  same work as required for (iii), which we may thus dispatch
  simultaneously; and since (ii) is trivial, this will conclude the proof.

  To start with, let us show that
  $\mathsf{b}_A \colon (\oc A, \oc^X \alpha) \rightarrow (\oc A, \oc^X
  \alpha) \boxtimes (A, \alpha)$. Consider the following diagram,
  where here, and henceforth, we may for clarity's sake write $.$ in
  place of $\otimes$ where this causes no confusion:
  \begin{equation*}
    \cd{
      \oc A .  X \ar[r]^-{\mathsf{b}_A . 1} \ar[d]_-{\mathsf{b}_A . 1} & \oc A .  A .  X \ar[r]^-{1 . \sigma} \ar[dr]^-{\mathsf{b}_A . 1 . 1} & \oc A .  X .  A \ar[dr]^-{\mathsf{b}_A . 1 . 1} \\
      \oc A .  A .  X \ar[r]^-{\mathsf{b}_A . 1 . 1} \ar[d]_-{1 . \alpha} & \oc A .  A .  A .  X \ar[d]_-{1 . 1 . \alpha} \ar[r]^-{1 . \sigma . 1} & \oc A  .  A  .  A  .  X \ar[r]^-{1 . 1 . \sigma} & \oc A  .  A  .  X  .  A \ar[dl]^-{1 . \alpha . 1} \\
      \oc A  .  A \ar[r]^-{\mathsf{b}_A . 1} & \oc A  .  A  .  A \ar[r]^-{1 . \sigma} & \oc A  .  A  .  A\rlap{ .}
    }
  \end{equation*}
  The top-left region commutes by Lemma~\ref{lem:11}(i), and
  the remaining regions commute by axioms of a symmetric monoidal
  category. We thus have that:
  \begin{align*}
    & \ \ \ \ \ \mathsf{b}_A \circ \mathsf{d}_A \circ (\oc A . \alpha) \circ (\mathsf{b}_A . X)\\
    & = [(\oc A . \alpha) \circ (\mathsf{b}_A . X)] + [(\mathsf{d}_A
    . A) \circ (\oc A . \sigma) \circ (\mathsf{b}_A
    . A) \circ (\oc A . \alpha) \circ (\mathsf{b}_A . X)] \\
    & = [(\oc A . \alpha) \circ (\mathsf{b}_A . X)] + ((\mathsf{d}_A
    . A) \circ (\oc A . \alpha . A) \circ (\mathsf{b}_A . X . A) \circ (\oc A . \sigma) \circ (\mathsf{b}_A . X)) \\
    & = [(\oc A . \alpha) + ((\mathsf{d}_A . A) \circ (\oc A . \alpha . A) \circ (\mathsf{b}_A . X . A)\circ (\oc A . \sigma_{AX}))] \circ (\mathsf{b}_A . X)\\
    &= (\oc^{X} \alpha \boxtimes \alpha) \circ (\mathsf{b}_A . X)
  \end{align*}
  using, in succession: Lemma~\ref{lem:11}(i); the preceding
  observation; distributivity of $+$ over $\circ$; and the definition of
  $\oc^{X} \alpha \boxtimes \alpha$. This proves we have a morphism of
  $X$-actions $\mathsf{b}_A \colon (\oc A, \oc^X \alpha) \rightarrow (\oc A,
  \oc^X \alpha) \boxtimes (A, \alpha)$ as desired.

  It remains only to prove that we have
  $\mathsf{d}_A \colon (\oc A, \oc^X \alpha) \boxtimes (A, \alpha)
  \rightarrow (\oc A, \oc^X \alpha)$. For this, note that on passing
  from $\C$ to $\C^\mathrm{op}$, we obtain another instance of the same
  structure, with the roles of $\mathsf{b}$ and $\mathsf{d}$ reversed,
  with Lemma~\ref{lem:11}(i) trading places with the interchange law,
  and with Lemma~\ref{lem:11}(ii) being self dual. Thus, the preceding
  argument in $\C^\mathrm{op}$ shows
  $\mathsf{d}_A \colon (\oc A, \oc^X \alpha) \boxtimes (A, \alpha)
  \rightarrow (\oc A, \oc^X \alpha)$ as desired.
\end{proof}

\subsection{Lifting the comonoid structure}
\label{sec:lift-comon-struct}

We now describe conditions under which the comonoid maps $\Delta_A$
and $\mathsf{e}_A$ of a differential modality lift from $\C$ to $\ca$.
Again, we isolate the equational properties of the coderiving
transformation which are necessary to do this.

\begin{Lemma}
  \label{lem:12}
  The coderiving transformation of a differential modality $\oc$
  satisfies the identities:
  \begin{enumerate}[(i)]
  \item $(\Delta_X \otimes X) \circ \mathsf{b}_X = (\oc
    X \otimes \mathsf{b}_X) \circ \Delta_X$;
  \item $(\oc X \otimes \mathsf{b}_X) \circ \Delta_X = (\oc X \otimes
    \sigma_{X, \oc X})\circ (\mathsf{b}_X \otimes \oc X) \circ \Delta_X$.
  \end{enumerate}
\end{Lemma}
\begin{proof}
  Immediate from the coassociativity and cocommutativity of $\Delta$;
  see also~\cite[Proposition~2.1, parts 3 and 4]{Cockett2019Integral} for a
  string-diagrammatic proof.
\end{proof}
\begin{Lemma}
  \label{lem:14}
  In the situation of Lemma~\ref{lem:15}, suppose given further
  natural transformations $\mathsf{e}_X \colon \oc X \rightarrow I$
  and $\Delta_X \colon \oc X \rightarrow \oc X \otimes \oc X$
  satisfying the constant and the product rule with respect to
  $\mathsf{d}$ and the identities of Lemma~\ref{lem:12}. For any
  $(A, \alpha) \in \ca$, the maps below are maps of commuting
  $X$-actions, so that $\mathsf{e}$ and $\Delta$ lift to corresponding
  structure on the symmetric monoidal category $\left(\ca, \boxtimes, (I,0) \right)$.
  \begin{enumerate}[(i)]
  \item $\mathsf{e}_A \colon \oc^X (A, \alpha) \rightarrow (I, 0)$;
  \item $\Delta_A \colon \oc^X(A, \alpha) \rightarrow \oc^X(A, \alpha) \boxtimes
    \oc^X(A, \alpha)$.
  \end{enumerate}
\end{Lemma}
\begin{proof}
  For (i), observe that we have
  \begin{equation*}
    \mathsf{e}_A \circ \oc^X \alpha = \mathsf{e}_A \circ \mathsf{d}_A \circ (\oc A \otimes \alpha) \circ (\mathsf{b}_A \otimes X)
  \end{equation*}
  which is the zero map by the constant rule, and so equal to the
  composite $0 \circ (\mathsf{e}_A \otimes X)$ as required. For (ii),
  we have that:
  \begin{align*}
    \Delta_A \circ \oc^X \alpha
    & = \Delta_A \circ \mathsf{d}_A \circ (\oc A . \alpha) \circ (\mathsf{b}_A . X)\\
    & = [(\oc A . \mathsf{d}_A) 
    \,+\, (\mathsf{d}_A . \oc A) \circ (\oc A . \sigma_{\oc A,A})] \circ (\Delta_A . \alpha) \circ (\mathsf{b}_A . X)
\\
    & = [(\oc A . \mathsf{d}_A) 
    \,+\, (\mathsf{d}_A . \oc A) \circ (\oc A . \sigma_{\oc A,A})] \circ (\oc A. \oc A . \alpha) \circ (\oc A.\mathsf{b}_A.X) \circ (\Delta_A . X)
  \end{align*}
  using the definition of $\oc^X \alpha$; the product rule; and Lemma~\ref{lem:12}(i).
  Concentrating on the second summand, we have 
  \begin{align*}
    & \ \ \ \ (\mathsf{d}_A . \oc A) \circ (\oc A . \sigma_{\oc A,A}) \circ (\oc A. \oc A . \alpha) \circ (\oc A.\mathsf{b}_A.X) \circ (\Delta_A . X) \\
    & = (\mathsf{d}_A . \oc A) \circ (\oc A . \sigma_{\oc A,A}) \circ (\oc A. \oc A . \alpha) \circ (\oc A . \sigma_{A, \oc A}. X)\circ (\mathsf{b}_A.\oc A.X) \circ (\Delta_A . X)\\
    & = (\mathsf{d}_A . \oc A) \circ (\oc A . \alpha . \oc A) \circ (\mathsf{b}_A . X . \oc A)\circ (\oc A . \sigma_{\oc A, X}) \circ  (\Delta_A . X)
  \end{align*}
  using Lemma~\ref{lem:12}(ii) and axioms for the symmetry $\sigma$.
  Thus we have
  \begin{align*}
    & \ \ \ \ \ \Delta_A \circ \oc^X \alpha\\
    & = [(\oc A . \mathsf{d}_A) 
    \,+\, (\mathsf{d}_A . \oc A) \circ (\oc A . \sigma_{\oc A,A})] \circ (\oc A. \oc A . \alpha) \circ (\oc A.\mathsf{b}_A.X) \circ (\Delta_A . X)\\
    & = [(\oc A . \mathsf{d}_A) \!\circ\! (\oc A. \oc A . \alpha) \!\circ\! (\oc A.\mathsf{b}_A.X) 
    + (\mathsf{d}_A . \oc A) \!\circ\! (\oc A . \sigma_{\oc A,A}) \!\circ\! (\oc A. \oc A . \alpha) \!\circ\! (\oc A.\mathsf{b}_A.X)] \circ (\Delta_A . X)\\
    & = [(\oc A . \mathsf{d}_A) \!\circ\! (\oc A. \oc A . \alpha) \!\circ\! (\oc A.\mathsf{b}_A.X) 
    + (\mathsf{d}_A . \oc A) \!\circ\! (\oc A . \alpha . \oc A) \!\circ\! (\mathsf{b}_A . X . \oc A)\!\circ\! (\oc A . \sigma_{\oc A, X})] \circ (\Delta_A . X)\\
    & = (\oc^X \alpha \boxtimes \oc^X \alpha) \circ (\Delta_A . X)
  \end{align*}
  as desired, using the first calculation, distributivity of $+$ over
  $\circ$, the second calculation, and definition of
  $\oc^X \alpha \boxtimes \oc^X \alpha$.
\end{proof}

\subsection{Lifting the comonad structure}
\label{sec:lift-comon-struct-1}

Finally, we determine what is needed to lift the comonad structure of
a differential modality from $\C$ to $\ca$.

\begin{Lemma}
  \label{lem:16}
  The coderiving transformation of a differential modality $\oc$
  satisfies the identities:
  \begin{enumerate}[(i)]
  \item $(\mathsf{e}_X \otimes X) \circ \mathsf{b}_X = \varepsilon_X$;
  \item $(\delta_X \otimes \oc X) \circ \Delta_X = \mathsf{b}_{\oc
    X} \circ \delta_X$.
  \end{enumerate}
\end{Lemma}
\begin{proof}
  (i) follows from counitality of $\Delta$; (ii) from the fact that
  $\delta_X$ is a comonoid morphism and the comonad unit law
  $\varepsilon \oc \circ \delta = 1$.
  See~\cite[Proposition~2.1, part 1]{Cockett2019Integral} for a
  string-diagrammatic proof of (i); the corresponding proof for (ii)
  is not given directly in \emph{op.~cit.}, but is contained in the
  proof of part 7 of Proposition~2.1.
\end{proof}

\begin{Lemma}
  \label{lem:9}
  In the situation of Lemma~\ref{lem:14}, suppose given further
  natural transformations $\varepsilon_X \colon \oc X \rightarrow \oc
  \oc X$
  and $\delta_X \colon \oc X \rightarrow \oc \oc X$.
  \begin{enumerate}[(i),itemsep=0.25\baselineskip]
  \item Assuming $\varepsilon$ satisfies Lemma~\ref{lem:16}(i), each
    $\varepsilon_A$ 
    lifts to a map of commuting $X$-actions $\varepsilon_A \colon
    !^X(A, \alpha) \rightarrow (A, \alpha)$;
  \item Assuming $\delta$ satisfies Lemma~\ref{lem:16}(ii) and the chain rule
    with respect to $\mathsf{d}$, each $\delta_A$ lifts to a map
    of commuting $X$-actions
    $\delta_A \colon !^X(A, \alpha) \rightarrow !^X !^X(A, \alpha)$.
  \end{enumerate}
\end{Lemma}

\begin{proof}
  For (i), $\varepsilon_A$ will lift since Lemma~\ref{lem:16}(i)
  defines it in terms of maps which are already known to lift. For
  (ii), we consider the diagram:
  \begin{equation*}
    \cd{
      &
      \oc A . X
       \ar[r]^-{\mathsf{b}_A . 1} \ar[dl]_-{\delta_A . 1} \ar[d]|-{\Delta_A . 1} \ar@{}[dr]|-{\text{(2)}}&
      \oc A.A.X
       \ar[r]^-{1.\alpha} \ar[d]|-{\Delta_A.1.1} &
      \oc A . A
       \ar[r]^-{\mathsf{d}_A} \ar[d]_-{\Delta_A . 1} \ar@{}[dr]|-{\text{(3)}}&
      \oc A
      \ar[dr]^-{\delta_A} \\
      \oc \oc A . X
      \ar[dr]_-{\mathsf{b}_{\oc A} . 1} \ar@{}[r]|-{\text{(1)}}&
      \oc A . \oc A . X
      \ar[d]|-{\delta_A . 1 . 1} \ar[r]^-{1.\mathsf{b}_A.1}&
      \oc A . \oc A . A . X
      \ar[r]^-{1.1.\alpha} \ar[d]|-{\delta_A.1.1.1} &
      \oc A . \oc A . A
      \ar[r]^-{\delta_A . \mathsf{d}_A} \ar[dr]^-{\delta_A . 1 . 1} &
      \oc \oc A . \oc A \ar[r]^-{\mathsf{d}_{\oc A}} &
      \oc \oc A \rlap{ .}\\
      &
      \oc \oc A . \oc A . X
      \ar[r]^-{1.\mathsf{b}_A . 1} &
      \oc \oc A . \oc A . A . X
      \ar[rr]^-{1.1.\alpha} &&
      \oc \oc A. \oc A. A
      \ar[ur]_-{1.\mathsf{d}_A}
    }
  \end{equation*}
  The upper and lower composites are the maps
  $\delta_A \circ \oc^X \alpha$ and
  $\oc^X \oc^X \alpha \circ (\delta_A \otimes X)$ which are to be
  shown equal. The regions (1), (2) and (3) commute by, respectively,
  Lemma~\ref{lem:16}(i), Lemma~\ref{lem:12}(i), and the chain rule,
  and all other regions commute by bifunctoriality of $\otimes$.
\end{proof}

\subsection{Putting it all together}
\label{sec:putting-it-all}

From the above lemmas in this section, we have:
\begin{Prop}
  \label{prop:5}
  If $\oc$ is a differential modality on a $k$-linear symmetric
  monoidal category $(\C, \otimes, I)$, then it lifts to a differential modality
  $\oc^X$ on $\left(\ca, \boxtimes, (I,0) \right)$, where the functor $\oc^X$ is defined
  as in Lemma~\ref{lem:15}, and the structural natural transformations
  are lifts of those of $\oc$.
\end{Prop}

It would be natural to think that there is an extension of this
result: if $\oc$ is a \emph{monoidal} differential modality on $\C$,
then $\oc^X$ is a monoidal differential modality on $\ca$. However,
this need not be the case:
\begin{Lemma}
  \label{lem:1}
  Let $\oc$ be a monoidal differential modality on a $k$-linear
  symmetric monoidal $(\C, \otimes, I)$. The lifted differential modality
  $\oc^X$ on $\ca$ need not be monoidal.
\end{Lemma}
\begin{proof}
  Consider the case where $\C = \cat{Vect}_\mathbb{R}^\mathrm{op}$,
  the opposite of the category of real vector spaces, and where $\oc$
  is the free symmetric algebra construction; by Example~\ref{ex:6}
  this equips $\cat{Vect}_\mathbb{R}^\mathrm{op}$ with the structure
  of a monoidal differential modality. We show that the lifting of
  $\oc$ to $\ca$, where $X = \mathbb{R}$ is the one-dimensional real
  vector space, is not monoidal. By Proposition~\ref{prop:10}, it
  suffices to show that the maps
  $\mathsf{u}_A \colon I \rightarrow \oc A$ do not lift to maps
  $\mathsf{u}_{(A, \alpha)} \colon (I,0) \rightarrow \oc^X(A,
  \alpha)$.

  Now, when $X = \mathbb{R}$, much as in Example~\ref{ex:1}, any
  $X$-action in $\cat{Vect}^\mathrm{op}$ is automatically commutative,
  and $\ca$ is isomorphic to $(\cat{End}_\mathbb{R})^\mathrm{op}$,
  where $\cat{End}_\mathbb{R}$ is the category of real vector spaces
  $V$ equipped with a linear endomorphism.
  Consider, in particular, the object
  $(\mathbb{R}, \mathrm{id}) \in \ca$. In this case:
  \begin{itemize}
  \item $\oc \mathbb{R}$
    is the vector space of single-variable polynomials $\mathbb{R}[x]$;
  \item $\mathsf{u}_\mathbb{R} \colon I \rightarrow \oc \mathbb{R}$ is
    the algebra map $\mathbb{R} \leftarrow \mathbb{R}[x]$ given by
    evaluation at $0$;
  \item the object $\oc^\mathbb{R}(\mathbb{R}, \mathrm{id})$ equips
    $\mathbb{R}[x]$ with the linear endomorphism $\oc^\mathbb{R} f_k$
    defined on monomials by:
    \begin{equation*}
      \cd[@R-2em@C+0.5em]{
        \mathbb{R}[x] \ar@{<-}[r]^-{\mathsf{b}_\mathbb{R}} &
        \mathbb{R}[x] \otimes \mathbb{R} \ar@{<-}[r]^-{\mathsf{d}_\mathbb{R}} &
        \mathbb{R}[x] \\
        nx^{n-1} \ar@{<-|}[r]&
        x^{n-1} \otimes n \ar@{<-|}[r]&
        x^n\rlap{ ,}
      }
    \end{equation*}
    i.e., the usual differential operator $\partial \colon \mathbb{R}[x]
    \rightarrow \mathbb{R}[x]$.  
  \end{itemize}
  Thus, to say that $\mathsf{u}_\mathbb{R}$ lifts to a map $(I, 0)
  \rightarrow \oc^\mathbb{R}(\mathbb{R}, \mathrm{id})$ is to say that
  the following square of linear maps commutes:
  \begin{equation*}
    \cd{
      \mathbb{R} & \mathbb{R}[x] \ar[l]_-{\mathrm{ev}_0} \\
      \mathbb{R} \ar[u]^-{0} & \mathbb{R}[x] \ar[u]^-{\partial} \ar[l]_-{\mathrm{ev}_0}
    }
  \end{equation*}
  which is false, since the two composites take $x \in
  \mathbb{R}[x]$ to $0$ and $1$ respectively.
\end{proof}

\section{Differential modality axioms via commuting actions}
\label{sec:diff-modal-axioms}

In this section, we describe a second relationship between
differential modalities and commuting actions, by showing that the
axioms for a differential modality may be re-expressed in terms of
certain morphisms being maps of commuting actions. Again, this will be
important for our main theorems, and so again we will be somewhat
granular about the hypotheses needed to establish the results of this
section.

\subsection{Interchange, constant and product rules}
\label{sec:interchange-rule}

We start with two easy lemmas.

\begin{Lemma}
  \label{lem:8}
  Let $\oc$ be an endofunctor on a $k$-linear symmetric
  monoidal categoryy $(\C, \otimes, I)$. Suppose we are given maps $\mathsf{d}_X \colon
  \oc X \otimes X \rightarrow \oc X$ for each $X \in \C$. 
 The maps $\mathsf{d}_X$ satisfy the interchange axiom if, and
    only if, for each $X \in \C$ the pair
    $(\oc X, \mathsf{d}_X \colon \oc X \otimes X \rightarrow \oc X)$
    is a commuting $X$-action.
\end{Lemma}
\begin{proof}
  Immediate from the definitions.
\end{proof}

For the second lemma, recall that $\boxtimes$ is the lifted symmetric
monoidal structure on commuting $X$-actions of Definition~\ref{def:3}.
\begin{Lemma}
  \label{lem:3}
  Suppose the equivalent conditions of~Lemma \ref{lem:8} are
  satisfied, and suppose given maps
  $\mathsf{e}_X \colon \oc X \rightarrow I$ and
  $\mathsf{\Delta}_X \colon \oc X \rightarrow \oc X \otimes \oc X$
  endowing each $X$ with cocommutative comonoid structure. These
  become maps of commuting $X$-actions
  \begin{equation*}
    \mathsf{e}_X \colon (\oc X, \mathsf{d}_X) \rightarrow (I, 0) \quad \text{respectively} \quad 
    \Delta_X \colon (\oc X, \mathsf{d}_X) \rightarrow (\oc X, \mathsf{d}_X) \boxtimes (\oc X, \mathsf{d}_X)
  \end{equation*}
  if, and only if, $\mathsf{d}$ satisfies the constant rule,
  respectively, the product rule.
\end{Lemma}
\begin{proof}
  Again, this is immediate from the definitions.
\end{proof}

\subsection{Linear rule}
\label{sec:linear-rule}

The linear rule is less direct than the previous ones, and will
require a further assumption on $\C$.

\begin{Defn}[Nilsquare $X$-action]
  \label{def:12}
  An $X$-action $(A, \alpha \colon A \otimes X \rightarrow A)$ is
  called \emph{nilsquare} if $\alpha \circ (\alpha \otimes X) \colon A
  \otimes X \otimes X \rightarrow A \otimes X \rightarrow A$ is the
  zero map.
\end{Defn}

Note that, a fortiori, a nilsquare action is a commuting action.
If we assume that $\C$ admits finite biproducts, then we can
construct the free nilsquare $X$-action on $A \in \C$ as the object
$A \oplus (A \otimes X)$ equipped with the action
\begin{equation*}
  \bigl( A \oplus (A \otimes X) \bigr) \otimes X \xrightarrow{\pi_0 \otimes X} A \otimes X \xrightarrow{\iota_1} A \oplus (A \otimes X)\rlap{ .}
\end{equation*}
In particular, we have the free nilsquare $X$-action on $I$, which we
can describe as the object $I \oplus X$, equipped with the action:
\begin{equation*}
  (I \oplus X) \otimes X \xrightarrow{\pi_0 \otimes X} X \xrightarrow{\iota_1} I \oplus X \rlap{ .}
\end{equation*}
\begin{Lemma}
  \label{lem:4}
  Suppose the equivalent conditions of~Lemma \ref{lem:3} are
  satisfied, and suppose $\C$ admits
  finite biproducts. A family of maps $\varepsilon_X \colon \oc X
  \rightarrow X$ induces maps of commuting $X$-actions:
  \begin{equation*}
    (\mathsf{e}_X, \varepsilon_X) \colon (\oc X, \mathsf{d}_X) \rightarrow (I \oplus X, \iota_1 \circ (\pi_0 \otimes X))\rlap{ .}
  \end{equation*}
  if, and only if, $\mathsf{d}$ satisfies the linear rule.
\end{Lemma}
\begin{proof}
  To say that this is a map of $X$-actions is equally, on
  postcomposing with the two product projections $I \leftarrow I
  \oplus X \rightarrow X$, to require the two equalities
  $\mathsf{e}_X \circ \mathsf{d}_X = 0$ and
  $\varepsilon_X \circ \mathsf{d}_X = \mathsf{e}_X
  \otimes X$. The first of these is again the constant rule, while the
  second is the linear rule.
\end{proof}

\subsection{Chain rule}
\label{sec:chain-rule}
Finally, we deal with the chain rule. For this, we take advantage of
the results of Section~\ref{sec:lift-mono-diff}. Indeed, if the
hypotheses of~Lemma \ref{lem:4} are verified, then (using the linear
and product rules) the two equations of Lemma~\ref{lem:11} will be
satisfied; whence, by Lemma~\ref{lem:15}, we can lift $\oc$ to an
endofunctor $\oc^X \colon \ca \rightarrow \ca$.
\begin{Lemma}
  \label{lem:7}
  Suppose the equivalent conditions of~Lemma \ref{lem:4} are
  satisfied. Given a family of maps
  $\delta_X \colon \oc X \rightarrow \oc \oc X$ which are comonoid
  morphisms for $(\Delta, \mathsf{e})$ and which satisfy the counit
  law $\varepsilon_{\oc X} \circ \delta_X = 1$, we have that the
  following is a map of commuting $X$-actions:
  \begin{equation*}
    \delta_X \colon (\oc X, \mathsf{d}_X) \rightarrow \oc^X (\oc X, \mathsf{d}_X)
  \end{equation*}
  if, and only if, $\mathsf{d}$ satisfies the chain rule.
\end{Lemma}
\begin{proof}
  To say that $\delta_X$ is a map of $X$-actions is to say that we
  have the equality
  \begin{equation*}
    \delta_X \circ \mathsf{d}_X = \mathsf{d}_{\oc X} \circ (\oc \oc X \otimes \mathsf{d}_X) \circ (\mathsf{b}_{\oc X} \otimes X) \circ (\delta_X \otimes X) \rlap{ .}
  \end{equation*}
  The hypotheses on $\delta$ suffice to verify
  Lemma~\ref{lem:16}(ii), i.e., that
  $(\mathsf{b}_{\oc X} \otimes X) \circ (\delta_X \otimes X) =
  (\delta_{X} \otimes \oc X \otimes X) \circ (\Delta_{X} \otimes X)$;
  so asserting the above equality is equivalent to asserting
  \begin{align*}
    \delta_X \circ \mathsf{d}_X &= \mathsf{d}_{\oc X} \circ (\oc \oc X \otimes \mathsf{d}_X) \circ (\delta_{\oc X} \otimes \oc X \otimes X) \circ (\Delta_{X} \otimes X) \rlap{ ,}\\
    \text{i.e., that }\ \ \ 
    \delta_X \circ \mathsf{d}_X &= \mathsf{d}_{\oc X} \circ (\delta_{\oc X} \otimes \mathsf{d}_X) \circ (\Delta_{X} \otimes X) \rlap{ ,}
  \end{align*}
  which is precisely the chain rule.
\end{proof}

\section{The free differential modality on a 
  coalgebra modality}

In this section, we are finally ready to give the first main result of
the paper. For any $k$-linear symmetric category $(\C, \otimes, I)$
which admits algebraically-free commutative monoids and finite
biproducts, we will show that the forgetful functor
$\cat{DiffMod}(\C) \rightarrow \cat{CoalgMod}(\C)$ from differential
modalities to coalgebra modalities has a left adjoint. In other words,
we will construct the \emph{free} differential modality on a given
coalgebra modality.

\label{sec:constr-free-mono}
\subsection{The plan}

For the rest of this section, we suppose $(\C, \otimes, I)$ is a $k$-linear
symmetric monoidal category with finite biproducts that admits the
construction of algebraically-free commutative monoids. Thus, for each
$X \in \C$, we have an object $SX$ equipped with a unit map and a
universal commuting $X$-action of the form
\begin{align*}
  \mathsf{u}^S_X \colon I &\rightarrow SX & \mathsf{d}^S_X \colon SX \otimes X &\rightarrow SX\rlap{ ;}
\end{align*}
note the superscripts, which will help us to distinguish the structure
maps of $S$ from those for other modalities. Indeed, we now suppose
that $\oc$ is a coalgebra modality on $\C$, whose structure maps will
be written unadorned; and we aim to construct the free differential
modality $\oc$, written as $\ocf$ and with structure maps
superscripted by ``$\partial$''.

The freeness asserts that there is a morphism of coalgebra modalities
$\zeta \colon \oc \rightarrow \ocf$ such that, for any differential
modality $\oc'$ and map $\psi \colon \oc \rightarrow \oc'$ of
coalgebra modalities, there is a unique morphism of differential modalities
$\psi^\sharp \colon \ocf \rightarrow \oc'$ for which
$\psi^\sharp \circ \zeta = \psi$, that is, which renders commutative the diagram
\begin{equation}
  \label{eq:24}
  \cd[@C+1em@R-1em]{
    \oc \ar[d]_-{\zeta} \ar[r]^-{\psi \text{ coalg}} & \oc'\rlap{ .} \\
    \ocf \ar@{-->}[ur]_-{\exists ! \,\psi^\sharp \text{ diff}} 
  }
\end{equation}
We break our task
into several steps:
\begin{itemize}[itemsep=0.25\baselineskip]
\item \textbf{First step}: we construct the endofunctor
  $\ocf \colon \C \rightarrow \C$ and the natural families of maps
  $\mathsf{d}^\partial_X \colon \ocf X \otimes X \rightarrow \ocf X$
  and $\zeta_X \colon \oc X \rightarrow \ocf X$, and then show that
  $\mathsf{d}^\partial$ satisfies the interchange rule.
\item \textbf{Second step}: we construct the natural families of maps
  $\mathsf{e}^\partial_X \colon \ocf X \rightarrow I$ and
  $\Delta^\partial_X \colon \ocf X \rightarrow \ocf X \otimes \ocf X$, show they satisfy the constant and product rules for $\mathsf{d}^\partial$
  and also endow each $\ocf X$ with cocommutative comonoid structure.
\item \textbf{Third step}: we construct the natural family of maps
  $\varepsilon^\partial_X \colon \ocf X \rightarrow X$ and show that
  it satisfies the linear rule for $\mathsf{d}^\partial$.
\item \textbf{Fourth step}: we construct the natural family of maps
  $\delta^\partial_X \colon \ocf X \rightarrow \ocf \ocf X$.
\item \textbf{Fifth step}: we verify that each $\delta^\partial_X$ is
  a map of comonoids, and that the comonad counit axiom
  $\ocf \varepsilon^\partial \circ \delta^\partial = 1$ holds. Using
  this, we can verify that $\delta^\partial_X$ satisfies the chain
  rule for $\mathsf{d}^\partial$, and finally the remaining two comonad
  axioms. Thus, we
  conclude that $\ocf$ is a differential modality.
\item \textbf{Sixth step}: we verify that the
  $\zeta \colon \oc \rightarrow \ocf$ found in the first step exhibits
  $\ocf$ as the free differential modality on $\oc$.
\end{itemize}

\subsection{The endofunctor}
\label{sec:first-step}
To motivate the definition of $\ocf \colon \C \rightarrow \C$,
let us consider the component at some $X \in \C$ of the desired
universal extension~\eqref{eq:24}:
\begin{equation}
  \label{eq:25}
  \cd[@C+1em@R-1em]{
    \oc X \ar[d]_-{\zeta} \ar[r]^-{\psi_X} & \oc' X\rlap{ .} \\
    \ocf X \ar@{-->}[ur]_-{\psi^\sharp_X} 
  }
\end{equation}
Here, since $\oc'$ is a differential modality, the object $\oc' X$
bears, by Lemma~\ref{lem:8}(ii), a commuting $X$-action
$\mathsf{d}'_X \colon \oc' X \otimes X \rightarrow X$ . We know from
Section~\ref{sec:defin-relat-free} that the \emph{free} commuting
$X$-action on $\oc X$ is given by
$(\oc X \otimes SX, \oc X \otimes \mathsf{d}^S_X)$ with the freeness
exhibited by the unit map
$\oc X \otimes \mathsf{u}^S_X \colon \oc X \rightarrow \oc X \otimes
SX$. In particular, this means that there is a \emph{unique} map
${\bar \psi}_X \colon \oc X \otimes SX \rightarrow \oc' X$ which
commutes with the $X$-actions and renders commutative the triangle:
\begin{equation}
  \label{eq:26}
  \cd[@C+1em@R-1em]{
    \oc X \ar[d]_-{\oc X \otimes \mathsf{u}^S_X} \ar[r]^-{\psi_X} & \oc' X\rlap{ .} \\
    \oc X \otimes SX \ar@{-->}[ur]_-{\bar \psi_X} 
  }
\end{equation}
Comparing~\eqref{eq:26} with~\eqref{eq:25}, we are thus motivated to give the following definition. 
\begin{Defn}
  \label{def:13}
We define $\ocf \colon \C \rightarrow \C$ on objects
and maps by:
\begin{align}\label{eq:30}
  \ocf X \colon = \oc X \otimes S X  \qquad \text{and} \qquad \ocf f = \oc f \otimes Sf\rlap{ ,}
\end{align}
and we define the families of maps
$\mathsf{d}^{\partial}_X \colon \ocf X \otimes X \rightarrow \ocf X$ and
$\zeta_X \colon \oc X \rightarrow \ocf X$ as:
\begin{align}\label{def:!S-deriving}
  \mathsf{d}^{\partial}_X \defeq \oc X \otimes \mathsf{d}^S_X \colon & \oc X \otimes SX \otimes X \rightarrow \oc X \otimes SX\\
  \label{def:!S-zeta}
    \zeta_X \defeq \oc X \otimes \smash{\mathsf{u}^S_X} \colon & \oc X \rightarrow \oc X \otimes SX\rlap{ .}
  \end{align}
\end{Defn}
\begin{Lemma}
  \label{lem:5}
  $\ocf$ is a functor, and $\mathsf{d}^\partial$ and $\zeta$ are
  natural transformations. Moreover $\mathsf{d}^\partial$ satisfies
  the interchange rule.
\end{Lemma}
\begin{proof}
  In~\eqref{eq:30}, we use the functorial action of $S$ from
  Remark~\ref{rk:4}, which, as noted there, makes the maps
  $\mathsf{u}^S_X$ and $\mathsf{d}^S_X$ natural in $X$. The final claim
  follows from Lemma~\ref{lem:8}.
\end{proof}
\subsection{The comonoid structure}
\label{sec:second-step}
We next construct the natural families of maps
$\mathsf{e}^\partial_X \colon \ocf X \rightarrow I$ and
$\Delta^\partial_X \colon \ocf X \rightarrow \ocf X \otimes \ocf X$.
Note that, since $\mathsf{d}^\partial$ is supposed to satisfy the
constant and product rules, these maps must by 
Lemma~\ref{lem:3} underlie maps of commuting $X$-actions
\begin{equation}
  \label{eq:28}
  \mathsf{e}^\partial_X \colon (\ocf X, \mathsf{d}^\partial_X) \rightarrow (I, 0) \quad \text{and} \quad 
  {\Delta^\partial_X} \colon (\ocf X, \mathsf{d}^\partial_X) \rightarrow (\ocf X, \mathsf{d}^\partial_X) \boxtimes (\ocf X, \mathsf{d}^\partial_X)\rlap{ .}
\end{equation}
Thus, as $\zeta_X$ exhibits $(\ocf X, \mathsf{d}^\partial_X)$ as the
free commuting $X$-action on $\oc X$, these maps must be determined by
their precomposition with $\zeta_X$. But since
$\zeta \colon \oc \rightarrow \ocf$ is to be a morphism of coalgebra
modalities, we know what these precomposites must be, since we are to
have commutativity in the diagrams:
\begin{equation}
  \label{eq:27}
  \cd{
    \oc X \ar[d]_-{\zeta_X} \ar[r]^-{\mathsf{e}_X} & I \ar@{=}[d]_-{}&&
    \oc X \ar[d]_-{\zeta_X} \ar[r]^-{ \Delta_X} &
    \oc X \otimes \oc X \ar[d]^-{ \zeta_X \otimes \zeta_X } \\
    \ocf X \ar[r]^-{\mathsf{e}^{\partial}_X} & I &&
    \ocf X \ar[r]^-{\Delta^{\partial}_X} & \ocf X \otimes \ocf X \rlap{ .}
  }
\end{equation}
Thus, we are forced to give:
\begin{Defn}
  \label{def:14}
  We define $\mathsf{e}^\partial_X
  \colon \ocf X \rightarrow I$ and $\Delta^\partial_X \colon \ocf
  X \rightarrow \ocf X \otimes \ocf X$ as the unique maps of commuting
  $X$-actions~\eqref{eq:28} which render commutative the diagrams~\eqref{eq:27}.
\end{Defn}

\begin{Lemma}\label{lem:initial-comonoid} $\mathsf{e}^{\partial}$ and
  $\Delta^{\partial}$ are natural transformations and satisfy the
  constant and product rules with respect to $\mathsf{d}^\partial$.
  Furthermore, for each object $X$,
  $(\ocf X, \Delta^{\partial}_X, \mathsf{e}^{\partial}_X)$ is a
  cocommutative comonoid.
\end{Lemma}
\begin{proof}
  The constant and product rules are immediate from~\eqref{eq:28} and
  Lemma~\ref{lem:7}. For naturality, we note that each edge of the
  desired naturality squares can be enhanced to a map of commuting
  $X$-actions:
  \begin{equation*}
    \cd[@C-1.5em]{
      (\ocf X, \mathsf{d}^\partial_X) \ar[r]^-{\ocf f} \ar[d]_-{\mathsf{e}_X^\partial} &
      f^\ast(\ocf Y, \mathsf{d}^\partial_Y) \ar[d]^-{\mathsf{e}_Y^\partial} &
      (\ocf X, \mathsf{d}^\partial_X) \ar[r]^-{\ocf f} \ar[d]_-{\mathsf{\Delta}_X^\partial} &
      f^\ast(\ocf Y, \mathsf{d}^\partial_Y) \ar[d]^-{\mathsf{\Delta}_Y^\partial}
      \\
      (I, 0) \ar[r]^-{1_I} & f^\ast(I,0) &
      \sh{l}{1em}(\ocf X, \mathsf{d}^\partial_X) \boxtimes (\ocf X, \mathsf{d}^\partial_X) \ar[r]^-{\ocf f \otimes \ocf f}  &
      \sh{r}{1em}f^\ast\bigl((\ocf Y, \mathsf{d}^\partial_Y) \boxtimes (\ocf Y, \mathsf{d}^\partial_Y)\bigr)\rlap{ .}
    }
  \end{equation*}
  Since $\zeta_X$ exhibits $(\ocf X, \mathsf{d}^\partial_X)$ as the
  free commuting $X$-action on $\oc X$, it thus suffices to show
  commutativity of these squares on precomposition by $\zeta_X$. But
  by naturality of $\zeta$ and commutativity in~\eqref{eq:27}, these
  precompositions yield the naturality square for $\mathsf{e}$ at $f$,
  respectively, the naturality square for $\Delta$ at $f$ postcomposed
  by $\zeta_Y \otimes \zeta_Y$; whence the result.

  For the cocommutative comonoid axioms, we argue similarly. For
  example, the coassociativity square for $\Delta^\partial$ can be
  enhanced to a square
  \begin{equation*}
    \cd[@C+1em]{
      (\ocf X, \mathsf{d}^\partial_X)
      \ar[r]^-{{\mathsf{\Delta}_X^\partial}}
      \ar[d]_-{{\mathsf{\Delta}_X^\partial}} &
      (\ocf X, \mathsf{d}^\partial_X) \boxtimes (\ocf X, \mathsf{d}^\partial_X) \ar[d]_-{{\mathsf{\Delta}_X^\partial} \otimes \ocf X} \\
      (\ocf X, \mathsf{d}^\partial_X) \boxtimes (\ocf X, \mathsf{d}^\partial_X) \ar[r]^-{\ocf X \otimes {\mathsf{\Delta}_X^\partial}} &
      (\ocf X, \mathsf{d}^\partial_X) \boxtimes(\ocf X, \mathsf{d}^\partial_X) \boxtimes (\ocf X, \mathsf{d}^\partial_X)
    }
  \end{equation*}
  of commuting $X$-actions, and will commute since it does so on
  precomposition with $\zeta_X$---for this yields the
  coassociativity square for $\Delta$, postcomposed by $\zeta_X
  \otimes \zeta_X \otimes \zeta_X$. The argument for the remaining
  axioms is identical.
\end{proof}

\subsection{The comonad counit}
\label{sec:third-step}
We now construct the maps
$\varepsilon^\partial_X \colon \ocf X \rightarrow X$  providing
the comonad counit. Arguing as in the preceding section, since
$\mathsf{d}^\partial$ is to satisfy the linear rule, we must
by Lemma~\ref{lem:4} have that the following is a map
of commuting $X$-actions:
\begin{equation}
  \label{eq:29}
  (\mathsf{e}^\partial_X, \varepsilon^\partial_X) \colon (\ocf X, \mathsf{d}^\partial_X) \rightarrow \bigl(I \oplus X, \iota_1 \circ (\pi_0 \otimes X)\bigr)
\end{equation}
Moreover, for $\zeta \colon \oc \rightarrow \ocf$ to be a map of
coalgebra modalities, we must have commutativity of the square to the
left below and hence---given~\eqref{eq:27}---also commutativity of the
square to the right:
\begin{equation}
  \label{eq:6}
  \cd[@C+0.5em]{
    \oc X \ar[d]_-{\zeta_X} \ar[r]^-{\varepsilon_X} & X \ar@{=}[d]_-{} & &
    \oc X \ar[d]_-{\zeta_X} \ar[r]^-{(\mathsf{e}_X, \varepsilon_X)} & I\oplus X \ar@{=}[d]\\
    \ocf X \ar[r]^-{\varepsilon^{\partial}_X} & X&&
    \ocf X \ar[r]^-{(\mathsf{e}_X^\partial,\varepsilon^{\partial}_X)} & I \oplus X\rlap{ .}
  }
\end{equation}
Again, since $\zeta_X$ exhibits $\ocf X$ as the free commuting
$X$-action on $\oc X$, we are thus forced to make:

\begin{Defn}
  \label{def:16}
  We define $\varepsilon_X^\partial \colon \ocf X \rightarrow X$ as
  the unique morphism which, when taken together with
  $\mathsf{e}^\partial_X$, yields a map of commuting $X$-actions as
  in~\eqref{eq:29} and renders commutative the right diagram
  in~\eqref{eq:6}.
\end{Defn}

\begin{Lemma}
  \label{lem:6}
  $\varepsilon^{\partial}$ is a natural
  transformation and satisfies the linear rule for $\mathsf{d}^\partial$.
\end{Lemma}
\begin{proof}
  The linear rule is immediate from~\eqref{eq:29} and
  Lemma~\ref{lem:4}. For naturality, we argue as in
  Lemma~\ref{lem:initial-comonoid}. Indeed, we have a square of
  commuting $X$-actions:
  \begin{equation*}
    \cd[@C+0.5em]{
      (\ocf X, \mathsf{d}^\partial_X) \ar[r]^-{\ocf f} \ar[d]_-{(\mathsf{e}^\partial_Y, \varepsilon^\partial_X)} &
      f^\ast(\ocf Y, \mathsf{d}^\partial_Y) \ar[d]^-{(\mathsf{e}^\partial_Y, \varepsilon^\partial_Y)} 
      \\
      (I\oplus X, \iota_1 \circ (\pi_0 \otimes X)) \ar[r]^-{I \oplus f} & f^\ast(I\oplus Y, \iota_1 \circ (\pi_0 \otimes Y)) }
  \end{equation*}
  which commutes on precomposition by $\zeta_X$ using naturality of
  $\zeta$,~\eqref{eq:6}, and naturality of $\mathsf{e}$ and
  $\varepsilon$; whence the result.
\end{proof}

\subsection{The comonad comultiplication}
\label{sec:comon-comult}

Now that we have both $\Delta^\partial$ and
$\varepsilon^\partial$, we can also define the coderiving transformation
$\mathsf{b}^\partial$ as in~\eqref{eq:35}:
\begin{equation*}
    \mathsf{b}^\partial_X \defeq \ocf X \xrightarrow{\Delta^\partial_X} \ocf X \otimes \ocf X \xrightarrow{\ocf X \otimes \varepsilon^\partial_X} \ocf X \otimes X\rlap{ .}
\end{equation*}
We may note right away that:
\begin{Lemma}
  \label{lem:10}
  $\mathsf{b}^\partial$ satisfies the axioms of Lemma~\ref{lem:11},
  Lemma~\ref{lem:12}, and Lemma~\ref{lem:16}(i).
\end{Lemma}
\begin{proof}
  All of these follow from the cocommutative comonoid axioms for $\Delta^\partial$
  except for Lemma~\ref{lem:11}, which additionally uses the product
  and linear rules for $\mathsf{d}^\partial$; these were verified in
  Lemmas~\ref{lem:initial-comonoid} and \ref{lem:6} respectively.
\end{proof}
Since $\mathsf{d}^\partial$ is also known to satisfy the interchange
rule (Lemma~\ref{lem:5}), we can conclude by Lemma~\ref{lem:15} that,
for any $X \in \C$, the functor $\ocf$ lifts to a functor
\begin{equation*}
  \oc^{\partial X} \defeq (\ocf)^X \colon \ca \rightarrow \ca
\end{equation*}
defined as in~\eqref{eq:23}. Using this fact, we can now construct the
maps $\delta^\partial_X \colon \ocf X \rightarrow X$ providing the
comonad comultiplication.

Indeed, since $\mathsf{d}^\partial$ is to satisfy the chain rule, we
must by Lemma~\ref{lem:7} have that the following is a map of
commuting $X$-actions:
\begin{equation}
  \label{eq:34}
  {\delta^\partial_X} \colon (\ocf X, \mathsf{d}^\partial_X) \rightarrow \oc^{\partial X}(\ocf X, \mathsf{d}^\partial X)\rlap{ .}
\end{equation}
Moreover, for $\zeta \colon \oc \rightarrow \ocf$ to be a morphism of
coalgebra modalities, we must have commutativity of the square:
\begin{equation}
  \label{eq:36}
  \cd[@C+1em]{
  \oc X \ar[d]_-{\zeta_X} \ar[r]^-{ \delta_X} &
  \oc \oc X  \ar[d]^-{ (\zeta \zeta)_X } \\
  \ocf X \ar[r]^-{\delta^{\partial}_X} & \ocf \ocf X \rlap{ .}
  }
\end{equation}
As $\zeta_X$ exhibits $\ocf X$ as the free commuting
$X$-action on $\oc X$, this forces:

\begin{Defn}
  \label{def:17}
  We define $\delta_X^\partial \colon \ocf X \rightarrow X$ as the
  unique map of commuting $X$-actions~\eqref{eq:34} which renders
  commutative~\eqref{eq:36}.
\end{Defn}

By exactly the same argument as in the preceding sections, we have:
\begin{Lemma} $\delta^{\partial}$ is a natural transformation.
\end{Lemma}

\subsection{The differential modality structure}
\label{sec:diff-modal-struct}
Note that we cannot yet apply Lemma~\ref{lem:7} to conclude that
$\delta^\partial$ satisfies the chain rule for $\mathsf{d}^\partial$,
because we have not yet verified the side-conditions required for the
lemma. We do this next. 

\begin{Lemma}\label{lem:delta-comonoid}
  Each $\delta^{\partial}_X \colon \ocf X \rightarrow \ocf \ocf X$ is
  a comonoid morphism for $(\Delta^\partial, \mathsf{e}^\partial)$.
\end{Lemma}
\begin{proof} We have diagrams of commuting $X$-actions
  \begin{equation*}
    \cd[@C-1.6em]{
      {(\ocf X, \mathsf{d}^\partial_X)} \ar[r]^-{\delta^\partial_X} \ar[d]_{\mathsf{e}^\partial_X} &
      \sh{r}{0.5em}{\oc^{\partial X}(\ocf X, \mathsf{d}^\partial_X)} \ar[d]^{\mathsf{e}^\partial_{\ocf X}} & 
      {(\ocf X, \mathsf{d}^\partial_X)} \ar[r]^-{\delta^\partial_X} \ar[d]_{\Delta^\partial_X} &
      {\oc^{\partial X}(\ocf X, \mathsf{d}^\partial_X)} \ar[d]^{\Delta^\partial_{\ocf X}}
      \\
      {(I, 0)} \ar@{=}[r] &
      {(I, 0)} &
      \sh{l}{1.6em}{(\ocf X, \mathsf{d}^\partial_X) \boxtimes (\ocf X, \mathsf{d}^\partial_X)} \ar[r]^-{\delta^\partial_X \otimes \delta^\partial_X} &
      \sh{l}{0.3em}{\oc^{\partial X}(\ocf X, \mathsf{d}^\partial_X) \boxtimes \oc^{\partial X}(\ocf X, \mathsf{d}^\partial_X)}
    }
  \end{equation*}
  where the horizontal maps are maps of $X$-actions by~\eqref{eq:34},
  the left-hand vertical maps in each square are so by~\eqref{eq:28},
  and the right-hand vertical maps in each square are so by
  Lemma~\ref{lem:14}, whose preconditions were verified in
  Lemma~\ref{lem:10}. Moreover, the squares commute on precomposition
  by $\zeta_X$ due to naturality of $\zeta$, \eqref{eq:27},
  \eqref{eq:36}, and the fact that $\delta$ is a comonoid morphism for
  $(\Delta, \mathsf{e})$.
\end{proof}

\begin{Lemma}
  \label{lem:13}
  $(\ocf, \delta^{\partial}, \varepsilon^\partial)$ satisfies the
  counit law $\varepsilon^\partial_{\ocf X} \circ
  \delta^\partial_X = 1_{\ocf X}$ for all $X \in \C$.
\end{Lemma}
\begin{proof}
  Observe that we have a diagram of
  commuting $X$-actions
  \begin{equation*}
    \cd[@C-1em]{
      {(\ocf X, \mathsf{d}^\partial_X)} \ar[rr]^-{\delta^\partial_X} \ar@{=}[dr] & &
      {\oc^{\partial X}(\ocf X, \mathsf{d}^\partial_X)\rlap{ ,}} \ar[dl]^-{\varepsilon^\partial_{\ocf X}} \\ &
      {(\ocf X, \mathsf{d}^\partial_X)}
    }
  \end{equation*}
  where the top edge is a commuting action by~\eqref{eq:34}, and the
  right edge is one by Lemma~\ref{lem:9}(i). Moreover, this triangle
  commutes on precomposition by $\zeta_X$ due to naturality of
  $\zeta$, \eqref{eq:6} and \eqref{eq:36}; whence, by universality of
  $\zeta_X$, the result.
\end{proof}

\begin{Lemma}
  \label{lem:17}
  $\delta^\partial$ satisfies the chain rule for $\mathsf{d}^\partial$.
\end{Lemma}
\begin{proof}
  By Lemma~\ref{lem:7} together with~\eqref{eq:34}.
\end{proof}

All that remains to complete our construction is to verify the
other two comonad identities.

\begin{Lemma}\label{lem:initial-comonad}
  $(\ocf, \delta^{\partial}, \varepsilon^{\partial})$ is a comonad.
\end{Lemma}
\begin{proof}
  For the second counit law, we may argue as in Lemma~\ref{lem:13} and
  reduce to the problem of showing that the composite map:
  \begin{equation*}
    \ocf\varepsilon^\partial_{X} \circ \delta^\partial_X \colon {(\ocf X, \mathsf{d}^\partial_X)} \rightarrow {(\ocf X, \mathsf{d}^\partial_X)}
  \end{equation*}
  is a map of commuting $X$-actions. Then, as it precomposes with
  $\zeta_X$ to give $\zeta_X$, it must be the identity. Thus, consider
  the following diagram:
  \begin{equation*}
    \cd[@C-0.6em@R-0.4em]{
      \ocf X . X
      \ar[ddd]_-{\mathsf{d}^\partial_X}
      \ar@{=}[rr]
      \ar[dr]_-{\Delta^\partial_X.1} & &
      \ocf X . X
      \ar[r]^-{\delta^\partial_X.1} &
      \ocf \ocf X . X
      \ar[r]^-{\ocf \varepsilon^\partial_X.1} &
      \ocf X . X
      \ar[ddd]^-{\mathsf{d}^\partial_X}\\
      & \ocf X . \ocf X . X
      \ar[ur]|-{1.\mathsf{e}^\partial_X.1}
      \ar[r]_-{\delta^\partial_X.1.1} &
      \ocf \ocf X . \ocf X . X
      \ar[ur]|-{1.\mathsf{e}^\partial_X.1}
      \ar[dr]_-{1.\mathsf{d}^\partial_X}\\
      & & & \ocf \ocf X . \ocf X
      \ar[uu]|-{1.\varepsilon^\partial_X}
      \ar[d]|-{\mathsf{d}^\partial_{\ocf X}}
      \\ 
      \ocf X \ar[rrr]^-{\delta^\partial_X} & & & \ocf \ocf X  \ar[r]^-{\ocf \varepsilon} & \ocf X
    }
  \end{equation*}
  The top-left region is a counit axiom for $(\Delta^\partial,
  \mathsf{e}^\partial)$; the bottom-left region is the chain rule; the
  triangular region is the linear rule; the right region is naturality
  of $\mathsf{d}^\partial$; and the remaining region is
  bifunctoriality of $\otimes$. Thus, the outside commutes as
  required.
  Finally, for the coassociativity law, note that every edge in the
  following diagram is a map of commuting $X$-actions:
  \begin{equation}
    \label{eq:37}
    \cd{
      {(\ocf X, \mathsf{d}^\partial_X)} \ar[r]^-{\delta^\partial_X} \ar[d]_{\delta^\partial_X} &
      {\oc^{\partial X}(\ocf X, \mathsf{d}^\partial_X)} \ar[d]^{\ocf \delta^\partial_X} \\
      {\oc^{\partial X}(\ocf X, \mathsf{d}^\partial_X)} \ar[r]^-{\delta^\partial_{\ocf X}} &
      {\oc^{\partial X}\oc^{\partial X}(\ocf X, \mathsf{d}^\partial_X)}\rlap{ .}
    }
  \end{equation}
  For the top and left edges, this is direct from~\eqref{eq:34}; for
  the right and bottom edge, it follows from~\eqref{eq:34} together with functoriality of
  $\oc^{\partial X}$, respectively 
  Lemma~\ref{lem:9}(ii). Moreover, by naturality of $\zeta$ together
  with~\eqref{eq:36}, the square commutes on precomposition with
  $\zeta_X$; whence it commutes, as desired.
\end{proof}

We have thus proved:

\begin{Prop}
  $(\ocf, \delta^{\partial}, \varepsilon^{\partial},
  \Delta^{\partial}, \mathsf{e}^{\partial}, \mathsf{d}^{\partial} )$
  is a differential modality.
\end{Prop}

\subsection{Freeness}
\label{sec:freeness}

To complete this section, we prove that the differential modality of
the preceding proposition is free over $\oc$:

 \begin{Prop} $\ocf$ is the free differential modality on $\oc$. Explicitly: 
 \begin{enumerate}[(i)]
\item $\zeta\colon \oc \rightarrow \ocf$ is a coalgebra modality morphism;
\item For any differential modality $\oc'$ and coalgebra modality morphism $\psi\colon \oc \rightarrow \oc'$, there is a unique differential modality morphism $\psi^\sharp\colon \ocf \rightarrow \oc'$ with
  \begin{equation}
    \label{eq:40}
    \psi^\sharp \circ \zeta = \psi\rlap{ .}
  \end{equation}
\end{enumerate}
 \end{Prop}
 \begin{proof}
     For (i), $\zeta$ is a coalgebra modality morphism by virtue of 
     commutativity in~\eqref{eq:27}, \eqref{eq:6} and \eqref{eq:36}.
     For (ii), suppose we have a differential modality $\oc'$ and
     coalgebra modality morphism $\psi \colon \oc \rightarrow \oc'$.
     By Lemma~\ref{lem:8}, the deriving transformation of $\oc'$
     endows each $\oc' X$ with a commuting $X$-action
     $(\oc' X, \mathsf{d}'_X)$; and as differential modality
     morphisms preserve deriving transformations, the
     components of any extension $\psi^\sharp$ of $\psi$ must underlie
     maps of commuting $X$-actions:
     \begin{equation}
       \label{eq:39}
       \psi^\sharp_X \colon (\ocf X, \mathsf{d}^\partial_X) \rightarrow (\oc' X, \mathsf{d}'_X)\rlap{ .}
     \end{equation}
     Now, by the universal property of $\ocf X$, a map of the
     form~\eqref{eq:39} is determined by its precomposition with
     $\zeta_X$; but this precomposite is forced by~\eqref{eq:40}.
     Thus, there is \emph{at most} one possible choice for the
     components $\psi^\sharp_X$, given by the unique maps~\eqref{eq:39}
     satisfying~\eqref{eq:40}.

     It remains to show that these components do yield a differential
     modality morphism $\psi^\sharp \colon \ocf \rightarrow \oc'$.
     Clearly, $\psi^\sharp$ preserves $\mathsf{d}$ by~\eqref{eq:39}.
     The remaining conditions that $\psi^\sharp$ must satisfy are
     naturality, and preservation of the structure maps $\mathsf{e}$,
     $\Delta$, $\varepsilon$ and $\delta$. These conditions are
     expressed by the commutativity of the following squares, each of
     whose sides, we note, is a commuting $X$-action map:
     \begin{gather*}
       \cd{
         {(\ocf X, \mathsf{d}^\partial_X)} \ar[r]^-{\ocf f} \ar[d]_{\psi^\sharp_X} &
         {f^\ast(\ocf Y, \mathsf{d}^\partial_Y)} \ar[d]^{\psi^\sharp_Y} &
         {(\ocf X, \mathsf{d}^\partial_X)} \ar[r]^-{(\mathsf{e}^\partial_X, \varepsilon^\partial_X)} \ar[d]_{\psi^\sharp_X} &
         {(I\oplus X, \iota_1 \circ (\pi_0.X))} \ar@{=}[d] \\
         {(\oc' X, \mathsf{d}'_X)} \ar[r]^-{\oc' f} &
         {f^\ast(\oc' Y, \mathsf{d}'_Y)} &
         {(\oc' X, \mathsf{d}'_X)} \ar[r]^-{(\mathsf{e}'_X, \varepsilon'_X)} &
         {(I\oplus X, \iota_1 \circ (\pi_0.X))}
       }\\
       \cd{
         (\ocf X, \mathsf{d}^\partial_X)
         \ar[r]^-{\Delta_X^\partial}
         \ar[d]_-{\psi^\sharp_X} &
         (\ocf X, \mathsf{d}^\partial_X) \boxtimes (\ocf X, \mathsf{d}^\partial_X) \ar[d]_-{\psi^\sharp_X \boxtimes \psi^\sharp_X} &
                  (\ocf X, \mathsf{d}^\partial_X)
         \ar[r]^-{{\mathsf{\delta}_X^\partial}}
         \ar[d]_-{\psi^\sharp_X} &
         \oc^{\partial X}(\ocf X, \mathsf{d}^\partial_X)  \ar[d]_-{(\psi^\sharp \psi^\sharp)_X} \\
         (\oc' X, \mathsf{d}'_X) \ar[r]^-{\Delta_X'} &
         (\oc' X, \mathsf{d}'_X) \boxtimes(\oc' X, \mathsf{d}'_X) &
         (\oc' X, \mathsf{d}'_X) \ar[r]^-{\delta_X'} &
         \oc^{'X}(\oc' X, \mathsf{d}'_X) \rlap{ .}
       }
     \end{gather*}
     Now, if we paste these four squares by, respectively---the
     naturality square for $\zeta$ at $f$; the right square
     of~\eqref{eq:6}; the right square of~\eqref{eq:27};
     and~\eqref{eq:36}---then we obtain the corresponding squares for
     the coalgebra modality map $\psi \colon \oc \rightarrow \oc'$,
     which thus commute. By the universal property of $\ocf X$, it the
     above squares must themselves commute, as required.
\end{proof}

We summarise our results in the first main result of this paper:
\begin{Thm}
  \label{thm:1}
  Let $(\C, \otimes, I)$ be a $k$-linear symmetric monoidal category
  with finite biproducts which admits algebraically-free commutative
  monoids
  \begin{equation*}
    (SX, \, \mathsf{d}^S_X \colon SX \otimes X \rightarrow X, \, \mathsf{u}^S_X
    \colon X \rightarrow SX)\rlap{ .}
  \end{equation*}
  For any coalgebra modality
  $(\oc, \delta, \varepsilon, \Delta, \mathsf{e})$ on $\C$, the free
  differential modality $\ocf$ on $\oc$ exists, with structure
  $(\ocf, \delta^\partial, \varepsilon^\partial, \Delta^\partial,
  \mathsf{e}^\partial, \mathsf{d}^\partial)$ given as follows:
  \begin{itemize}[itemsep=0.25\baselineskip]
  \item The \textbf{functor} $\ocf$ is given by $\ocf X = \oc X
    \otimes SX$ and $\ocf f = \oc f \otimes Sf$, where $Sf$ is the
    unique map rendering commutative the following diagram:
    \begin{equation*}
    \cd[@!@C-1em@-2.5em@R-1em]{
      & I \ar[dl]_-{\mathsf{u}^S_X} \ar[dr]^-{\mathsf{u}^S_Y} \\
      SX \ar[rr]^-{Sf} & & SY \rlap{ .} \\
      SX \otimes X \ar[u]^-{\mathsf{d}^S_X} \ar[rr]^-{Sf \otimes f} & &
      SY \ar[u]_-{\mathsf{d}^S_Y}
    }
  \end{equation*}
  
  \item The \textbf{deriving transformation} maps
    $\mathsf{d}^\partial_X \colon \ocf X \otimes X \rightarrow \ocf X$
    are given by:
    \begin{equation*}
      \oc X \otimes \mathsf{d}^S_X \colon \oc X \otimes SX \otimes X \rightarrow \oc X \otimes SX\rlap{ .}
    \end{equation*}

  \item The \textbf{unit map} $\zeta \colon \oc \rightarrow \ocf$
    exhibiting the freeness has components:
    \begin{equation*}
      \oc X \otimes \mathsf{u}^S_X \colon \oc X \rightarrow \oc X \otimes SX\rlap{ .}
    \end{equation*}

  \item The \textbf{comonoid} structure maps $\mathsf{e}^\partial_X \colon \ocf X
    \rightarrow I$ and $\Delta^\partial_X \colon \ocf X \rightarrow \ocf X
    \otimes \ocf X$ are the unique maps rendering commutative the
    following diagrams:
    \begin{equation*}
      \cd[@C-1em]{
        \oc X \ar[d]_-{\zeta_X} \ar[rr]^-{\mathsf{e}_X} &&
        I \ar@{=}[d] &&
        \oc X \ar[d]_-{\zeta_X} \ar[r]^-{ \Delta_X} &
        \oc X \otimes \oc X \ar[d]^-{ \zeta_X \otimes \zeta_X }    \\
        \ocf X \ar[rr]^-{\mathsf{e}^{\partial}_X} && I &&
        \ocf X \ar[r]^-{\Delta^{\partial}_X} &
        \sh{r}{0.2em}\ocf X \otimes \ocf X}
    \end{equation*}
    and satisfying the constant and product rule for
    $\mathsf{d}^\partial$.

  \item The \textbf{comonad} structure maps
    $\varepsilon^\partial_X \colon \ocf X \rightarrow X$ and
    $\delta^\partial_X \colon \ocf X \rightarrow \ocf \ocf X$ are the unique
    maps rendering commutative the following diagrams:
    \begin{equation*}
  \cd[@C-0.5em]{
    \oc X \ar[d]_-{\zeta_X} \ar[rr]^-{\varepsilon_X} &&
    X\ar@{=}[d] &&
    \oc X \ar[d]_-{\zeta_X} \ar[r]^-{ \delta_X} &
    \oc \oc X \ar[d]^-{ (\zeta\zeta)_X }    \\
    \ocf X \ar[rr]^-{\varepsilon^\partial_X} && X &&
    \ocf X \ar[r]^-{\delta^{\partial}_X} & \ocf \ocf X }
\end{equation*}
  and satisfying the linear and chain rule for $\mathsf{d}^\partial$.
  \end{itemize}
\end{Thm}
We now provide more explicit descriptions of the structure maps
$\mathsf{e}^\partial$, $\Delta^\partial$, $\varepsilon^\partial$ and
$\delta^\partial$ in the preceding result. For $\mathsf{e}^\partial$
and $\Delta^\partial$, we recall from Section~\ref{sec:additive-case}
that each commutative monoid $SX$ in $\C$ is a bicommutative bialgebra with comonoid structure determined as
in~\eqref{eq:8}. We write this comonoid structure as
$\mathsf{e}^S_X \colon SX \rightarrow I$ and
$\Delta^S_X \colon SX \rightarrow SX \otimes SX$ to distinguish it
from the structure maps of $\oc$ and $\ocf$.

\begin{Lemma}
  \label{lem:18}
  In the situation of Theorem~\ref{thm:1}, the morphisms
  $\mathsf{e}^\partial_X \colon \ocf X \rightarrow I$ and
  $\Delta^\partial_X \colon \ocf X \rightarrow \ocf X \otimes \ocf X$
  of the free differential modality on $\oc$
  are given by:
  \begin{equation*}
    \begin{aligned}
      \oc X \otimes SX & {\xrightarrow{\mathsf{e}_X \otimes \mathsf{e}^S_X} I} \ \ \ \text{and}\\[5pt]
      \oc X \otimes SX &\smash{\xrightarrow{\Delta_X \otimes
          \Delta^S_X} \oc X \otimes \oc X \otimes SX \otimes SX
        \xrightarrow{1 \otimes \sigma \otimes 1} \oc X \otimes SX
        \otimes \oc X \otimes SX}\rlap{ .}
    \end{aligned}
  \end{equation*}
\end{Lemma}
\begin{proof}
  As we observed in Section~\ref{sec:additive-case}, if $B$ is a
  bialgebra in any symmetric monoidal category $\V$, then the monoidal
  structure of $\V$ lifts to one on the category of right $B$-modules.
  It is not hard to show that, with respect to this structure, the
  comonoid structure maps of $B$ lift to maps of right $B$-modules
  $\Delta^B \colon (B, \nabla^B) \rightarrow (B, \nabla^B)
  \tilde{\otimes} (B, \nabla^B)$ and
  $\mathsf{e}^B \colon (B, \nabla^B) \rightarrow \tilde I$.

  In particular, this is true for the bialgebra $SX$ in $\C$, and so,
  transporting under the monoidal equivalence
  $(\mod, \tilde \otimes) \simeq (\ca, \boxtimes)$, we conclude that
  the following are maps of commuting $X$-actions:
  \begin{equation*}
    \mathsf{e}^S_X \colon (SX, \mathsf{d}^S_X) \rightarrow (I,0)  \qquad \text{and} \qquad  \Delta^S_X \colon (SX, \mathsf{d}^S_X) \rightarrow (SX, \mathsf{d}^S_X) \boxtimes (SX, \mathsf{d}^S_X)\rlap{ .}
  \end{equation*}
  Since $(\ocf X, \mathsf{d}^\partial_X) \cong (X,0) \boxtimes (SX,
  \mathsf{d}^S_X)$, we easily conclude that the maps in the statement of
  this lemma
  are maps of commuting $X$-actions
  \begin{equation*}
    (\ocf X, \mathsf{d}^\partial_X)  \xrightarrow{\ \ } (I,0)  \qquad \text{and} \qquad  (\ocf X, \mathsf{d}^\partial_X) \xrightarrow{\ \ } (\ocf X, \mathsf{d}^\partial_X) \boxtimes (\ocf X, \mathsf{d}^\partial_X)\rlap{ .}
  \end{equation*}
  Moreover, they easily render the squares of~\eqref{eq:27}
  commutative, and so satisfy both of the properties which uniquely
  characterise $\mathsf{e}^\partial_X$ and
  $\Delta^\partial_X$.
\end{proof}

We now consider the maps $\varepsilon^\partial$. For this, we use the
fact, noted in Example~\ref{ex:6}, that $S^\mathrm{op}$ is a monoidal
differential modality on $\C^\mathrm{op}$, whose codereliction is, by
Example~\ref{ex:4} and the proof of Proposition~\ref{prop:3}, the
second component of the unique map of commuting $X$-actions
\begin{equation*}
  \E^S_X \defeq (\mathsf{e}^S_X, \varepsilon^S_X) \colon (SX, \mathsf{d}^S_X) \rightarrow (I\oplus X, \iota_1 \circ (\pi_0 \otimes X))
\end{equation*}
such that $\E_X^S \circ \mathsf{u}^S_X = \iota_0 \colon I \rightarrow I
\oplus X$. In particular, this means that $\varepsilon^S_X$ is characterised
the equalities
\begin{equation}
  \label{eq:55}
  \varepsilon^S_X \circ \mathsf{u}^S_X = 0 \qquad \text{and} \qquad \varepsilon^S_X \circ \mathsf{d}^S_X = \mathsf{e}^S_X \otimes X\rlap{ .}
\end{equation}

\begin{Lemma}
  \label{lem:22}
  In the situation of Theorem~\ref{thm:1}, the map
  $\varepsilon^\partial_X \colon \ocf X \rightarrow X$ is given by
  ${\mathsf{e}_X \otimes \varepsilon^S_X\, +\, \varepsilon_X \otimes
    \mathsf{e}^S_X} \colon \oc X \otimes SX \rightarrow X$.
\end{Lemma}
\begin{proof}
  By Definition~\ref{def:16}, $\varepsilon^\partial_X$ is the unique
  map $h \colon \ocf X \rightarrow X$ with
  $h \circ \zeta_X = \varepsilon_X$ and
  $ h \circ \mathsf{d}^\partial_X = \mathsf{e}^\partial_X \otimes X$.
  So it suffices to verify these two equalities taking
  $h$ to be $\mathsf{e}_X \otimes \varepsilon^S_X\, +\, \varepsilon_X
  \otimes \mathsf{e}^S_X$. For the first:
  \begin{align*}
    (\mathsf{e}_X
    \otimes \varepsilon^S_X\, +\, \varepsilon_X \otimes
    \mathsf{e}^S_X) \circ (SX \otimes \mathsf{u}^S_X) &=
    \mathsf{e}_X
    \otimes (\varepsilon^S_X \circ \mathsf{u}^S_X)\, +\, \varepsilon_X \otimes
    (\mathsf{e}^S_X \circ \mathsf{u}^S_X) \\
    &= 0 +\, \varepsilon_X = \varepsilon_X
  \end{align*}
  where from the first to the second line we use the 
  left equation in~\eqref{eq:55} and the bialgebra axiom
  $\mathsf{e}^S_X \circ \mathsf{u}^S_X = \mathrm{id}_I$. For the second
  equality we have:
  \begin{align*}
    (\mathsf{e}_X
    \otimes \varepsilon^S_X\, +\, \varepsilon_X \otimes
    \mathsf{e}^S_X) \circ (SX \otimes \mathsf{d}^S_X)  &=
    \mathsf{e}_X
    \otimes (\varepsilon^S_X \circ \mathsf{d}^S_X)\, +\, \varepsilon_X \otimes
    (\mathsf{e}^S_X \circ \mathsf{d}^S_X) \\
    &= \mathsf{e}_X \otimes \mathsf{e}^S_X \otimes X + 0 = \mathsf{e}^\partial_X \otimes X\rlap{ ,}
  \end{align*}
  where from the first to the second line we use the 
  right equation in~\eqref{eq:55} and the fact that $\mathsf{e}_X^S
  \colon (SX, \mathsf{d}^S_X) \rightarrow (I,0)$; and for the equality
  on the second line we use the previous lemma.
\end{proof}
Finally, we have $\delta^\partial$. It does not seem to be possible to
give a closed-form expression for this in terms of the structure maps
of $\oc$ and $S$; thus, in examples, the best we can do is to work
from the definition (Definition \ref{def:17}): thus, we first compute
the $X$-action structure on $\oc^{\partial X}(\ocf X,
\mathsf{d}^\partial X)$, and then find $\delta^\partial_X$ as the
unique map of $X$-actions $(\ocf X,
\mathsf{d}^\partial X) \rightarrow \oc^{\partial X}(\ocf X,
\mathsf{d}^\partial X)$ rendering commutative~\eqref{eq:36}.

\begin{Rk}
  \label{rk:1}
  Under the assumptions of Theorem~\ref{thm:1}, we have shown that the
  forgetful functor
  $U \colon \cat{DiffMod}(\C) \rightarrow \cat{CoalgMod}(\C)$ has a
  left adjoint $F$. It is not hard to see that this functor also
  creates coequalisers for $U$-absolute pairs; whence it is monadic.
  This allows for a more abstract explanation of the fact that
  the cofree cocommutative comonoid comonad $\oc^C$ is a differential
  modality, in a unique way. Indeed, $\oc^C$ is a terminal object in
  the category $\cat{CoalgMod}(\C)$, so bears a unique structure of
  $UF$-algebra, and so has a unique deriving transformation making it
  a differential modality.
\end{Rk}

\section{The free monoidal differential modality on a monoidal
  coalgebra modality}
\label{sec:free-mono-diff}
In this section, we continue to work in a $k$-linear symmetric category
$(\C, \otimes, I)$ which admits algebraically-free commutative monoids
and finite biproducts, and give the second main result of the paper by
constructing the free \emph{monoidal} differential modality on a
\emph{monoidal} coalgebra modality.
In other words, we seek a left adjoint to the forgetful functor
$\cat{MonDiffMod}(\C) \rightarrow \cat{MonCoalgMod}(\C)$.

In the previous section, we found a left adjoint to
$\cat{DiffMod}(\C) \rightarrow \cat{CoalgMod}(\C)$; but since, by
Proposition~\ref{prop:4}, $\cat{MonDiffMod}(\C)$ is a full subcategory of
$\cat{DiffMod}(\C)$ and $\cat{MonCoalgMod}(\C)$ is a full subcategory of
$\cat{CoalgMod}(\C)$, it suffices to show:

\begin{Prop}
  \label{prop:7}
  In the situation of Theorem~\ref{thm:1}, the free differential
  modality $\ocf$ on a coalgebra modality $\oc$ is monoidal whenever
  $\oc$ is so.
\end{Prop}
\begin{proof}
  As noted in Section~\ref{sec:additive-case}, the free commutative
  monoid monad $S$ on $\C$ is an example of a monoidal coalgebra
  modality on $\C^\mathrm{op}$. In particular, this means that the
  storage maps in $\C^\mathrm{op}$ are invertible, with inverses
  given as in~\eqref{eq:17}, by the costorage maps of
  $\C^\mathrm{op}$. Dualising back to $\C$, we see that the storage
  maps $\chi^S_0$ and $\chi^S_{XY}$ for $S$ are given by:
  \begin{equation*}
    S0 \xrightarrow{\!\mathsf{e}^S_0\!} I \text{ and}  \ 
    S(X \oplus Y) \xrightarrow{\!\Delta^S_{X\oplus Y}\!} S(X \oplus Y) \otimes S(X\oplus Y) \xrightarrow{\!S\pi_0 \otimes S\pi_1\!} SX \otimes SY
  \end{equation*}
  and so are invertible in $\C$.
  On the other hand, since $\oc$ is a monoidal coalgebra modality in
  $\C$, \emph{its} storage maps $\chi_0$ and $\chi_{XY}$, given by:
  \begin{equation*}
    \smash{\oc 0 \xrightarrow{\mathsf{e}_0} I \text{ and} } \ 
    \smash{\oc (X \oplus Y) \xrightarrow{\Delta_{X\oplus Y}} \oc(X \oplus Y) \otimes \oc(X\oplus Y) \xrightarrow{\oc\pi_0 \otimes \oc\pi_1} \oc X \otimes \oc Y}
  \end{equation*}
  are also invertible in $\C$. But now by Lemma~\ref{lem:18}, the storage
  maps for $\ocf$ must be:
  \begin{align*}
    \oc 0 \otimes S0 &\mathrel{\smash{\xrightarrow{\mathsf{e}_0 \otimes \mathsf{e}_0^S}}} I \ \ \ \text{ and }\\[-4pt]
    \oc (X \oplus Y) \!\otimes\! S (X \oplus Y) &\xrightarrow{\chi_{X,Y} \otimes \chi^S_{X,Y}} \oc X \!\otimes\! \oc Y \!\otimes\! SX \!\otimes\! SY \xrightarrow{1 \otimes \sigma \otimes 1} \oc X \!\otimes\! SX \!\otimes\! \oc Y \!\otimes\! SY
  \end{align*}
  and so invertible, as desired.
\end{proof}

\begin{Thm}
  \label{thm:2}
  Let $(\C, \otimes, I)$ be a $k$-linear symmetric monoidal category
  with finite biproducts which admits algebraically-free commutative
  monoids
  \begin{equation*}
    (SX, \, \mathsf{d}^S_X \colon SX \otimes X \rightarrow X, \, \mathsf{u}^S_X
    \colon X \rightarrow SX)\rlap{ .}
  \end{equation*}
  For any monoidal coalgebra modality
  $\oc$ on $\C$, the free
  monoidal differential modality $\ocf$ on $\oc$ exists, with
  differential modality structure given as in Theorem~\ref{thm:1}.
\end{Thm}
\begin{Rk}
  \label{rk:7}
  As in Remark~\ref{rk:1}, when the assumptions of this theorem are satisfied
  the forgetful functor $\cat{MonDiffMod}(\C) \rightarrow
  \cat{MonCoalgMod}(\C)$ will not only have a left adjoint, but will
  be monadic.
\end{Rk}

Since, by Proposition~\ref{prop:9}, a differential modality is a
monoidal differential modality in at most one way, the above result is
enough to determine completely the structure of the monoidal
differential modality $\ocf$; however, as in the previous section, it
is convenient to have, insofar as possible, a direct description of
the additional structure maps associated to a monoidal differential
modality: thus, the algebra structure maps $\mathsf{u}^\partial_X$ and
$\nabla^\partial_X$, the codereliction $\eta^\partial_X$, and the
monoidal constraint maps $m_I^\partial$ and $m^\partial_{XY}$.

\begin{Lemma}
  \label{lem:23}
  In the situation of Theorem~\ref{thm:2}, the morphisms $\mathsf{u}^\partial_X \colon I
\rightarrow \ocf X$ and $\nabla^\partial_X \colon \ocf X \otimes \ocf
X \rightarrow \ocf X$ are given by
\begin{equation*}
    \begin{aligned}
      I  & {\xrightarrow{\mathsf{u}^S_X \otimes \mathsf{u}^S_X} \oc X \otimes SX} \ \ \ \text{and}\\[5pt]
      \oc X \otimes SX
        \otimes \oc X \otimes SX &\smash{\xrightarrow{1 \otimes \sigma \otimes 1} \oc X \otimes \oc X \otimes SX \otimes SX
        \xrightarrow{\nabla_X \otimes
          \nabla^S_X} \oc X \otimes SX}\rlap{ .}
    \end{aligned}
  \end{equation*}
\end{Lemma}
\begin{proof}
  We can define $\mathsf{u}_X^\partial$ and $\nabla_X^\partial$ in
  terms of the inverses of the storage maps:
  \begin{equation}
    \label{eq:44}
    \begin{aligned}
      \mathsf{u}_X^\partial &= I \xrightarrow{(\chi_0^\partial)^{-1}} \ocf 0 \xrightarrow{\ocf 0} \ocf X\\[-4pt]
      \nabla_X^\partial &= \ocf X \otimes \ocf X
      \xrightarrow{(\chi_{XX}^\partial)^{-1}} \ocf (X \oplus X)
      \xrightarrow{\ocf \spn{\mathrm{id}, \mathrm{id}}} \ocf X\rlap{
        .}
    \end{aligned}
  \end{equation}
  But since $\chi_0^\partial = \chi_0 \otimes \chi_0^S$ and
  $\chi_{XX}^\partial = (1 \otimes \sigma \otimes 1) \circ (\chi_{XX}
  \otimes \chi_{XX}^S)$, we have
  $(\chi_0^\partial)^{-1} = \chi_0^{-1} \otimes (\chi_0^S)^{-1}$ and
  $(\chi_{XX}^\partial)^{-1} = \bigl((\chi_{XX})^{-1} \otimes
  (\chi_{XX^S})^{-1}\bigr) \circ (1 \otimes \sigma \otimes 1)$.
  Substituting these into~\eqref{eq:44}, and noting that the maps
  $\mathsf{u}_X$, $\mathsf{u}^S_X$, $\nabla_X$ and $\nabla_X^S$ admit
  corresponding characterisations to~\eqref{eq:44}, the result follows.
\end{proof}
In the next result, we use the unit maps of the free commutative monoid
monad $\eta^S_X \colon X \rightarrow SX$, given as
in~\eqref{eq:19} by
$\eta^S_X \defeq \mathsf{d}^S_X \circ (\mathsf{u}^S_X \otimes X)$; note
that these are equally the dereliction maps of $S$ qua monoidal
coalgebra modality on $\C^\mathrm{op}$.

\begin{Lemma}
  \label{lem:24}
  In the situation of Theorem~\ref{thm:2}, the morphism
  $\eta^\partial_X \colon X \rightarrow \ocf X$ is given by the composite
  \begin{equation*}
    X \xrightarrow{\mathsf{u}_X \otimes \eta^S_X} \oc X \otimes X
  \end{equation*}
\end{Lemma}
\begin{proof}
  This is immediate from~\eqref{eq:56}, \eqref{eq:19}, and the previous lemma.
\end{proof}

Note that this formula is \emph{not} dual to the formula for the
dereliction maps $\varepsilon^\partial_X$ of Lemma~\eqref{lem:22};
indeed, these were given not simply by $\mathsf{e}_X \otimes \varepsilon^S_X$
but rather by the sum
$\mathsf{e}_X \otimes \varepsilon^S_X + \varepsilon_X \otimes
\mathsf{e}^S_X$.

For the nullary monoidal constraint map, we have that $m^\partial_I =
\xi_I \circ m_I$ since
$\xi \colon \oc \rightarrow \ocf$ is a monoidal comonad morphism; and so
\begin{equation}
  \label{eq:58}
  m^\partial_I = I \xrightarrow{m_I \otimes \mathsf{u}_I} \oc I \otimes SI \rlap{ .}
\end{equation}

For the binary monoidal constraint map $m^\partial_{XY}$, there seems
to be no efficient direct description; the best we can do is to
use~\eqref{eq:42}, which, bearing in mind that the inverses of the
storage maps are as in~\eqref{eq:17}, we can express as follows.
First, define the auxiliary map $n_X$ to be:
\begin{equation*}
  \ocf X \otimes \ocf X \xrightarrow{\nabla^\partial_X} \ocf X \xrightarrow{\delta^\partial_X} \ocf \ocf X \xrightarrow{\ocf \Delta^\partial_X} \ocf (\ocf X \otimes \ocf X) \xrightarrow{\ocf(\varepsilon^\partial_X \otimes \varepsilon^\partial_X)} \ocf(X \otimes X)\rlap{ .}
\end{equation*}
Now $m^\partial_{XY}$ is given by the composite
\begin{equation*}
  \ocf X \otimes \ocf Y \!\xrightarrow{\!\ocf\iota_0\otimes\ocf\iota_1\!}\!
  \ocf (X\oplus Y )\otimes \ocf (X\oplus Y) \!\xrightarrow{\!n_{X \oplus Y}\!}\!
  \ocf((X \oplus Y) \otimes (X\oplus Y)) \!\xrightarrow{\!\ocf(\pi_0 \otimes \pi_1)\!}\! \ocf(X \otimes Y)\rlap{ .}
\end{equation*}

\section{The initial monoidal differential modality}
\label{sec:init-mono-diff}
In this section, we give our third main result, which exploits our
construction of the free monoidal differential modality in order to
describe the \emph{initial monoidal differential modality} on a
suitable symmetric monoidal category $\C$. Note that the problem of
describing the initial differential modality (not necessarily
monoidal) is trivial: it is the comonad constant at the zero object.
However, when we add monoidality, things become more interesting.

The assumptions under which we will make our construction are that:
\begin{enumerate}[(a)]
\item $(\C, \otimes, I)$ is a $k$-linear symmetric monoidal category with biproducts;   \label{assumption1}
\item $\C$ admits algebraically-free commutative monoids; \label{assumption2}
\item $\C$ admits all set-indexed coproducts of its unit object $I$,
  and these distribute over the tensor product. \label{assumption3}
\end{enumerate}
Assumption~\ref{assumption3} is just what we need to construct the
initial monoidal coalgebra modality $\oc$ on $\C$;
then~\ref{assumption1} and~\ref{assumption2} allow us to build the
free monoidal differential modality $\ocf$ on this.
Note that, in particular, if $\C$ is $k$-linear symmetric monoidal
closed and cocomplete, then all of these assumptions will hold; this
is immediate for~\ref{assumption1} and~\ref{assumption3}, while
for~\ref{assumption2}, it follows from Lemma~\ref{lem:2} and
Proposition~\ref{prop:2}.

\subsection{The initial monoidal coalgebra modality}
\label{sec:init-mono-coalg}

In a symmetric monoidal category $(\C, \otimes, I)$, recall that a
\emph{point} of an object $X$ is a map $I \to X$, and the ``set of
points of $X$'' is the homset $\C(I,X)$. This gives us a functor
$\C(I, \thg)\colon \C \to \SET$ which is symmetric (lax) monoidal as a functor
$(\C, \otimes) \rightarrow (\SET, \times)$ via the maps
\begin{align*}
  1 & \rightarrow \C(I,I) & \C(I, X) \times \C(I,Y) &\rightarrow \C(I, X \otimes Y)\rlap{ .}\\
  \ast & \mapsto 1_I & (x,y) & \mapsto x \otimes y\rlap{ .}
\end{align*}
The assumption~\ref{assumption3} of all set-indexed coproducts of
$I$ yields a left adjoint $L \colon \SET \rightarrow \C$ to
$\C(I, \thg)$, which on objects sends $X$ to $\bigoplus_{x: I \to X} I$.
The monoidal structure of $\C(I, \thg)$ yields an \emph{op}monoidal
structure on $L$, with nullary comparison map
$n_I \colon \bigoplus_{\ast: I \to 1} I \rightarrow I$ the evident
isomorphism (where $\ast: I \to 1$ is the unique map to the terminal object), and binary comparison map
\begin{equation*}
  n_{X,Y} \colon \bigoplus_{\langle x,y \rangle: I \to X \times Y} \!\!I \longrightarrow \Bigl(\bigoplus_{x: I \to X} I\Bigr) \otimes \Bigl(\bigoplus_{y: I \to Y} I\Bigr)
\end{equation*}
determined by the requirement that
$n_{X,Y} \circ \iota_{x,y} = \iota_x \otimes \iota_{y}$, where $\iota_{(-)}$ is the injection map into the coproduct. Now, the
assumption in~\ref{assumption3} that coproducts of the unit distribute over tensor
product ensures that the maps $n_{X,Y}$ are also invertible. Thus, $L$
is \emph{strong} symmetric monoidal, and so we have an adjunction
\begin{equation}
  \label{eq:41}
  \cd[@C+1em]{
    {(\C, \otimes, I)} \ar@<-4.5pt>[r]_-{\C(I, \thg)} \ar@{<-}@<4.5pt>[r]^-{L} \ar@{}[r]|-{\bot} &
    {(\cat{Set}, \times, 1)} 
  }
\end{equation}
in the $2$-category of symmetric monoidal categories, symmetric
monoidal functors, and monoidal natural transformations.

Like any such adjunction,~\eqref{eq:41} induces a monoidal comonad $P$
(for ``points'') on the symmetric monoidal category $(\C, \otimes, I)$
at the codomain of the left adjoint. However, the fact that the
monoidal category $(\cat{Set}, \times, 1)$ at the other end is
\emph{cartesian} monoidal means that this monoidal comonad is, in
fact, a monoidal coalgebra modality. Indeed, \eqref{eq:41} is a
so-called \emph{linear--non-linear adjunction} and so this follows
from~\cite[Theorem~3]{Benton1995A-mixed}; the main point is that
the cocommutative comonoid structure of each $PX = L(\C(I, X))$ is the
image of the unique cartesian comonoid structure on $\C(I, X)$ under
the comonoid-preserving functor $L$.

\begin{Prop}
  \label{prop:8}
  Let the symmetric monoidal category $(\C, \otimes, I)$ satisfy
  hypothesis~\ref{assumption3}. The monoidal coalgebra modality $P$
  induced by~\eqref{eq:41} is the initial monoidal coalgebra modality
  on $(\C, \otimes, I)$.
\end{Prop}
\begin{proof}
  Let $\oc$ be another monoidal coalgebra modality on $\C$; we will show
  that there is a unique morphism of monoidal coalgebra modalities
  $P \rightarrow \oc$; for which, by Proposition~\ref{prop:4}, it
  suffices to exhibit a unique map of monoidal comonads $P \rightarrow
  \oc$.

  Consider the category $\cat{Coalg}(\oc)$ of Eilenberg--Moore
  $\oc$-coalgebras and its forgetful functor
  $U^\oc \colon \cat{Coalg}(\oc) \rightarrow \C$. We begin by showing
  that monoidal comonad morphisms $P \rightarrow \oc$ are in bijection
  with factorisations
  \begin{equation}
    \label{eq:43}
    \cd{
      {(\cat{Set}, \times)} \ar@{-->}[rr]^-{F} \ar[dr]_-{L} & &
      {\mathsf{Coalg}(\oc)} \ar[dl]^-{U^\oc} \\ &
      {\C}
    }
  \end{equation}
  of $L$ through $U^\oc$ via a (necessarily strong) symmetric monoidal
  functor $F$. Indeed, to give $F$ is equally to endow each $L(X)$
  with $\oc$-coalgebra structure, naturally and monoidally in $X$, which
  is equally to give a monoidal natural transformation
  $\varphi \colon L \Rightarrow \oc \circ L$ satisfying counit and
  coassociativity axioms. Transposing under the adjunction
  $(\thg) \circ \C(I, \thg) \dashv (\thg) \circ L$, this is equally to
  give a monoidal natural transformation
  $\bar\varphi \colon L \circ \C(I, \thg) = P \Rightarrow \oc$
  satisfying counit and comultiplication axioms, i.e., a monoidal
  comonad morphism $P \Rightarrow \oc$.

  To conclude the proof, it thus suffices to show there is a
  \emph{unique} monoidal factorisation as in~\eqref{eq:43}. First,
  note that, by commutativity of~\eqref{eq:43}, the nullary constraint
  cell of $F$ must be of the form
  \begin{equation*}
    n_I^{-1} \colon (I, m_I) \rightarrow (L1, q) 
  \end{equation*}
  where $n_I^{-1} \colon I \rightarrow L1$ is the nullary constraint
  cell for $L$, $m_I \colon I \rightarrow \oc I$ is the nullary
  constraint cell for $\oc$, and $q_1$ is the assigned $\oc$-coalgebra
  structure on $F1$. This uniquely determines $q_1$ as the composite:
  \begin{equation*}
    q_1 \defeq L1 \xrightarrow{n_I} I \xrightarrow{m_I} \oc I \xrightarrow{\oc n_I^{-1}} \oc L1\rlap{ .}
  \end{equation*}
  So $F(1) = (L1, q_1)$ is forced. Now, because $L$ preserves coproducts
  and $U^\oc$ creates them strictly, we are further forced to take
  $F(X) = (LX, q_X)$, where $q_X$ is the unique map $LX \rightarrow
  \oc LX$ making the family of maps
  \begin{equation*}
    \bigl(\,(L1, q_1) \xrightarrow{L\iota_x} (LX, q_X)\,\bigr)_{x \in X}
  \end{equation*}
  into a coproduct cocone in $\cat{Coalg}(\oc)$.

  Now, since $U^\oc$ is faithful, the remaining data for $F$ are completely
  determined by the requirement of commutativity in~\eqref{eq:43}. So
  there is at most one factorisation~\eqref{eq:43}, and all that is
  required for existence is to show that maps in the image of $L$
  and the constraint maps of $L$ lift along $U^\oc$. We leave
  this (easy) check to the reader.
\end{proof}

Let us now give an explicit description of the initial monoidal
coalgebra modality $P$. The functor $P\colon \C \to \C$ sends an
object $X$ to the coproduct of copies of $I$ indexed by the points of
$X$:
\begin{align}
    PX \defeq \bigoplus \limits_{x\colon I\to X} I
\end{align}
For every point $x\colon I \to X$, let $\iota_x\colon I \to PX$ be the
associated injection, and note that a map with domain $PX$ is
completely determined by giving its composite with each of these
injections. Using this fact, we may now say that:
\begin{itemize}[itemsep=0.25\baselineskip]
\item On morphisms, $P$ sends $f \colon X \rightarrow Y$ to the unique
  map $Pf \colon PX \rightarrow PY$ such that, for all points $x
  \colon I \rightarrow X$:
  \begin{equation*}
  Pf \circ \iota_x = \iota_{f \circ x}\rlap{ .}
\end{equation*}
\item The comonoid counit $\mathsf{e}^P_X\colon PX \to I$ the unique
  map such that for all points $x
  \colon I \rightarrow X$:
  \begin{equation*}
  \mathsf{e}^P_X \circ \iota_x = 1_I\rlap{ .}
\end{equation*}
\item The comonoid comultiplication $\Delta^P_X\colon PX \to PX
  \otimes PX$ is the unique map such that for all points $x
  \colon I \rightarrow X$: 
  \begin{equation*}
    \Delta^P_X \circ \iota_x = \iota_x \otimes \iota_x\rlap{ .}
  \end{equation*}
\item The comonad counit $\varepsilon^P_X\colon PX \to X$ is the
  unique map such that for all points $x \colon I \rightarrow X$:
  \begin{equation*}
    \varepsilon^P_X \circ \iota_x = x\rlap{ .}
  \end{equation*}
\item For the comonad comultiplication, note that every coproduct injection
  $\iota_x\colon I \to PX$ is a point of $PX$. Thus, we may define
  $\delta^P_X\colon PX \to PPX$ as the unique map such that for all
  points $x \colon I \rightarrow X$:
  \begin{equation}
    \label{eq:38}
    \delta^P_X \circ \iota_x = \iota_{\iota_x}\rlap{ .}
\end{equation}
\item The nullary monoidal constraint $m^P_I \colon I \rightarrow PI$
  is the injection $\iota_{1_I}$ associated to the point $1_I \colon I
  \rightarrow I$ of $I$;
\item The binary monoidal constraint $m^P_{X,Y} \colon PX \otimes PY
  \rightarrow P(X \otimes Y)$ is the unique map such that for all
  pairs of points $x \colon I \rightarrow X$ and $y \colon I
  \rightarrow Y$:
  \begin{equation*}
  m^P_{X,Y} \circ (\iota_x \otimes \iota_y) = \iota_{x \otimes y}\rlap{ .}
\end{equation*}
\item Given another monoidal coalgebra modality $\oc$, the unique monoidal
coalgebra modality morphism $\rho \colon P \rightarrow \oc$ has 
components $\rho_X\colon PX \to \oc X$ determined as the unique maps such
that, for all points $x \colon I \rightarrow X$:
\begin{equation*}
 \rho_X \circ \iota_x = I \xrightarrow{m^\oc_I} \oc I \xrightarrow{\oc x} \oc X\rlap{ .}
\end{equation*}
\end{itemize}

If $(\C, \otimes, I)$ also has finite products, then $P$ is a storage
modality (indeed, the initial storage modality). Note that points of
$X \times Y$ correspond to pairs of points $x\colon I \to X$ and
$y\colon I \to Y$, by taking their pairing
$\langle x, y \rangle\colon I \to X \times Y$, and that the terminal
object $\top$ has a unique point $t \colon I \rightarrow \top$. In
these terms:
\begin{itemize}
\item 
The binary Seely map
$\chi^P_{X,Y}\colon P(X \times Y) \to PX \otimes PY$ and its inverse
 are the
unique maps such that for all points $x$ of $X$ and $y$ of $Y$, we have:
\begin{equation*}
\chi^P_{X,Y} \circ \iota_{\langle x, y \rangle} = \iota_x \otimes \iota_y \quad \text{and} \quad  (\chi^P_{X,Y})^{-1} \circ (\iota_x \otimes \iota_y) = \iota_{\langle x, y \rangle}\rlap{ .}
\end{equation*}
\item The nullary Seely map
  $\chi^P_\top \colon P\top \rightarrow I$ and its inverse
  are the unique maps~with:
  \begin{equation*}
  \chi^P_{\top} \circ \iota_{t} = 1_I \quad \text{and} \quad  (\chi^P_{X,Y})^{-1} \circ 1_I = \iota_t\rlap{ ;}
\end{equation*}
  said another way, 
  $\chi^P_\top \colon P\top \rightarrow I$ is simply $\iota_t^{-1}$.
\end{itemize}

Lastly, suppose that $(\C, \otimes, I)$ is a $k$-linear additive
symmetric monoidal category with finite biproducts. In this
situation, $P$ is an additive bialgebra modality (and again, the
initial one), whose monoid structure is given as follows:
\begin{itemize}
\item The unit $\mathsf{u}^P_X\colon I \to PX$ is the injection
  $\iota_0$ for the zero point $0 \colon I \rightarrow X$.
\item The
multiplication $\nabla^P_X\colon PX \otimes PX \to PX$ is the unique
map such that for all points $x,y$ of $X$:
\[ \nabla^P_X \circ (\iota_x \otimes \iota_y) = \iota_{x+y} \]
where $x+y \colon I \rightarrow X$ is the sum of the points $x$ and $y$.
\end{itemize}

\subsection{The initial monoidal differential modality}
\label{sec:init-diff-modal}

In the $k$-linear context, the initial monoidal coalgebra modality $P$
need not come equipped with a deriving transformation or codereliction.
However, if our base symmetric monoidal category $(\C, \otimes, I)$
satisfies hypotheses \ref{assumption1}--\ref{assumption3} of the
preceding section, then by Theorem~\ref{thm:2}, we may construct the
initial monoidal differential modality $P^\partial$ over $P$; and
since $P$ is the initial monoidal coalgebra modality, and
$(\thg)^\partial$ is a left adjoint, $P^\partial$ will be the initial
monoidal differential modality on $\C$. Our objective in this section
is to describe $P^\partial$ explicitly.

To start with, note that
\begin{equation*}
  P^\partial X = PX \otimes SX = \Bigl(\bigoplus_{x\colon I\to X}\!\!
  I\,\Bigr) \otimes SX\rlap{ .}
\end{equation*}
Since $(\thg) \otimes SX$ preserves coproducts by
assumption, the right-hand side here is isomorphic to
$\bigoplus_{x\colon I\to X} SX$; and clearly, it does no harm to take
$P^\partial X$ to actually be given in this simpler way.
Thus, we take the functor $P^\partial\colon \C \to \C$ to be given by:
\begin{align}
  P^\partial X \defeq  \bigoplus_{x\colon I\to X} SX\rlap{ .}
\end{align}

We now describe the remaining structure of $P^\partial$. Let us write
$\jmath_x\colon SX \to P^\partial X$ for the coproduct coprojection
associated to the point $x\colon I \to X$. Using our description of
$P$ together with Theorem~\ref{thm:1}, Lemma~\ref{lem:18} and
Lemma~\ref{lem:22}, we quickly obtain:
\begin{itemize}[itemsep=0.25\baselineskip]
\item On morphisms, $P^\partial$ sends $f \colon X \rightarrow Y$ to
  the unique map $P^\partial f \colon PX \rightarrow PY$ such that the
  following diagram commutes for all points
  $x \colon I \rightarrow X$:
  \begin{equation*}
\cd{ S X \ar[r]^-{\jmath_x} \ar[d]_-{S f} & P^\partial X \ar@{-->}[d]^-{P^\partial f} \\
        S Y \ar[r]^-{\jmath_{f \circ x}} & P^\partial Y\rlap{ .} } 
  \end{equation*}
\item The deriving transformation $\mathsf{d}^\partial_X\colon
  P^\partial X  \otimes X \to P^\partial X$ is the unique map which
  makes the following diagram commute for all points $x \colon I
  \rightarrow X$: 
  \begin{equation}
    \label{eq:63}
  \cd{
    SX \otimes X \ar[d]_-{\mathsf{d}^S_X}   \ar[r]^-{\jmath_{x} \otimes 1_X}
    & P^\partial X \otimes X  \ar@{-->}[d]^-{\mathsf{d}^\partial_X} \\
    SX \ar[r]_-{\jmath_x} & P^\partial X\rlap{ .}
  }      
\end{equation}
\item The maps $\zeta_X \colon PX \rightarrow P^\partial X$ exhibiting
  the freeness are the unique maps which make the following diagram
  commute for all points $x \colon I
  \rightarrow X$: 
  \begin{equation*}
  \cd{
    I \ar[d]_-{\mathsf{u}^S_X} \ar[r]^-{\iota_{x}}
    & PX  \ar@{-->}[d]^-{\zeta_X} \\
    SX \ar[r]_-{\jmath_x} & P^\partial X\rlap{ .}
  }      
\end{equation*}
\item The comonoid counit $\mathsf{e}^\partial_X\colon P^\partial X \to I$ the unique
  map such that the following diagram commutes for all points $x
  \colon I \rightarrow X$:
  \begin{equation*}
  \cd{ SX \ar[dr]_-{\mathsf{e}^S_X}    \ar[r]^-{\jmath_{x}}  &  P^\partial X  \ar@{-->}[d]^-{\mathsf{e}^\partial_X} \\
& I \rlap{ .}}   
  \end{equation*}
\item The comonoid comultiplication $\Delta^\partial_X\colon
  P^\partial X \to P^\partial X
  \otimes P^\partial X$ is the unique map such that the following
  diagram commutes for all points $x
  \colon I \rightarrow X$:
  \begin{equation}
    \label{eq:61}
    \cd[@C+0.5em]{ SX  \ar[d]_-{\Delta^S_X}   \ar[r]^-{\jmath_{x}}  &   P^\partial X  \ar@{-->}[d]^-{\Delta^\partial_X} \\
SX \otimes SX \ar[r]^-{\jmath_x \otimes \jmath_x} & P^\partial X \otimes P^\partial X  \rlap{ .}}      
  \end{equation}
\item The comonad counit $\varepsilon^\partial_X\colon P^\partial X \to X$ is the
  unique map such that the following diagram commutes for all points $x \colon I \rightarrow X$:
  \begin{equation}
    \label{eq:62}
    \cd{ SX \ar[dr]_-{x\, \circ\, \mathsf{e}^S_X \,+\, \varepsilon^S_X}    \ar[r]^-{\jmath_{x}}  &  P^\partial X  \ar@{-->}[d]^-{\varepsilon^\partial_X} \\
      & X\rlap{ .} } 
  \end{equation}
\end{itemize}

We skip over describing the comonad comultiplication
$\delta^\partial \colon P^\partial \Rightarrow P^\partial P^\partial$
for now, and continue with the following, found using
Lemma~\ref{lem:23}, Lemma~\ref{lem:24} and~\eqref{eq:58}:

\begin{itemize}
\item The monoid multiplication $\nabla^{\partial}_X\colon  P^\partial X \otimes P^\partial X \to P^\partial X$ is the unique map which makes the following diagram commute for all points: 
\begin{gather*}
\cd{ SX \otimes SX \ar[d]_-{\nabla^S_X}   \ar[r]^-{\jmath_{x} \otimes \jmath_y}  &   P^\partial X  \ar@{-->}[d]^-{\nabla^\partial_X} \\
SX \ar[r]_-{\jmath_{x+y}} & P^\partial X \rlap{ .}}
\end{gather*}
\item The monoid unit $\mathsf{u}^\partial_X\colon I \to P^\partial X$ is the
  composite
  \begin{equation*}
    \smash{\mathsf{u}^\partial_X = I \xrightarrow{\mathsf{u}^S_X} SX \xrightarrow{\jmath_0} P^\partial X\rlap{ .}}
  \end{equation*}
\item The nullary monoidality constraint $m^\partial_I$ is the composite:
  \begin{equation*}
    \smash{m^\partial_I = I \xrightarrow{\mathsf{u}^S_X} SI \xrightarrow{\jmath_{1_I}} P^\partial I}\rlap{ .}
  \end{equation*}
\item The codereliction $\eta^\partial_X\colon X \to P^\partial X$
  is the composite
  \begin{equation*}
    \smash{\eta^\partial_X = X \xrightarrow{\eta^S_X} SX \xrightarrow{\jmath_0} P^\partial_X\rlap{ .}}
  \end{equation*}
\end{itemize}

It remains to describe the comonad comultiplication $\delta^\partial$
and the binary monoidality constraints $m^\partial_{XY}$. For the
latter, there seems to be no real improvement possible over the
description at the end of Section~\ref{sec:free-mono-diff}; however,
we can be more explicit about $\delta^\partial$, with a little work.
To begin with, recall that
$\delta^\partial_X \colon \ocf X \rightarrow P^\partial P^\partial X$
is the unique map of commuting $X$-actions
$(P^\partial X, \mathsf{d}^\partial_X) \rightarrow P^{\partial
  X}(P^\partial X, \mathsf{d}^\partial_X)$ which renders commutative
the square to the left in:
\begin{equation*}
  \cd[@C+1em]{
    P X \ar[d]_-{\zeta_X} \ar[r]^-{ \delta_X} &
    PP X  \ar[d]^-{ (\zeta \zeta)_X } & &
    P X \ar[d]_-{\zeta_X} \ar[r]^-{ \zeta_{PX} \circ \delta_X} &
    P^\partial P X  \ar[d]^-{ P^\partial(\zeta_X) } 
    \\
    P^\partial X \ar[r]^-{\delta^{\partial}_X} & P^\partial P^\partial X & &
    P^\partial X \ar[r]^-{\delta^{\partial}_X} & P^\partial P^\partial X \rlap{ ,}
  }
\end{equation*}
which is equally well the square to the right. We claim that, for
every point $x \colon 1 \rightarrow X$, this right-hand square can be
made into the front face of a commuting cube:
\begin{equation}
  \label{eq:60}
  \cd[@!C@C-1.7em@R-0.2em]{
    I \ar[dd]_-{\mathsf{u}^S_X} \ar[rr]^-{m^\partial_I}
    \ar[dr]^-{\iota_x} & &
    P^\partial I \ar[dr]^-{P^\partial \iota_x}
    \ar'[d]_-{P^\partial \mathsf{u}^S_X}[dd] \\
    & P X \ar[dd]_(0.25){\zeta_X} \ar[rr]_(0.3){ \zeta_{PX} \circ \delta_X} &&
    P^\partial P X  \ar[dd]^-{ P^\partial(\zeta_X) } 
    \\
    SX \ar[dr]_-{\jmath_x} \ar'[r][rr]_-{\delta'_X} & &
    P^\partial SX \ar[dr]^-{P^\partial \jmath_x} \\
    & P^\partial X \ar[rr]_-{\delta^{\partial}_X} & & P^\partial P^\partial X\rlap{ .}
  }
\end{equation}
Here, the left face commutes by definition of $\iota_x$, $\jmath_x$
and $\zeta_X$, and the right face commutes as it is $P^\partial$ of the
left face. The top face commutes due to the decomposition
\begin{equation*}
  \cd{
    I \ar[r]^-{m_I} \ar[d]_-{\iota_x} &
    PI \ar[d]_-{P\iota_x} \ar[r]^-{\zeta_I} &
    P^\partial I \ar[d]_-{P^\partial \iota_x} \\
    PX \ar[r]^-{\delta_X} & PPX \ar[r]^-{\zeta_{PX}} & P^\partial X
  }
\end{equation*}
whose left square commutes by definition of the maps involved, and
whose right square is naturality of $\zeta$.

It remains to explain the bottom and rear faces of~\eqref{eq:60}.
These involve a new map
$\delta'_X \colon SX \rightarrow P^\partial SX$, which we define
using algebraic-freeness as the unique map of $X$-actions
$(SX, \mathsf{d}^S_X) \rightarrow P^{\partial X}(SX, \mathsf{d}^S_X)$
rendering the back square commutative. Now the bottom face can be
enhanced to a square
of $X$-actions
\begin{equation*}
  \cd{
    (SX, \mathsf{d}^S_X) \ar[r]^-{\delta'_X} \ar[d]_-{\jmath_x} & P^{\partial X} (SX, \mathsf{d}^S_X) \ar[d]^-{P^\partial \jmath_x} \\
    (P^\partial X, \mathsf{d}^\partial_X) \ar[r]^-{\delta^\partial_X} &
    P^{\partial X}(P^\partial X, \mathsf{d}^\partial_X)
  }
\end{equation*}
which, by freeness of $(SX, \mathsf{d}^S_X)$, will commute if it
does so on precomposition with $\mathsf{u}^S_X$---which is true by
commutativity of the other faces of~\eqref{eq:60}.

Thus, to describe $\delta^\partial_X$ it suffices to describe 
$\delta'_X$. To start with, we may re-express commutativity of the
back face 
of~\eqref{eq:60} as the commutativity of the square
\begin{equation*}
  \cd{
    I \ar[d]_-{\mathsf{u}^S_X} \ar[r]^-{\mathsf{u}^S_{SX}} &
    SSX \ar[d]^-{\jmath_{\mathsf{u}^S_X}} \\
    SX \ar[r]^-{\delta'_X} & \displaystyle\bigoplus_{x \colon I \rightarrow SX} SSX\rlap{ .}
  }
\end{equation*}
So $\delta'_X$ is determined by how the
$X$-action of $P^{\partial X}(SX, \partial^S_X)$ behaves on the
$\mathsf{u}^S_X$-summand. The following lemma describes this in detail.

\begin{Lemma}
  \label{lem:26}
  For any $X \in \C$, the commuting $X$-action of $P^{\partial
    X}(SX, \mathsf{d}^S_X)$ maps the $\mathsf{u}^S_X$-summand to
  itself via the sum of the two morphisms
  \begin{equation}
    \label{eq:64}
    \begin{aligned}
      \mathsf{D}^\flat &= SSX \otimes X \!\xrightarrow{\!\mathsf{b}^S_{SX} \otimes 1\!}\! SSX \otimes SX \otimes X \!\xrightarrow{\!1 \otimes \mathsf{d}^S_X\!}\! SSX \otimes SX \xrightarrow{\!\mathsf{d}^S_{SX}\!} SSX\\
      \mathsf{D}^\sharp &= SSX
      \otimes X \xrightarrow{1 \otimes \eta^S_X} SSX \otimes SX
      \xrightarrow{\mathsf{d}^S_{SX}} SSX\rlap{ .}
    \end{aligned}\!\!\!\!
  \end{equation}
\end{Lemma}

\begin{proof}
  Using~\eqref{eq:23}, \eqref{eq:63}, \eqref{eq:61} and \eqref{eq:62},
  we see that for any commuting $X$-action $(A, \alpha)$, the
  $X$-action structure of $P^{\partial X}(A, \alpha)$ is the map
  \begin{equation*}
    \bigoplus_{a \colon 1 \rightarrow A} SA \otimes X \rightarrow 
    \bigoplus_{a \colon 1 \rightarrow A} SA
  \end{equation*}
  which maps the $a$-summand to itself via the composite morphism
  \begin{equation*}
    SA \otimes X \xrightarrow{\mathsf{b}^S_{A} \otimes 1 + 1 \otimes a \otimes 1} SA \otimes A \otimes X \xrightarrow{1 \otimes \alpha} SA \otimes A \xrightarrow{\mathsf{d}^S_{A}} SA \rlap{ .}
  \end{equation*}
  The result follows on taking $(A, \alpha) = (SA, \mathsf{d}^S_A)$
  and using $\eta^S_X = \mathsf{d}^S_X \circ (\mathsf{u}^S_X \otimes 1)$.
\end{proof}

Putting all of the above together, we obtain:
\begin{Prop}
  \label{prop:11}
  The comonad comultiplication $\delta^\partial_X \colon P^\partial X
  \rightarrow P^\partial P^\partial X$ of the initial monoidal
  differential modality is the unique map which renders commutative
  the following diagram for each point $x \colon I \rightarrow X$: 

\begin{equation*}
\cd[@C+2em]{ SX  \ar[d]_-{\delta^S_X}   \ar[rr]^-{\jmath_{x}}  && P^\partial X  \ar@{-->}[d]^-{\delta^\partial_X} \\
  SSX \ar[r]^-{S\jmath_x} & SP^\partial X \ar[r]^-{\jmath_{ \jmath_x \circ \mathsf{u}^S_X}} & P^\partial P^\partial X\rlap{ ,} }      
\end{equation*}
where $\delta^S_X \colon SX \rightarrow SSX$ is the unique map of
$X$-actions $(SX, \mathsf{d}^S_X) \rightarrow (SSX,
\mathsf{D}^\flat + \mathsf{D}^\sharp)$ such that $\delta^S_X \circ
\mathsf{u}^S_X = \mathsf{u}^S_{SX}$.
\end{Prop}

Finally, let us describe explicitly how $P^\partial$ becomes the
initial monoidal differential modality. Given another monoidal
differential modality $\oc$ on $\C$, we define the map
$\rho^\flat_X\colon SX \to \oc X$ as the unique $X$-action morphism
$(SX, \mathsf{d}^S_X) \rightarrow (\oc X, \mathsf{d}^\oc_X)$ which
makes the following triangle commute:
     \begin{equation*}
\xymatrix{
    & I \ar[dl]_-{\mathsf{u}_X} \ar[dr]^-{\mathsf{u}^\oc_X} \\
    SX \ar@{-->}[rr]^-{\exists!\, \rho^\flat_X} && \oc X
  }
 \end{equation*}
 Then the unique differential modality morphism
 $\rho^\sharp \colon P^\partial \rightarrow \oc$ has as its
 $X$-component the unique map such that the following diagram commutes
 for all points $x \colon I \rightarrow X$:
\begin{equation*}\begin{gathered} 
\cd{ SX \ar[dr]_-{\rho^\flat_X}   \ar[r]^-{\iota_{x} }  & P^\partial X  \ar@{-->}[d]^-{\rho^\sharp_X} \\
& \oc X \rlap{ .} }      
\end{gathered}\end{equation*}

\section{Examples}
\label{sec:examples}

In this final section, we describe the initial monoidal differential
modality in some well-known contexts: the category of sets and
relations; the category of $k$-modules over a commutative rig; the
category of super vector spaces; and the category of linear species.

\subsection{Sets and relations}
\label{sec:sets-relations}

Let $\REL$ be the category of sets and relations. We verify
  the assumptions~\ref{assumption1}--\ref{assumption3} of
  Section~\ref{sec:init-mono-diff}:
  \begin{itemize}[itemsep=0.25\baselineskip]
  \item 
  For assumption~\ref{assumption1}: $\REL$ is an $\mathbb{N}$-linear symmetric
  monoidal category with all set-indexed biproducts, where the
  monoidal product $\otimes$ is given by cartesian product, the
  monoidal unit $I$ is a chosen singleton set, and the biproduct is
  given by disjoint union $\bigoplus$.

\item For assumption~\ref{assumption2}: as observed in
  Section~\ref{sec:existence}, $\REL$ inherits the symmetric algebra
  construction from $\SET$, and thus has algebraically-free
  commutative monoids, with the algebraically-free commutative monoid
  on $X$ given by the set of finite multisets of $X$. Since this
  construction underlies the cofree cocommutative comonoid exponential
  of $\REL$, we will denote it by $\oc$ rather than $S$. Thus, for a
  set $X$, we write $\oc X$ for the set of finite multisets (also
  called finite bags) of $X$. We denote a finite multiset as
  $B = \llbracket x_1, \hdots, x_n \rrbracket$, the empty multiset as
  $\llbracket ~ \rrbracket$, and the (disjoint) union of multisets as
  $B_1 + B_2$.

  \item For assumption~\ref{assumption3}: we already observed that $\REL$
  has \emph{all} biproducts, and these are preserved by tensor in each
  variable because $\REL$ is compact closed, so in particular,
  monoidal closed.
  \end{itemize}

  Using assumption~\ref{assumption3}, we can construct the initial
  monoidal coalgebra modality on $\REL$. Note that points
  $I \rightarrow X$ of an object $X$ in $\REL$ are subsets
  $U \subseteq \lbrace \ast \rbrace \times X$, and so correspond
  precisely to subsets of $X$. Therefore, we may identify $PX$ with
  the power-set of $X$; its monoidal coalgebra modality structure is,
  by construction, that induced by the linear--non-linear adjunction
  between $\REL$ and $\SET$\footnote{This ``power-set exponential'' is
  known to the community as a degenerate model of linear logic.}.

  Thus, we discover that the power-set exponential $P$ and the
  multiset exponential $\oc$ both have universal characterisations:
  the former is the initial monoidal coalgebra modality, and the
  latter is the terminal one. $\oc$ is also the terminal monoidal
  differential modality; however:
  \begin{Lemma}
    \label{lem:28}
    The monoidal coalgebra modality $P$ is not a differential
    modality.
  \end{Lemma}
  \begin{proof}
    Suppose the maps $\mathsf{d}_X \colon PX \otimes X \rightarrow PX$
    were the components of a deriving transformation for $P$ in
    $\REL$. Since $\mathsf{e}_X \circ \mathsf{d}_X = 0$, and since
    $\mathsf{e}_X \colon PX \rightarrow I = \{*\}$ is the maximal
    relation, this forces each $\mathsf{d}_X$ to be the empty
    relation. But this contradicts the fact that, by the linear rule
    and bialgebra axioms,
    $\varepsilon_X \circ \mathsf{d}_X \circ (\mathsf{u}_X \otimes X) =
    1_X$.
  \end{proof}
  Instead, we may apply our main result to construct the free
  differential modality over $P$---and this, the initial monoidal
  differential modality $P^\partial$ on $\REL$, is something else, and
  something new.

  First of all, we have $P^\partial X = PX \times \oc X$, so that
  elements of $P^\partial X$ are pairs $(U,B)$ consisting of a subset
  $U \subseteq X$ and a finite multiset $B \in \oc X$. The rest of the
  monoidal differential modality structure can now be read off from
  the descriptions in the previous section, as follows:
  \begin{itemize}[itemsep=0.25\baselineskip]
  \item On morphisms, $P^\partial$ acts as $P \otimes \oc$ does: thus,
    if $R \colon X \rightarrow Y$ is a relation, then $P^\partial R$
    relates $(U, \dbr{b_1, \dots, b_k})$ to
    $(V,\dbr{c_1, \dots, c_\ell})$ just when, firstly, we have
    $V = \{ y \in Y : (x,y) \in R \text{ for some } x \in U\}$ and,
    secondly, there is a bijection $\varphi \colon k \rightarrow \ell$
    such that $(b_i, c_{\varphi(i)}) \in R$ for all $i \in k$.
    
  \item The comonad counit is the relation
    $\varepsilon^\partial_X\colon P^\partial X \rightarrow X$ which
    relates any pair $(U, \llbracket x \rrbracket)$ to $x$, and
    relates the pair $(U, \llbracket ~ \rrbracket)$ to each $u \in U$.

  \item The comonad comultiplication is the relation
    $\delta^\partial_X\colon P^\partial X \to P^\partial P^\partial X$
    which relates a pair $(U,B)$ to all pairs of the form
    \begin{equation*}
      \bigl(\,\{(U, \dbr{\ })\}, \ \dbr{(U, B_1), \dots, (U, B_n)}\,\bigr)
    \end{equation*}
    whose first component is the singleton subset of $P^\partial X$
    containing the element $(U, \dbr{\ })$, and whose second component
    is a multiset of elements of $P^\partial X$ of the form
    $(U, B_i)$, where each $B_i$ is a \emph{non-empty} finite multiset
    and the $B_i$'s are a decomposition of $B$, i.e.,
    $B_1 + \hdots + B_n = B$. Note that this is substantively
    different from the comonad comultiplication of $\oc$, which
    relates a finite multiset $B$ to all finite multisets of (possibly
    empty!) finite multisets which are a decomposition of $B$.

    \item The comonoid comultiplication is the relation
      $\Delta^\partial_X\colon P^\partial X \to P^\partial X \otimes
      P^\partial X$ which relates a pair $(U,B)$ to the pairs of pairs $\bigl((U, B_1),
      (U, B_2)\bigr)$ for which $B_1 + B_2 = B$. 

    \item The comonoid counit is the relation
      $\mathsf{e}^\partial_X\colon P^\partial X \to I$ which, for all
      subsets $U \subseteq X$, relates the pair
      $(U, \llbracket~ \rrbracket)$ to $\ast$.

    \item The deriving transformation is the relation
    $\mathsf{d}^\partial_X\colon P^\partial X \otimes X \to P^\partial X$ induced by the
    function $(U,B,x) \mapsto (U,B+\llbracket x \rrbracket)$. 

    \item The monoid multiplication is the relation
      $\Delta^\partial_X\colon P^\partial X \otimes P^\partial X \to
      P^\partial X$ induced by the function $\left( (U_1, B_1), (U_2,
        B_2) \right) \mapsto ( U_1 \cup U_2, B_1 + B_2)$. 
    \item The monoid unit is the relation $\mathsf{u}^\partial_X\colon
      I \to P^\partial X$ induced by the function $\ast \mapsto
      (\emptyset, \llbracket ~ \rrbracket)$.
    \item The codereliction is the relation $\eta^\partial_X\colon X
      \to P^\partial X$ induced by the function $x \mapsto (\emptyset, \llbracket x \rrbracket)$.
    \item The monoidal unit constraint is the relation $m^\partial_I
      \colon I \rightarrow P^\partial X$ induced by the singleton
      mapping $\ast \rightarrow \{\ast\}$.
    \item The binary monoidal constraint $m^\partial_{XY} \colon
      P^\partial X \otimes P^\partial Y \rightarrow P^\partial(X
      \otimes Y)$ is rather subtle. It relates a pair of pairs $\left(
        (U,\dbr{b_1, \dots, b_k}), (V,\dbr{c_1, \dots, c_\ell}) \right)$ to all pairs of the form $(U \times V,
      D)$, where $D$ is a multiset determined by the following
      conditions:
      \begin{itemize}
      \item There is a subset $I \subseteq \{1, \dots, k\}$ and $J
        \subseteq \{1, \dots, \ell\}$ and a bijection $\varphi \colon
        I \rightarrow J$; 
      \item There are functions $f \colon \{1, \dots, k\} \setminus I
        \rightarrow U$ and $g \colon \{1, \dots, \ell\} \setminus J
        \rightarrow V$;
      \item $D$ is the disjoint union of the singletons $\dbr{(b_i,
          c_{\varphi(i)})}$ for $i \in I$, the singletons
        $\dbr{(b_{i'}, f(i')}$ for $i' \in \{1, \dots, k\} \setminus
        I$, and the singletons
        $\dbr{(g(j'), c_{j'}}$ for $j' \in \{1, \dots, \ell\} \setminus
        J$.
      \end{itemize}
\end{itemize}
It is interesting to note that, by contrast with the cofree
exponential $\oc$ on $\REL$, the monoid multiplication, monoid unit,
and codereliction of $P^\partial$ are not dual to the comonoid
comultiplication, comonoid counit, and comonad counit. In particular,
this means that, unlike $\oc$, the differential modality $P^\partial$
fails to be a \emph{reverse differential modality} in the sense
of~\cite{Cruttwell2022Monoidal}.

\subsection{$k$-modules}
\label{sec:k-modules}

If $k$ is a commutative rig, then $\kmod$ is a cocomplete,
$k$-linear symmetric monoidal closed category, and so satisfies the
assumptions of Section~\ref{sec:init-mono-diff}. So we can construct
the initial monoidal coalgebra modality and the initial monoidal
differential modality on $\kmod$.

For the initial monoidal coalgebra modality, we note that the
monoidal unit in $\kmod$ is $k$ itself, so that points of a
$k$-module $V$ correspond precisely to the elements of $V$. As such,
the initial monoidal coalgebra modality $P$ on $\kmod$ sends a
$k$-module $V$ to the free $k$-module over the underlying set of
$V$.

Turning now to the initial monoidal differential modality
$P^\partial$, we observe that the algebraically-free commutative
monoid on a $k$-module $V$ is given by the well-known symmetric
algebra $\Sym(V)$ over $V$. Therefore, $P^\partial V$ can be described as
the free $\Sym(V)$-module over the underlying set of $V$, or more
explicitly as:
\begin{equation}
  \label{eq:65}
  P^\partial V \defeq \bigoplus \limits_{v \in V} \Sym(V)\rlap{ .}
\end{equation}
In this case, the monoidal differential modality $P^\partial$ was
described in detail in~\cite[Definition~4.8]{Garner2021Cartesian};
though in \emph{loc.~cit.}, it was stated without proof that
$P^\partial$ was initial---a statement which our Theorem~\ref{thm:2}
now justifies.

For any commutative rig $k$, the $P^\partial$ on $\kmod$ is an extremely
natural construction---indeed, as shown
in~\cite[Proposition~4.9]{Garner2021Cartesian}, the co-Kleisli
category of $P^\partial$ is the cofree \emph{cartesian differential
  category} on the category of $k$-modules and arbitrary functions---and it is natural to wonder how it relates to
the terminal monoidal differential modality, i.e., the cofree
cocommutative coalgebra comonad. In general, this seems like a
difficult problem; however, in the special case where $k$ is a field
of characteristic zero, we can establish a direct relationship.

In the special case where $k$ is not only a field of characteristic
zero but also algebraically closed, the counit map
$\varepsilon_M \colon P^\partial V \rightarrow V$ in fact exhibits
$P^\partial V$ as the cofree cocommutative coalgebra on
$V$---see~\cite{Clift2020Cofree, Sweedler1969Hopf}. Thus, in this
case, the terminal (monoidal) differential modality and the initial
monoidal differential modality coincide. When $k$ is a
\emph{non-}algebraically closed field of characteristic zero (for
example, $\mathbb{R}$), $P^\partial$ is not quite the cofree
cocommutative comonoid comonad on $\kmod$---but does nonetheless admit
a direct characterisation. To explain this, let us first clarify some
of the structure of $P^\partial V$ in this case.

For simplicity, let us assume given a basis $\{e_i : i \in I\}$ for
the $k$-vector space $V$; then $\mathrm{Sym}(V)$ can be identified
with the $k$-algebra $k[x_i : i \in I]$ of polynomials in the
indeterminates $(x_i)_{i \in I}$, and $P^\partial V$ has a
basis given given by the vectors
\begin{equation*}
  \iota_v({x_{i_1} \cdots x_{i_k}}) \qquad \text{for $v \in V$ and $\dbr{i_1, \dots, i_k} \in \oc I$,}
\end{equation*}
where $\oc I$ denotes the multisets in $I$, and $\iota_v$ is the
$v$-th direct summand injection.

In these terms, the comonoid structure on
$P^\partial V = \bigoplus_{v \in V} \Sym(V)$ is given on basis
elements by
\begin{gather*}
  \Delta (\iota_v(x_{i_1} \cdots x_{i_k}))  = \bigoplus_{I \subseteq \{1, \dots, k\}} \iota_{v}\bigl(\mathop\Pi_{i \in I} x_{i}\bigr) \otimes
  \iota_v\bigl(\mathop\Pi_{j \notin I} x_{j}\bigr) \text{for $k \geqslant 0$,}\\
  \text{and} \quad 
  \mathsf{e} (\iota_v(x_{i_1} \cdots x_{i_k})) =
  \begin{cases}
    1 & \text{if $k = 0$;}\\
    0 & \text{if $k > 0$.}
  \end{cases}
\end{gather*}
Furthermore, the comonad counit $\varepsilon_V \colon P^\partial V
\rightarrow V$ is defined by
\begin{equation*}
  \varepsilon_V(\iota_v(x_{i_1} \cdots x_{i_k})) =
  \begin{cases}
    v & \text{ if $k = 0$;}\\
    e_{i_1} & \text{ if $k = 1$;}\\
    0 & \text{ otherwise.}
  \end{cases}
\end{equation*}

Let us now show:

\begin{Lemma}
  \label{lem:25}
  Over the field $\mathbb{R}$, the initial monoidal differential
  modality $P^\partial$ is not the cofree cocommutative comonoid
  comonad.
\end{Lemma}
\begin{proof}
Let $C$ be the dual
coalgebra to the $\mathbb{R}$-vector space $\mathbb{C}$; thus, $C$
is spanned by vectors $1$ and $i$ satisfying
\begin{equation*}
  \mathsf{e}(1) = 1 \qquad  \mathsf{e}(i) = 0 \qquad \Delta(1) = 1 \otimes 1 - i \otimes i \qquad \Delta(i) = i \otimes 1 + 1 \otimes i\rlap{ .}
\end{equation*}
We claim that $C$ is not a $P^\partial$-coalgebra. Indeed, it is
easy to see that $C$ is simple, i.e., has no proper non-zero
($\mathbb{R}$-)subcoalgebras. Thus, if $C$ were a
$P^\partial$-coalgebra, then its coalgebra structure map
$C \rightarrow P^\partial C$---which is a (split)
monomorphism---would exhibit a simple subcoalgebra of $P^\partial C$
isomorphic to $C$. By~\cite[Proposition~8.0.3(a)]{Sweedler1969Hopf},
this subcoalgebra would have to lie within one of the
$\Sym(C)$-summands of $P^\partial C$. But $\Sym(C)$ has the unique
simple subcoalgebra $k1 \subseteq \mathrm{Sym}(C)$, which is
manifestly not isomorphic to the two-dimensional $C$: a
contradiction.
\end{proof}

So over non-algebraically closed fields of characteristic zero,
$P^\partial V$ need not be a cofree cocommutative coalgebra; however, it
turns out to be cofree among what~\cite{Sweedler1969Hopf} terms
\emph{pointed} cocommutative coalgebras. First of all, a cocommutative
coalgebra $C$ is \emph{pointed irreducible} if it contains exactly one
simple subcoalgebra, and this subcoalgebra is one-dimensional. Note
this subcoalgebra is then spanned by a unique group-like element $g$,
and we have a (vector space) direct sum decomposition
\begin{equation}
  \label{eq:66}
  C = kg \oplus \mathrm{ker}(\mathsf{e})\rlap{ .}
\end{equation}
A cocommutative coalgebra is now called \emph{pointed} if it is a
direct sum of pointed irreducible coalgebras, and we have:

\begin{Prop}
  \label{prop:12}
  If $k$ is a field of characteristic zero, and $V$ is a $k$-vector
  space then $\varepsilon_V \colon P^\partial V \rightarrow V$
  exhibits $P^\partial V$ as the cofree cocommutative pointed
  coalgebra over $V$.
\end{Prop}
Over an algebraically closed field, \emph{every} cocommutative
coalgebra is pointed by~\cite[Lemma~8.0.1(b)]{Sweedler1969Hopf}, and so
this recovers the characterisation of the cofree cocommutative
coalgebra comonad over a field of characteristic zero.
\begin{proof}
  Let $C$ be a cocommutative pointed
  $k$-coalgebra; we must show that, for any linear map $f \colon C
  \rightarrow V$, there is a unique coalgebra homomorphism $\bar f
  \colon C \rightarrow P^\partial V$ such that $\varepsilon_V \circ
  \bar f = f$.
  We may write $C$ as a coproduct of irreducible cocommutative
  pointed $k$-coalgebras, and using the universal property of
  coproduct, may verify the above condition on each summand
  separately. Thus we may reduce to the case where $C$ is
  pointed irreducible.
  
  In this case, the image of the unique group-like element $g \in C$
  under $\bar f$ must be one of the group-like elements $\iota_v(1)$
  of $P^\partial(V)$. Since $\varepsilon_V$ sends each such element
  to $v \in V$, and since $\varepsilon_V \circ \bar f$ is to equal
  $f$, we must have that
  $\bar f(g) = \iota_{f(g)}(1) \in P^\partial V$. But since $C$ is
  irreducible, its image under
  $\bar f \colon C \rightarrow P^\partial V$ must by
  \cite[Corollary~8.0.9]{Sweedler1969Hopf} be an irreducible
  subcoalgebra of $P^\partial V$, and so must \emph{all} be
  contained within the $f(g)$-summand of $P^\partial V$.

  Thus, we must find a unique coalgebra morphism $\tilde f \colon C
  \rightarrow \mathrm{Sym}(V)$ rendering commutative the diagram to
  the left in:
  \begin{equation*}
    \cd{
      C \ar[drr]_-{f} \ar[r]^-{\tilde f} & \mathrm{Sym}(V) \ar[r]^-{\iota_{f(g)}} & P^\partial V \ar[d]^-{\varepsilon_V} \\
      & & V
    } \qquad
    \quad
    \cd{
      \ker(\mathsf{e}) \ar[drr]_-{\res{f}{\ker(\mathsf{e})}} \ar[r]^-{\res{\tilde f}{\ker(\mathsf{e})}} & \mathrm{Sym}(V) \ar[r]^-{\iota_{f(g)}} & P^\partial V \ar[d]^-{\varepsilon_V} \\
      & & V
    }
  \end{equation*}
  Now, of course, this $\tilde f$ must map $g \in C$ to
  $1 \in \mathrm{Sym}(V)$; thus, via the
  decomposition~\eqref{eq:66}, it suffices to show that $\tilde f$
  is unique such that the diagram to the right above commutes. Note
  that $\res {\smash{\tilde f}} {\ker(\mathsf{e})}$ lands within the
  subspace of constant-free polynomials $\mathrm{Sym}_{>0}(V)$;
  and on this subspace, the composite $\varepsilon_V \circ
  \iota_{f(g)}$ acts identically to the linear map $\varepsilon'
  \colon \mathrm{Sym}(V) \rightarrow V$ given by
  \begin{equation*}
    \varepsilon'(x_{i_1} \cdots x_{i_k}) =
    \begin{cases}
      e_{i_1} & \text{if $k = 1$;} \\
      0 & \text{otherwise.}
    \end{cases}
  \end{equation*}
  Thus, we have reduced to the case of finding a coalgebra
  homomorphism $\tilde f \colon C \to \mathrm{Sym}(V)$ which is
  unique such that the following triangle commutes:
  \begin{equation*}
    \cd[@!C@C-2em]{
      \ker(\mathsf{e}) \ar[dr]_-{\res{f}{\ker(\mathsf{e})}} \ar[rr]^-{\res{\tilde f}{\ker(\mathsf{e})}} & & \mathrm{Sym}(V) \ar[dl]^-{\varepsilon'} \\ 
      & V\rlap{ .}
    }
  \end{equation*}

  Now,~\cite[Theorem~12.2.5]{Sweedler1969Hopf} asserts the existence
  for any $k$-vector space $V$ of a ``cofree'' irreducible pointed
  cocommutative coalgebra $B(V)$: the cofreeness being exhibited by
  a linear map $\pi \colon B(V) \rightarrow V$ with the property
  that, for any cocommutative pointed coalgebra $C$ and any linear
  map $g \colon \mathrm{ker}(\mathsf{e}) \rightarrow V$, there is a
  unique coalgebra map $\tilde f \colon C \rightarrow B(V)$ such
  that
  $\pi \circ \res{\smash{\tilde f}}{\mathrm{ker}(\mathsf{e})} = g$.
  Clearly, if we can exhibit an isomorphism of coalgebras
  $B(V) \cong \mathrm{Sym}(V)$ over $V$, then applying the
  ``cofreeness'' of $B(V)$ to $g \defeq \res f {\ker(\mathsf{e})}$
  will conclude the proof.

  But indeed, $B(V)$ is described as a coalgebra
  by~\cite[Theorem~12.3.2]{Sweedler1969Hopf} and the 
  exercise that precedes it; it has a basis of vectors $\{u_{B}\}_{B \in \oc I}$ with
  coalgebra structure
  \begin{gather*}
    \Delta(u_B) = \bigoplus_{B = B_1 + B_2} u_{B_1} \otimes u_{B_2} \qquad \text{and} \qquad \mathsf{e}(u_B) =     \begin{cases}
      1 & \text{if $B  = \emptyset$;}\\
      0 & \text{if $B \neq \emptyset$,}
    \end{cases}
  \end{gather*}
  and with $\pi(u_B) = e_i$ if $B = \dbr{i}$ and $\pi(u_B) =
  0$ otherwise. It thus suffices to exhibit such a basis for
  $\mathrm{Sym}(V)$. Given
  $B = \dbr{i_1, \dots, i_n} \in \oc I$, let us define
  $q_B$ to be the size of the permutation group
  $\{\sigma \in \mathfrak S_n : (i_1, \dots, i_n) = (i_{\sigma(1)},
  \dots, i_{\sigma(n)})\}$; so, for example,
  $q_{\dbr{i, i, j, j, j, k, k, k, k}} = 2! \times 3! \times 4!$. Then
  taking
  \begin{equation}
    \label{eq:67}
    u_{\dbr{i_1, \dots, i_n}} \defeq \frac {x_{i_1} \cdots x_{i_n}} {q_{\dbr{i_1, \dots, i_n}}} 
  \end{equation}
  yields the required basis.
\end{proof}

Note that this proof does not work over fields of non-zero
characteristic---even algebraically closed ones---because the
vectors in~\eqref{eq:67} are not well-defined. Indeed:

\begin{Lemma}
  \label{lem:27}
  Over the algebraically closed field $\mathbb{Z}_2$, the initial
  monoidal differential modality $P^\partial$ is not the cofree
  cocommutative comonoid comonad.
\end{Lemma}
\begin{proof}
  Let $C$ be the dual of the $\mathbb{Z}_2$-algebra
  $\mathbb{Z}_2[x]/x^2$; so $C$ is generated by vectors $1$ and $d$
  satisfying
  \begin{equation*}
    \mathsf{e}(1) = 1 \qquad  \mathsf{e}(d) = 0 \qquad \Delta(1) = 1 \otimes 1 \qquad \Delta(d) = d \otimes 1 + 1 \otimes d \rlap{ .}
  \end{equation*}
  We claim that there are multiple coalgebra maps
  $C \rightarrow P^\partial(\mathbb{Z}_2)$ which lift the
  zero map $C \rightarrow \mathbb{Z}_2$ along
  $\varepsilon_{\mathbb{Z}_2} \colon P^\partial(\mathbb{Z}_2)
  \rightarrow \mathbb{Z}_2$. Identifying $P^\partial(\mathbb{Z}_2)$
  with $\oplus_{x \in \mathbb{Z}_2} \mathbb{Z}_2[x]$, the natural
  choice for such a coalgebra map $f$ is given by
  \begin{equation*}
    f(1) = \iota_0(1) \quad \text{and} \quad f(d) = \iota_0(x)\rlap{ .}
  \end{equation*}
  However, there is another choice $f'$ given by
  \begin{equation*}
    f'(1) = \iota_0(1) \quad \text{and} \quad f'(d) = \iota_0(x) + \iota_0(x^2)\rlap{ .}
  \end{equation*}
  Of course, since $\varepsilon(\iota_0(x^2)) = 0$, this $f'$ still
  lifts $f$ through $\varepsilon$; the point is to show that it is
  still a coalgebra map. But in $P^\partial(\mathbb{Z}_2)$, we have
  \begin{equation*}
    \Delta(\iota_0(x^2)) = \iota_0(x^2 \otimes 1 + x \otimes x + x \otimes x + 1 \otimes x^2) =\iota_0(x^2) \otimes \iota_0(1) + \iota_0(1) \otimes \iota_0(x^2)
  \end{equation*}
  from which it follows that: 
  \begin{equation*}\Delta(\iota_0(x)+\iota_0(x^2)) =
  (\iota_0(x)+\iota_0(x^2)) \otimes \iota_0(1) + \iota_0(1) \otimes
  (\iota_0(x)+\iota_0(x^2)) \end{equation*}
  as required for $f'$ to be a coalgebra homomorphism.
\end{proof}

Indeed, over a field of non-zero characteristic, the cofree
cocommutative coalgebra on $V$ will involve a direct sum of copies not
of the coalgebra $\mathrm{Sym}(V)$, but rather the so-called shuffle
algebra on $V$, which is a divided-power version of $\mathrm{Sym}(V)$;
see~\cite{Sweedler1969Hopf}. 

Of course, when $k$ is not a field, nor even a ring, we expect that
$P^\partial$ will diverge even further from being the cofree
cocommutative coalgebra modality. 

\subsection{Super vector spaces}
\label{sec:super-vector-spaces}

For our next example, let $k$ be a field and $\SVEC$ the symmetric
monoidal category of super $k$-vector spaces. Recall that, as a
category $\SVEC = \kvec^2$, i.e., a super $k$-vector space $V$ is a
pair of $k$-vector spaces $(V_0, V_1)$, and that the monoidal
structure on
$\SVEC$ has unit $(k, 0)$ and tensor product 
\begin{equation*}
  (V \otimes W)_0 = (V_0 \otimes W_0) \oplus (V_1 \otimes W_1) \qquad 
  (V \otimes W)_1 = (V_0 \otimes W_1) \oplus (V_1 \otimes W_0)\rlap{ .}
\end{equation*}
Moreover, the symmetry map $V \otimes W \rightarrow W \otimes V$ is
given on $0$- and $1$-components by
\begin{align*}
  (V_0 \otimes W_0) \oplus (V_1 \otimes W_1) &\smash{{} \xrightarrow{\sigma \oplus (-\sigma)} W_0 \otimes V_0 \oplus W_1 \otimes V_1} \\ \text{and }
  (V_0 \otimes W_1) \oplus (V_1 \otimes W_0) &\smash{{} \xrightarrow{\spn{\iota_2 \sigma, \iota_1\sigma}} (W_0 \otimes V_1) \oplus (V_0 \otimes W_1)\rlap{ .}}
\end{align*}
With this structure, $\SVEC$ is a cocomplete symmetric monoidal closed
$k$-linear category, and so satisfies the assumptions of
Section~\ref{sec:init-mono-diff}.

This time, since the monoidal unit is $(k,0)$, points
$(k,0) \to (V_0, V_1)$ correspond to elements of $V_0$. So the initial
monoidal coalgebra modality on $\SVEC$ sends a super vector space
$(V_0, V_1)$ to $(P^{\kvec}(V_0), 0)$.
On the other hand, the symmetric algebra construction in $\SVEC$
involves the exterior algebra: we have 
\begin{equation*}
  S(V_0, V_1) = (\mathrm{Sym}(V_0) \otimes E_e V_1, \mathrm{Sym}(V_0) \otimes E_o V_1)
\end{equation*}
where
$E_e$ denotes all even-degree components of the exterior algebra, and
$E_o$ denotes all odd-degree components. Thus, in $\SVEC$, the initial
monoidal differential modality is
given by:
  \begin{equation*}
  P^\partial(V_0, V_1) = \bigoplus \limits_{x \in V_0}\Bigl(  \mathrm{Sym}(V_0) \otimes E_e V_1,  \mathrm{Sym}(V_0) \otimes E_o V_1  \Bigr)
\end{equation*}

\subsection{Linear species}
\label{sec:linear-species}

For our final example, we consider \emph{$k$-linear species} over a commutative
rig $k$, for which our reference is the
compendious~\cite{Aguiar2010Monoidal}. As per Chapter 8 of
\emph{op.~cit.}, a $k$-linear species $\vec p$ is a functor
$\cat{Bij} \rightarrow \kmod$, where $\cat{Bij}$ is the category of
finite sets and bijections, and we follow~\cite{Aguiar2010Monoidal} in 
writing the value of such a $\vec p$ at finite set $I$ as $\vec p[I]$.

The category of $k$-linear species $\SP$ is the presheaf category
$[\cat{Bij}, \kmod]$; it is cocomplete since $\kmod$ is so, and is
symmetric monoidal closed under the \emph{Cauchy tensor product}
\begin{equation*}
  (\vec p \otimes \vec q)[I] = \bigoplus_{J \subseteq I} \vec p[J] \otimes \vec q[I\setminus J]\rlap{ ,}
\end{equation*}
with as unit, the representable species at the empty set
$\varnothing \in \cat{Bij}$---i.e., the species which is $k$
concentrated in degree $\varnothing$. In general, given $A \in \kmod$,
we write $\Delta A$ for the species which is $A$ concentrated in
degree $0$, i.e.,
\begin{equation*}
  (\Delta A)[\varnothing] = A \qquad \text{and} \qquad (\Delta A)[I] = 0 \text{ for $I \neq \varnothing$,}
\end{equation*}
and this gives us a strong monoidal functor $\Delta \colon \kmod
\rightarrow \SP$, which is both left and right adjoint to
the functor $\SP \rightarrow \kmod$ which evaluates at $\varnothing$.

Since $\SP$ satisfies the assumptions of
Section~\ref{sec:init-mono-diff} it admits an initial monoidal
differential modality. Towards describing this, we note that the
(algebraically-)free commutative monoid on a species $\vec p$ is given
by
\begin{equation}
  \label{eq:22}
  S(\vec p)[I] = \bigoplus_{n \in \mathbb{N}} \Bigl(\bigoplus_{f \colon I \rightarrow \{1, \dots, n\}} \vec p[f^{-1}(1)] \otimes \dots \otimes \vec p[f^{-1}(n)]\Bigr) / \mathfrak{S}_n\rlap{ .}
\end{equation}
Now since the monoidal unit in $\SP$ is represented by the empty set,
points $I \rightarrow \vec p$ of a species correspond to elements of
$\vec p[\varnothing]$. Thus, the initial monoidal differential
modality $P^\partial$ on $\SP$ has action
\begin{equation*}
  \vec p \qquad \mapsto \qquad P^\partial(\vec p) = \bigoplus_{v \in \vec p[\varnothing]} S(\vec p)
\end{equation*}
where $S(\vec p)$ is as in~\eqref{eq:22}. The remaining structure can
be described in a similar way to before, and we leave this as an
exercise to the reader.

In the case where $k$ is a field, it is reasonable to ask whether
$P^\partial$ coincides with the cofree cocommutative coalgebra
differential modality. Section~11.5 of~\cite{Aguiar2010Monoidal}
analyses carefully the cofree cocommutative comonoid in $\SP$ over a
\emph{positive} species $\vec p$, i.e., one for which
$\vec p[\emptyset] = 0$, showing that it is given by $S(\vec p)$. For
a general species $\vec p$, we have a direct sum decomposition
\begin{equation*}
  \vec p \cong \vec p^+ \oplus \Delta(\vec p[\emptyset])
\end{equation*}
where $\vec p^+$ is positive. Clearly, then, the cofree cocommutative
comonoid on the first summand is $S(\vec p^+)$; on the other hand,
since $\Delta$ is a strong monoidal right adjoint, it preserves cofree
cocommutative comonoids, so that the cofree cocommutative
comonoid on the second summand is $\Delta C$, for $C$ the cofree
cocommutative $k$-coalgebra on $\vec p[\varnothing]$. It follows using
the Seely isomorphisms that the cofree cocommutative comonoid on $\vec
p$ is $S(\vec p^+) \otimes \Delta C$.

Now, when $k$ is an algebraically-closed field of characteristic zero,
the cofree cocommutative $k$-coalgebra on $\vec p[\varnothing]$ is, as
already discussed, $\bigoplus_{v \in \vec p[\varnothing]}
\mathrm{Sym}(\vec p[\varnothing])$. Thus, in this case, the cofree
cocommutative comonoid on $\vec p$ is given by
\begin{align*}
  S(\vec p^+) \otimes \Delta\Bigl(\bigoplus_{v \in \vec p[\varnothing]} \mathrm{Sym}(\vec p[\varnothing])\Bigr) &\cong
  \bigoplus_{v \in \vec p[\varnothing]} S(\vec p^+) \otimes \Delta( \mathrm{Sym}(\vec p[\varnothing])) \\
  &\cong \bigoplus_{v \in \vec p[\varnothing]} S(\vec p^+) \otimes S(\Delta(\vec p[\varnothing])) \\
  &\cong \bigoplus_{v \in \vec p[\varnothing]} S(\vec p^+ \oplus \Delta(\vec p[\varnothing])) \cong \bigoplus_{v \in \vec p[\varnothing]} S(\vec p)\rlap{ ,}
\end{align*}
i.e., by $P^\partial \vec p$. Thus, in this context, the initial
monoidal differential modality and the terminal monoidal differential
modality again coincide.

In a similar way, when we move away from algebraically closed fields
of characteristic zero, these two monoidal comonads will no longer
coincide. It seems to be an interesting question as to whether there
are general conditions under which we may expect such a coincidence;
but this must await further work. 

\bibliographystyle{acm}
\bibliography{bibdata}

\end{document}